\documentclass[fleqn]{article}

\usepackage[a4paper, left=2cm, right=2cm, top=3cm, bottom=3cm]{geometry}                           
\usepackage[applemac]{inputenc}                                                                    
\usepackage[T1]{fontenc}                                                                           
\usepackage[colorlinks, linkcolor=black, citecolor=black, urlcolor=black]{hyperref}                
\usepackage{ifthen}                                                                                
\usepackage{amsmath}                                                                               
\usepackage{amssymb}                                                                               
\usepackage{amsthm}                                                                                
\usepackage{upgreek}                                                                               
\usepackage{mathdots}                                                                              
\usepackage{tikz}                                                                                  

\usetikzlibrary{calc,decorations.markings,matrix,positioning}                                               

\allowdisplaybreaks                                                                                

\swapnumbers
\theoremstyle{definition}
\newtheorem{definition}{Definition}[section]
\newtheorem{corollary}[definition]{Corollary}

\newtheorem{notation}[definition]{Notation}
\newtheorem{proposition}[definition]{Proposition}
\newtheorem{question}[definition]{Question}
\newtheorem{remark}[definition]{Remark}
\newtheorem{theorem}[definition]{Theorem}
\newtheorem*{theorem*}{Theorem}

\makeatletter
\renewcommand*\p@enumii{}                                                                          
\makeatother

\newcommand{\booktitle}[1]{\textsl{#1}}                                                            
\newcommand{\eigenname}[1]{\textsc{#1}}                                                            
\newcommand{\newnotion}[1]{\textit{#1}}                                                            

\newcommand{\nbd}{\nobreakdash-\hspace{0pt}}                                                       

\DeclareMathOperator{\Denominators}{\mathrm{Den}}                                                  
\DeclareMathOperator{\Mor}{\mathrm{Mor}}                                                           
\DeclareMathOperator{\Ob}{\mathrm{Ob}}                                                             
\newcommand{\categoricallyequivalent}{\simeq}                                                      
\newcommand{\CategoryOfObjectsWithSReplacement}[2]{{#1}_{\mathrm{Rpl}_{\text{S}}(#2)}}             
\newcommand{\comp}{\circ}                                                                          
\newcommand{\ForgetfulFunctor}{\mathrm{U}}                                                         
\newcommand{\GabrielZismanLocalisation}{\mathsf{GZ}}                                               
\newcommand{\id}{\mathrm{id}}                                                                      
\newcommand{\InducedFunctorOfSReplacementFunctor}[1][]{\hat{\mathrm{Q}}_{\text{S}\ifthenelse{\equal{#1}{}}{}{, #1}}} 
\newcommand{\InducedFunctorOfTotalSReplacementFunctor}{\hat{\bar{\mathrm{Q}}}_{\text{S}}}          
\newcommand{\Integers}{\mathbb{Z}}                                                                 
\newcommand{\isomorphic}{\cong}                                                                    
\newcommand{\LocalisationFunctor}[1][]{\mathrm{loc}^{#1}}                                          
\newcommand{\map}{\rightarrow}                                                                     
\newcommand{\Naturals}{\mathbb{N}}                                                                 
\newcommand{\SReplacementFunctor}[1][]{\mathrm{Q}_{\text{S}\ifthenelse{\equal{#1}{}}{}{, #1}}}     
\newcommand{\set}[2][]{\{\ifthenelse{\equal{#1}{}}{#2}{#1 \mid #2}\}}                              
\newcommand{\StructureCategory}[2]{{#1}_{#2}}                                                      
\newcommand{\StructureChoiceFunctor}[1]{\mathrm{I}_{#1}}                                           
\newcommand{\Stwoarrow}[3]{{#1} \rightarrow {#2} \leftarrow {#3}}                                  
\newcommand{\TotalSReplacementFunctor}{\bar{\mathrm{Q}}_{\text{S}}}                                

\tikzset{equality/.style={-, double}}
\tikzset{exists/.style={dotted}}
\tikzset{implies/.style={->, double distance=1pt}}
\tikzset{logicallyequivalent/.style={<->, double distance=1pt}}

\tikzset{diagram/.style={matrix of math nodes, row sep=#1, column sep=#1, text height=1.6ex, text depth=0.45ex, inner sep=0pt, nodes={inner sep=0.3333em}}, diagram/.default=2.5em}

\tikzset{den/.style={decoration={markings, mark=at position #1 with {\node[fill=white, inner sep=0.5pt, transform shape]{\(\scriptstyle{\approx}\)};}}, postaction=decorate}, den/.default=0.5}

\tikzset{box/.style={outer sep=0.1cm, text width=#1, text centered}, box/.default=3cm}
\tikzset{input/.style={box=#1, rounded corners=2mm, very thick, draw=yellow!50!black!50, top color=white, bottom color=yellow!50!black!20}, input/.default=3cm}
\tikzset{machine/.style={box=#1, very thick, draw=blue!50!black!50, top color=white, bottom color=blue!50!black!20}, machine/.default=3cm}
\tikzset{old machine/.style={box=#1, very thick, draw=blue!50!black!50, top color=white, bottom color=blue!50!black!10}, old machine/.default=3cm}
\tikzset{output/.style={box=#1, rounded corners=2mm, very thick, draw=green!50!black!50, top color=white, bottom color=green!50!black!20}, output/.default=3cm}
\tikzset{phantom/.style={box=#1, draw=none, color=white, white}, phantom/.default=3cm}

\title{Equivalences between localisations of categories \\ provided by replacements}
\author{Sebastian Thomas}
\date{October~10, 2018}

\begin{document}

\renewcommand{\thefootnote}{\fnsymbol{footnote}}
\renewcommand{\thefootnote}{\arabic{footnote}}

\maketitle

\begin{abstract}
We give a characterisation of functors whose induced functor on the level of localisations is an equivalence and where the isomorphism inverse is induced by some kind of replacements such as projective resolutions or cofibrant replacements.
\end{abstract}

\section{Introduction} \label{sec:introduction}

To a category \(\mathcal{C}\) together with a set of distinguished morphisms, called denominators in \(\mathcal{C}\), one might attach its (Gabriel/Zisman) localisation \(\GabrielZismanLocalisation(\mathcal{C})\), that is, the universal category where the denominators in \(\mathcal{C}\) become~invertible. Given categories \(\mathcal{C}\) and \(\mathcal{D}\) together with sets of denominators and a functor \(F\colon \mathcal{C} \map \mathcal{D}\) that maps denominators in \(\mathcal{C}\) to denominators in \(\mathcal{D}\), the universal property of \(\GabrielZismanLocalisation(\mathcal{C})\) yields a functor \(\GabrielZismanLocalisation(F)\colon \GabrielZismanLocalisation(\mathcal{C}) \map \GabrielZismanLocalisation(\mathcal{D})\) such that the following quadrangle commutes, where \(\LocalisationFunctor\) denotes the canonical functor from a category with denominators to its localisation.
\[\begin{tikzpicture}[baseline=(m-2-1.base)]
  \matrix (m) [diagram]{
    \mathcal{C} & \mathcal{D} \\
    \GabrielZismanLocalisation(\mathcal{C}) & \GabrielZismanLocalisation(\mathcal{D}) \\ };
  \path[->, font=\scriptsize]
    (m-1-1) edge node[above] {\(F\)} (m-1-2)
            edge node[left] {\(\LocalisationFunctor\)} (m-2-1)
    (m-1-2) edge node[right] {\(\LocalisationFunctor\)} (m-2-2)
    (m-2-1) edge[exists] node[above] {\(\GabrielZismanLocalisation(F)\)} (m-2-2);
\end{tikzpicture}\]

As every functor is an equivalence if and only if it is dense, full and faithful, this in particular holds for the induced functor \(\GabrielZismanLocalisation(F)\). In this article, we focus on functors \(F\) with a particular property that ensures density of \(\GabrielZismanLocalisation(F)\), and establish a characterisation for \(\GabrielZismanLocalisation(F)\) to be an equivalence.

An arbitrary morphism in the Gabriel/Zisman localisation may consist of arbitrarily but finitely many numerators and denominators: Every morphism of the form \(Y \map Y'\) in \(\GabrielZismanLocalisation(\mathcal{D})\) is represented by a diagram of the form
\[\begin{tikzpicture}[baseline=(m-1-1.base)]
  \matrix (m) [diagram]{
    Y & \tilde Y_1 & Y_1 & \dots & Y_{n - 1} & Y' \\ };
  \path[->, font=\scriptsize]
    (m-1-1) edge (m-1-2)
    (m-1-3) edge (m-1-4)
            edge[den] (m-1-2)
    (m-1-5) edge (m-1-6)
            edge[den] (m-1-4);
\end{tikzpicture}\]
in \(\mathcal{D}\), where the {``}backward{''} arrows are denominators in \(\mathcal{D}\). In particular, density of \(\GabrielZismanLocalisation(F)\) means that for every object \(Y\) in \(\mathcal{D}\) there exists a diagram of the form
\[\begin{tikzpicture}[baseline=(m-1-1.base)]
  \matrix (m) [diagram]{
    F X & \tilde Y_1 & Y_1 & \dots & Y_{n - 1} & Y \\ };
  \path[->, font=\scriptsize]
    (m-1-1) edge (m-1-2)
    (m-1-3) edge (m-1-4)
            edge[den] (m-1-2)
    (m-1-5) edge (m-1-6)
            edge[den] (m-1-4);
\end{tikzpicture}\]
in \(\mathcal{D}\). Typically, in this case the {``}forward{''} arrows are also denominators in \(\mathcal{D}\), but in general even that is not guaranteed.

To obtain a suitable criterion, we suppose that density of \(\GabrielZismanLocalisation(F)\) is provided by \emph{single} denominators in~\(\mathcal{D}\), so-called \newnotion{S{\nbd}replacements}: For every object \(Y\) in \(\mathcal{D}\) there is supposed to be an object \(X\) in~\(\mathcal{C}\) and a denomina\-tor~\(q\colon F X \map Y\) in \(\mathcal{D}\).
\[\begin{tikzpicture}[baseline=(m-1-1.base)]
  \matrix (m) [diagram]{
    F X & Y \\ };
  \path[->, font=\scriptsize]
    (m-1-1) edge[den] node[above] {\(q\)} (m-1-2);
\end{tikzpicture}\]
This property will be called \newnotion{S-density} in the following. If \(F\) is S-dense and~\(\GabrielZismanLocalisation(F)\) is an equivalence, we call~\(F\) an \newnotion{S-equivalence}. With this restriction, we obtain the following result.

\begin{theorem*}[{characterisation of S-equivalences, see~\ref{cor:characterisation_of_s-equivalences_as_s-dense_s-full_and_s-faithful_morphisms_of_categories_with_denominators}}] \label{th:characterisation_of_s-equivalences}
We suppose that the denominators in \(\mathcal{D}\) are closed under composition and identities. Then \(F\) is an S-equivalence if and only if it is S-dense, S-full and S-faithful.
\end{theorem*}

Here, \newnotion{S-fullness} resp.\ \newnotion{S-faithfulness} of \(F\) is defined as fullness resp.\ faithfulness of \(\GabrielZismanLocalisation(F)\) on images of \newnotion{S-\(2\)-arrows} in \(\mathcal{D}\): For all objects \(X\) and \(X'\) in \(\mathcal{C}\) and every diagram of the form
\[\begin{tikzpicture}[baseline=(m-1-1.base)]
  \matrix (m) [diagram]{
    F X & \tilde Y' & F X' \\ };
  \path[->, font=\scriptsize]
    (m-1-1) edge node[above] {\(g\)} (m-1-2)
    (m-1-3) edge[den] node[above] {\(b\)} (m-1-2);
\end{tikzpicture}\]
in \(\mathcal{D}\) there is a unique morphism \(\varphi\colon X \map X'\) in \(\GabrielZismanLocalisation(\mathcal{C})\) such that
\[\LocalisationFunctor(g) \, \LocalisationFunctor(b)^{- 1} = \GabrielZismanLocalisation(F) \varphi\]
in \(\GabrielZismanLocalisation(\mathcal{D})\). So the theorem states that being an S-equivalence can be decided by investigating morphisms in~\(\GabrielZismanLocalisation(\mathcal{D})\) that consist of precisely one numerator and precisely one denominator.

This characterisation of S-equivalences is based on the S-approximation theorem~\ref{th:s-approximation_theorem}, where an isomorphism inverse \(\InducedFunctorOfSReplacementFunctor\colon \GabrielZismanLocalisation(\mathcal{D}) \map \GabrielZismanLocalisation(\mathcal{C})\) to \(\GabrielZismanLocalisation(F)\colon \GabrielZismanLocalisation(\mathcal{C}) \map \GabrielZismanLocalisation(\mathcal{D})\) is explicitly constructed using a choice of an S{\nbd}replacement for every object in \(\mathcal{D}\).

A classical instance of this result is the fact that the inclusion of the category of bounded above complexes with entries in projective modules into the category of bounded above complexes with entries in all modules induces an equivalence between the according derived categories, where an isomorphism inverse on the derived categories is provided by pointwise projective replacements (aka projective resolutions of complexes). More generally, the inclusion of the full subcategory of cofibrant objects in a model category in the sense of \eigenname{Quillen}~\cite[ch.~I, sec.~1, def.~1]{quillen:1967:homotopical_algebra} or, even more generally, in a right derivable category in the sense of \eigenname{Cisinski}~\cite[2.22, dual of~1.1]{cisinski:2010:categories_derivables} is always an S-equivalence, where an isomorphism inverse to the induced functor on the homotopy categories is provided by cofibrant replacements.

Sufficient criteria for S-equivalences have been established by \eigenname{R{\u{a}}dulescu-Banu}~\cite[th.~5.5.1]{radulescu-banu:2006:cofibrations_in_homotopy_theory} and by \eigenname{Kahn} and \eigenname{Sujatha}~\cite[dual of th.~2.1, dual of cor.~4.4]{kahn_sujatha:2007:a_few_localisation_theorems}. Many techniques used in this article are similar to the techniques used in these two sources. In particular, to verify that \(\GabrielZismanLocalisation(F)\) is an equivalence of categories in their frameworks, \eigenname{R{\u{a}}dulescu-Banu} as well as \eigenname{Kahn} and \eigenname{Sujatha} also constructed an explicit isomorphism inverse, respectively. The advantage of these two sufficient approaches is their easier verifiability: Although S{\nbd}fullness and S{\nbd}faithfulness are particular cases of fullness and faithfulness of the induced functor on the localisation level, these properties still involve arbitrary morphisms in the localisation of the start category with denominators. As it can be hard to check S-faithfulness, it would be desirable to have a {``}decomposition{''} of this axiom into a conjunction of simpler conditions.

In his framework of left exact functors between left derivable categories, \eigenname{Cisinski} has given in~\cite[th.~3.19]{cisinski:2010:categories_derivables} a characterisation of morphisms whose right derived functor is an equivalence. Since density is (in general) obtained by zigzags of \emph{two} denominators in his theory, this approach is independent of the one presented in this article.

\paragraph*{Outline} Some preliminaries on localisations of categories are recalled in section~\ref{sec:preliminaries}. In section~\ref{sec:s-replacements}, our main tools for the construction of an isomorphism inverse to \(\GabrielZismanLocalisation(F)\), the S-replacements, are introduced. S{\nbd}equivalences and their characterising conditions are defined in section~\ref{sec:s-equivalences_and_the_characterising_conditions}. The final and main part of the article is section~\ref{sec:s-replacement_functors_and_the_s-approximation_theorem}, where an isomorphism inverse to \(\GabrielZismanLocalisation(F)\) is constructed.

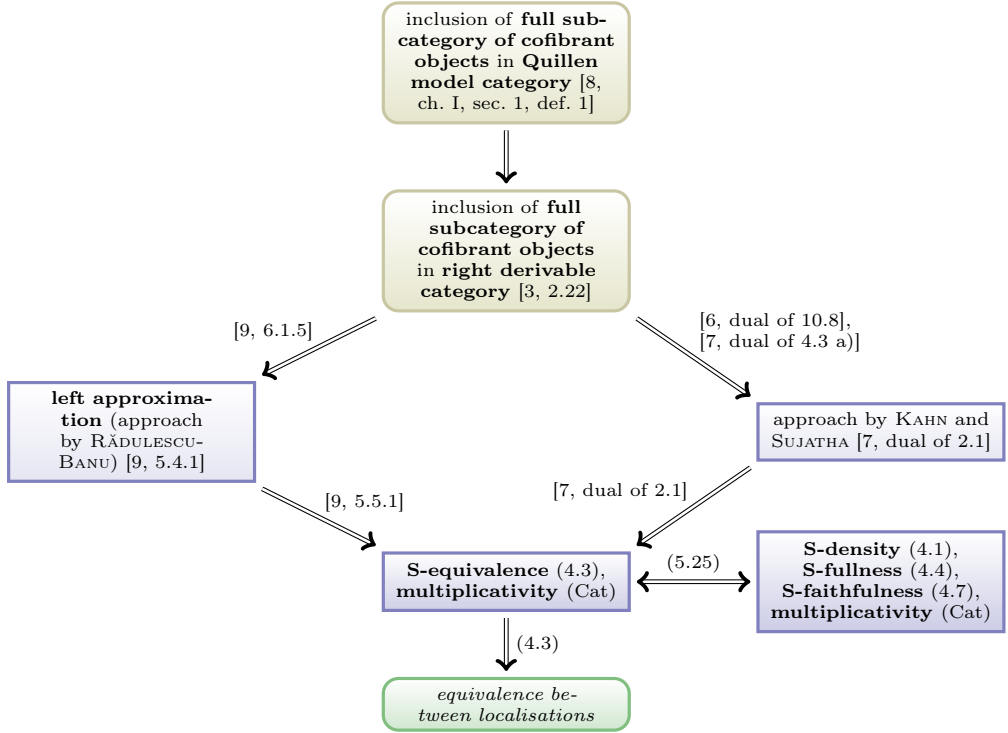
\begin{figure}[t] 
\centering
\begin{tikzpicture}[node distance=2.0em and 4.25em, font=\scriptsize]
  \node[input]                         (inc cof qmc) {inclusion of \textbf{full subcategory of cofibrant objects} in \textbf{Quillen model category}~\cite[ch.~I, sec.~1, def.~1]{quillen:1967:homotopical_algebra}};

  \node[input, below=of inc cof qmc]   (inc cof rdc) {inclusion of \textbf{full subcategory of cofibrant objects} in \textbf{right derivable category}~\cite[2.22]{cisinski:2010:categories_derivables}};
  
  \node[phantom, minimum height=4em, below=of inc cof rdc] (X) {};
  \node[old machine, left=of X]        (rbsaf) {\textbf{left approximation} (approach by \eigenname{R{\u{a}}dulescu-Banu}) \cite[5.4.1]{radulescu-banu:2006:cofibrations_in_homotopy_theory}};
  \node[old machine, right=of X]       (sskssaf) {approach by \eigenname{Kahn} and \eigenname{Sujatha} \cite[dual of~2.1]{kahn_sujatha:2007:a_few_localisation_theorems}};

  \node[machine, below=of X] (seq) {\textbf{S-equivalence} \ref{def:s-equivalence}, \textbf{multiplicativity} (Cat)};
  \node[machine, right=of seq]         (char of seq) {\textbf{S-density} \ref{def:s-density}, \textbf{S-fullness} \ref{def:s-fullness}, \textbf{S{\nbd}faithfulness} \ref{def:s-faithfulness}, \textbf{multiplicativity} (Cat)};

  \node[output, below=of seq]          (eq loc) {\textit{equivalence between localisations}};

  \path[font=\scriptsize]
    (inc cof qmc)            edge[implies] (inc cof rdc)
    (inc cof rdc.south east) edge[implies] node[above right=-2pt, align=left] {\cite[dual of~10.8]{kahn_maltsiniotis:2008:structures_de_derivabilite}, \\ \cite[dual of 4.3 a)]{kahn_sujatha:2007:a_few_localisation_theorems}} (sskssaf.north west)
    (inc cof rdc.south west) edge[implies] node[above left=-2pt] {\cite[6.1.5]{radulescu-banu:2006:cofibrations_in_homotopy_theory}} (rbsaf.north east)
    (rbsaf.south east)       edge[implies] node[above right=-2pt] {\cite[5.5.1]{radulescu-banu:2006:cofibrations_in_homotopy_theory}} (seq.north west)
    (sskssaf.south west)     edge[implies] node[above left=-2pt] {\cite[dual of~2.1]{kahn_sujatha:2007:a_few_localisation_theorems}} (seq.north east)
    (seq)        edge[logicallyequivalent] node[above] {\ref{cor:characterisation_of_s-equivalences_as_s-dense_s-full_and_s-faithful_morphisms_of_categories_with_denominators}} (char of seq)
                 edge[implies] node[right] {\ref{def:s-equivalence}} (eq loc);
\end{tikzpicture}

\caption{S-equivalences: a concept for equivalences between localisations.} \label{fig:s-equivalences_a_concept_for_equivalences_between_localisations}
\end{figure}

\subsection*{Conventions and notations} \label{sec:introduction:conventions_and_notations}

We use the following conventions and notations.

\begin{itemize}
\item To avoid set-theoretical difficulties, we (implicitly) work with Grothendieck universes~\cite[exp.~I, sec.~0]{artin_grothendieck_verdier:1972:sga_4_1}. In particular, every category has a \emph{set} of objects and a \emph{set} of morphisms.
\item The composite of morphisms \(f\colon X \map Y\) and \(g\colon Y \map Z\) is usually denoted by \(f g\colon X \map Z\). The composite of functors \(F\colon \mathcal{C} \map \mathcal{D}\) and \(G\colon \mathcal{D} \map \mathcal{E}\) is usually denoted by \(G \comp F\colon \mathcal{C} \map \mathcal{E}\).
\item Given objects \(X\) and \(Y\) in a category \(\mathcal{C}\), we denote the set of morphisms from \(X\) to \(Y\) by \({_{\mathcal{C}}}(X, Y)\).
\item Given a functor \(F\colon \mathcal{C} \map \mathcal{D}\), we denote its map on the objects by \(\Ob F\colon \Ob \mathcal{C} \map \Ob \mathcal{D}\), its map on the morphisms by \(\Mor F\colon \Mor \mathcal{C} \map \Mor \mathcal{D}\), and its maps on the hom-sets by \(F_{X, X'}\colon {_{\mathcal{C}}}(X, X') \map {_{\mathcal{D}}}(F X, F X')\) for \(X, X' \in \Ob \mathcal{C}\) .
\item If \(X\) is isomorphic to \(Y\), we write \(X \isomorphic Y\).
\item Given a set \(X\), we denote the identity map of \(X\) by \(\id_X\colon X \map X\). Likewise, given a category \(\mathcal{C}\), we denote the identity functor of \(\mathcal{C}\) by \(\id_{\mathcal{C}}\colon \mathcal{C} \map \mathcal{C}\).
\item We suppose given categories \(\mathcal{C}\) and \(\mathcal{D}\). A functor \(F\colon \mathcal{C} \map \mathcal{D}\) is said to be an equivalence (of categories) if there exists a functor \(G\colon \mathcal{D} \map \mathcal{C}\) such that \(G \comp F \isomorphic \id_{\mathcal{C}}\) and \(F \comp G \isomorphic \id_{\mathcal{D}}\). Such a functor \(G\) is then called an isomorphism inverse of \(F\). The categories \(\mathcal{C}\) and \(\mathcal{D}\) are said to be equivalent, written \(\mathcal{C} \categoricallyequivalent \mathcal{D}\), if an equivalence of categories \(F\colon \mathcal{C} \map \mathcal{D}\) exists.
\item We use the notation \(\Naturals = \set{1, 2, 3, \dots}\).
\item Given \(a, b \in \Integers\) with \(a \leq b + 1\), we write \([a, b] := \set[z \in \Integers]{a \leq z \leq b}\) for the set of integers lying between~\(a\) and \(b\).
\item When defining a category via its hom-sets, these are considered to be formally disjoint. In other words, a morphism between two given objects may be formally seen as a triple consisting of an underlying morphism and its source and target object.
\end{itemize}

\paragraph*{A comment on the terminology} \label{sec:introduction:a_comment_on_the_terminology}

The notions S-replacements, S-density, S-fullness, \dots{} are adapted to the notion of an S-\(2\)-arrow, whereas the terminology of an S-\(2\)-arrow is inspired from~\cite[def.~4.2]{thomas:2011:on_the_3-arrow_calculus_for_homotopy_categories}: an S-\(2\)-arrow may be interpreted as a \(3\)-arrow whose {``}T-part{''} is trivial. The dual concepts may be named \newnotion{T-replacements}, \newnotion{T-density}, \newnotion{T-fullness}, \dots, respectively.

\section{Preliminaries} \label{sec:preliminaries}

In this section, we collect some preliminaries, particularly on localisations, connectedness and contractibility of categories. Its main purpose is to fix notation and terminology.

\subsection*{Categories with denominators} \label{sec:preliminaries:categories_with_denominators}

A \newnotion{category with denominators} (\footnote{\eigenname{Kahn} and \eigenname{Maltsiniotis} use the terminology \newnotion{localisateur} (\newnotion{localisator})~\cite[sec.~3.1]{kahn_maltsiniotis:2008:structures_de_derivabilite}.}) consists of a category \(\mathcal{C}\) together with a subset \(D\) of \(\Mor \mathcal{C}\). By abuse of notation, we refer to the said category with denominators as well as to its underlying category by \(\mathcal{C}\). The elements of \(D\) are called \newnotion{denominators} (\footnote{\eigenname{Kahn} and \eigenname{Maltsiniotis} use the terminology \newnotion{{\'{e}}quivalences faibles} (\newnotion{weak equivalences})~\cite[sec.~3.1]{kahn_maltsiniotis:2008:structures_de_derivabilite}.}) in \(\mathcal{C}\).

Given a category with denominators \(\mathcal{C}\) with set of denominators \(D\), we write \(\Denominators \mathcal{C} := D\). In diagrams, a denominator \(d\colon X \map Y\) in \(\mathcal{C}\) will usually be depicted as
\[\begin{tikzpicture}[baseline=(m-1-1.base)]
  \matrix (m) [diagram]{
    X & Y \\ };
  \path[->, font=\scriptsize]
    (m-1-1) edge[den] node[above] {\(d\)} (m-1-2);
\end{tikzpicture}.\]

Given categories with denominators \(\mathcal{C}\) and \(\mathcal{D}\), a \newnotion{morphism of categories with denominators} from \(\mathcal{C}\) to \(\mathcal{D}\) is a functor \(F\colon \mathcal{C} \map \mathcal{D}\) that \newnotion{preserves denominators}, that is, such that \(F d\) is a denominator in \(\mathcal{D}\) for every denominator \(d\) in \(\mathcal{C}\).

Given morphisms of categories with denominators \(F, G\colon \mathcal{C} \map \mathcal{D}\), a \newnotion{\(2\)-morphism of categories with denominators} from \(F\) to \(G\) is a transformation \(\alpha\colon F \map G\).

Given categories with denominators \(\mathcal{C}\) and \(\mathcal{D}\), a functor \(F\colon \mathcal{C} \map \mathcal{D}\) is said to \newnotion{reflect denominators} if every morphism \(d\) in \(\mathcal{C}\) such that \(F d\) is a denominator in \(\mathcal{D}\) is a denominator in \(\mathcal{C}\).

\subsection*{Multiplicativity and isosaturatedness} \label{sec:preliminaries:multiplicativity_and_isosaturatedness}

A category with denominators \(\mathcal{C}\) is said to be \newnotion{multiplicative} if the following holds.
\begin{itemize}
\item[(Cat)] \emph{Multiplicativity}. For all denominators \(d\colon X \map \tilde X\) and \(e\colon \tilde X \map \bar X\) in \(\mathcal{C}\), the composite \(d e\colon X \map \bar X\) is also a denominator in \(\mathcal{C}\). For every object \(X\) in \(\mathcal{C}\), the identity \(1_X\colon X \map X\) is a denominator in \(\mathcal{C}\).
\end{itemize}

A category with denominators \(\mathcal{C}\) is said to be \newnotion{isosaturated} if the following holds.
\begin{itemize}
\item[(Iso)] \emph{Isosaturatedness}. Every isomorphism in \(\mathcal{C}\) is a denominator in \(\mathcal{C}\).
\end{itemize}

While multiplicativity of categories with denominators occurs quite often troughout this article, the notion of isosaturatedness is solely used in proposition~\ref{prop:characterisation_of_s-replacement_axiom}.

\subsection*{Localisations} \label{sec:preliminaries:localisations}

We suppose given a category with denominators \(\mathcal{C}\). A \newnotion{localisation} of \(\mathcal{C}\) consists of a category \(\mathcal{L}\) and a functor~\(L\colon \mathcal{C} \map \mathcal{L}\) with \(L d\) invertible in \(\mathcal{L}\) for every denominator \(d\) in \(\mathcal{C}\), and such that for every category~\(\mathcal{D}\) and every functor~\(F\colon \mathcal{C} \map \mathcal{D}\) with \(F d\) invertible in \(\mathcal{D}\) for every denominator \(d\) in \(\mathcal{C}\), there exists a unique functor~\(\hat{F}\colon \mathcal{L} \map \mathcal{D}\) with \(F = \hat{F} \comp L\).
\[\begin{tikzpicture}[baseline=(m-2-1.base)]
  \matrix (m) [diagram]{
    \mathcal{C} & \mathcal{D} \\
    \mathcal{L} & \\ };
  \path[->, font=\scriptsize]
    (m-1-1) edge node[above] {\(F\)} (m-1-2)
            edge node[left] {\(L\)} (m-2-1)
    (m-2-1) edge[exists] node[right] {\(\hat{F}\)} (m-1-2);
\end{tikzpicture}\]
By abuse of notation, we refer to the said localisation as well as to its underlying category by \(\mathcal{L}\). The functor~\(L\) is called the \newnotion{localisation functor} of \(\mathcal{L}\).

Given a localisation \(\mathcal{L}\) of \(\mathcal{C}\) with localisation functor \(L\colon \mathcal{C} \map \mathcal{L}\), we write \(\LocalisationFunctor = \LocalisationFunctor[\mathcal{L}] := L\).

A localisation \(\mathcal{L}\) also fulfils the following universal property with respect to transformations, see e.g.~\cite[prop.~(1.15)]{thomas:2012:a_calculus_of_fractions_for_the_homotopy_category_of_a_brown_cofibration_category}: For every category \(\mathcal{D}\), all functors \(G, G'\colon \mathcal{L} \map \mathcal{D}\) and every transformation \(\alpha\colon G \comp \LocalisationFunctor \map G' \comp \LocalisationFunctor\) there exists a unique transformation \(\hat \alpha\colon G \map G'\) with \(\alpha = \hat \alpha * \LocalisationFunctor\).
\[\begin{tikzpicture}[baseline=(m-2-1.base)]
  \matrix (m) [diagram=5.0em]{
    \mathcal{C} & \mathcal{D} \\
    \mathcal{L} & \\ };
  \path[->, font=\scriptsize]
    (m-1-1) edge[out=20, in=160] node[above] {\(G \comp \LocalisationFunctor\)} node (F) {} (m-1-2)
            edge[out=-20, in=-160] node[below=-2pt, near start] {\(G' \comp \LocalisationFunctor\)} node (G) {} (m-1-2)
            edge node[left] {\(\LocalisationFunctor\)} (m-2-1)
    (m-2-1) edge[out=60, in=-150] node[left] {\(G\)} node[sloped] (Fhat) {} (m-1-2)
            edge[out=30, in=-120] node[right] {\(G'\)} node[sloped] (Ghat) {} (m-1-2);
  \path[->, font=\scriptsize]
    (Fhat) edge[exists] node[above right] {\(\hat{\alpha}\)} (Ghat)
    (F)    edge node[right] {\(\alpha\)} (G);
\end{tikzpicture}\]

\subsection*{The Gabriel/Zisman localisation} \label{sec:preliminaries:the_gabriel-zisman_localisation}

We suppose given a category with denominators \(\mathcal{C}\). In~\cite[sec.~1.1]{gabriel_zisman:1967:calculus_of_fractions_and_homotopy_theory}, \eigenname{Gabriel} and \eigenname{Zisman} constructed a localisation of \(\mathcal{C}\). We call this particular localisation the \newnotion{Gabriel/Zisman localisation} of \(\mathcal{C}\) and denote it by~\(\GabrielZismanLocalisation(\mathcal{C})\). As the notion of a localisation is defined by a universal property, a localisation of \(\mathcal{C}\) is unique up to isomorphism.

We will use the following two facts about the Gabriel/Zisman localisation of a category with denominators \(\mathcal{C}\): First, the localisation functor \(\LocalisationFunctor\colon \mathcal{C} \map \GabrielZismanLocalisation(\mathcal{C})\) is given on the objects by
\[\Ob \LocalisationFunctor = \id_{\Ob \mathcal{C}} = \id_{\Ob \GabrielZismanLocalisation(\mathcal{C})}.\]
Second, for every morphism \(\varphi\colon X \map X'\) in \(\GabrielZismanLocalisation(\mathcal{C})\) there exist \(n \in \Naturals\), morphisms \(f_i\colon X_{i - 1} \map \tilde X_i\) in~\(\mathcal{C}\) for \(i \in [1, n]\) and denominators \(a_i\colon X_i \map \tilde X_i\) in~\(\mathcal{C}\) for \(i \in [1, n - 1]\) with \(X = X_0\), \(X' = \tilde X_n\) and such that
\[\varphi = \LocalisationFunctor(f_1) \, \LocalisationFunctor(a_1)^{- 1} \, \LocalisationFunctor(f_2) \, \dots \, \LocalisationFunctor(a_{n - 1})^{- 1} \, \LocalisationFunctor(f_n)\]
in \(\GabrielZismanLocalisation(\mathcal{C})\).
\[\begin{tikzpicture}[baseline=(m-1-1.base)]
  \matrix (m) [diagram]{
    X & \tilde X_1 & X_1 & \dots & X_{n - 1} & X' \\ };
  \path[->, font=\scriptsize]
    (m-1-1) edge[exists] node[above] {\(f_1\)} (m-1-2)
    (m-1-3) edge[exists] node[above] {\(f_2\)} (m-1-4)
            edge[exists, den] node[above] {\(a_1\)} (m-1-2)
    (m-1-5) edge[exists] node[above] {\(f_n\)} (m-1-6)
            edge[exists, den] node[above] {\(a_{n - 1}\)} (m-1-4);
\end{tikzpicture}\]

The Gabriel/Zisman localisation turns into a \(2\)-functor from the \(2\)-category of categories with denominators in a Grothendieck universe to the \(2\)-category of categories in this Grothendieck universe as follows. Given a morphism of categories with denominators \(F\colon \mathcal{C} \map \mathcal{D}\), then \(\GabrielZismanLocalisation(F)\colon \GabrielZismanLocalisation(\mathcal{C}) \map \GabrielZismanLocalisation(\mathcal{D})\) is the unique functor with \(\LocalisationFunctor[\GabrielZismanLocalisation(\mathcal{D})] \comp F = \GabrielZismanLocalisation(F) \comp \LocalisationFunctor[\GabrielZismanLocalisation(\mathcal{C})]\). Given morphisms of categories with denominators \(F, F'\colon \mathcal{C} \map \mathcal{D}\) and a \(2\)-morphism of categories with denominators \(\alpha\colon F \map F'\), the transformation \(\GabrielZismanLocalisation(\alpha)\colon \GabrielZismanLocalisation(F) \map \GabrielZismanLocalisation(F')\) is the unique transformation with \(\LocalisationFunctor[\GabrielZismanLocalisation(\mathcal{D})] * \alpha = \GabrielZismanLocalisation(\alpha) * \LocalisationFunctor[\GabrielZismanLocalisation(\mathcal{C})]\).
\[\begin{tikzpicture}[baseline=(m-2-1.base)]
  \matrix (m) [diagram=5.0em]{
    \mathcal{C} & \mathcal{D} \\
    \GabrielZismanLocalisation(\mathcal{C}) & \GabrielZismanLocalisation(\mathcal{D}) \\ };
  \path[->, font=\scriptsize]
    (m-1-1) edge[out=20, in=160] node[above] {\(F\)} node (F) {} (m-1-2)
            edge[out=-20, in=-160] node[below] {\(F'\)} node (G) {} (m-1-2)
            edge node[left] {\(\LocalisationFunctor\)} (m-2-1)
    (m-1-2) edge node[right] {\(\LocalisationFunctor\)} (m-2-2)
    (m-2-1) edge[out=20, in=160] node[above] {\(\GabrielZismanLocalisation(F)\)} node[sloped] (Fhat) {} (m-2-2)
            edge[out=-20, in=-160] node[below] {\(\GabrielZismanLocalisation(F')\)} node[sloped] (Ghat) {} (m-2-2);
  \path[->, font=\scriptsize]
    (Fhat) edge node[right] {\(\GabrielZismanLocalisation(\alpha)\)} (Ghat)
    (F)    edge node[right] {\(\alpha\)} (G);
\end{tikzpicture}\]

In this article, we study conditions on a morphism of categories with denominators \(F\colon \mathcal{C} \map \mathcal{D}\) implying that~\(\GabrielZismanLocalisation(F)\colon \GabrielZismanLocalisation(\mathcal{C}) \map \GabrielZismanLocalisation(\mathcal{D})\) is an equivalence of categories.

\subsection*{S-\(2\)-arrows} \label{sec:preliminaries:s-2-arrows}

We suppose given a category with denominators \(\mathcal{C}\). An \newnotion{S-\(2\)-arrow} in \(\mathcal{C}\) is a diagram
\[\begin{tikzpicture}[baseline=(m-1-1.base)]
  \matrix (m) [diagram]{
    X & \tilde Y & Y \\ };
  \path[->, font=\scriptsize]
    (m-1-1) edge node[above] {\(f\)} (m-1-2)
    (m-1-3) edge[den] node[above] {\(a\)} (m-1-2);
\end{tikzpicture}\]
in \(\mathcal{C}\) where \(a\) is supposed to be a denominator, denoted by \((f, a)\colon \Stwoarrow{X}{\tilde Y}{Y}\).

S-\(2\)-arrows are usually used in the description of locations of well-behaved categories with denominators, see e.g.~\cite[ch.~I, sec.~2.2, sec.~2.3]{gabriel_zisman:1967:calculus_of_fractions_and_homotopy_theory}, \cite[th.~1]{brown:1974:abstract_homotopy_theory_and_generalized_sheaf_cohomology}, \cite[sec.~III.2, lem.~8]{gelfand_manin:2003:methods_of_homological_algebra}, \cite[th.~(2.35), th.~(2.37), th.~(3.128), rem.~(3.129)]{thomas:2012:a_calculus_of_fractions_for_the_homotopy_category_of_a_brown_cofibration_category}, where every morphism in the localisation is represented by an S-\(2\)-arrow. We will not do so in this article; instead, we will use the notion of an S-\(2\)-arrow in the formulation of the characterising conditions for S-equivalences in section~\ref{sec:s-equivalences_and_the_characterising_conditions}.

\subsection*{On the construction of isomorphism inverses on localisation level} \label{sec:preliminaries:on_the_construction_of_isomorphism_inverses_on_localisation_level}

We suppose given a morphism of categories with denominators \(F\colon \mathcal{C} \map \mathcal{D}\). In corollary~\ref{cor:criterion_for_inducing_an_equivalence_on_localisations}, we characterise those functors \(G\colon \mathcal{D} \map \GabrielZismanLocalisation(\mathcal{C})\) that induce an isomorphism inverse to \(\GabrielZismanLocalisation(F)\colon \GabrielZismanLocalisation(\mathcal{C}) \map \GabrielZismanLocalisation(\mathcal{D})\). This criterion is most likely folklore.

Remark~\ref{rem:criterion_for_induced_co_retraction_up_to_isomorphism_on_localisations}\ref{rem:criterion_for_induced_co_retraction_up_to_isomorphism_on_localisations:retraction}\ref{rem:criterion_for_induced_co_retraction_up_to_isomorphism_on_localisations:retraction:sufficient} will be used in the proof of corollary~\ref{cor:canonical_lift_along_forgetful_functor_of_s-replacement_category_induces_equivalence_on_localisations}\ref{cor:canonical_lift_along_forgetful_functor_of_s-replacement_category_induces_equivalence_on_localisations:retraction_up_to_isomorphism_and_forgetful_functor}, \ref{cor:canonical_lift_along_forgetful_functor_of_s-replacement_category_induces_equivalence_on_localisations:retraction_up_to_isomorphism}.

\begin{remark} \label{rem:criterion_for_induced_co_retraction_up_to_isomorphism_on_localisations}
We suppose given a morphism of categories with denominators \(F\colon \mathcal{C} \map \mathcal{D}\) and a func- \linebreak 
tor~\(G'\colon \GabrielZismanLocalisation(\mathcal{D}) \map \GabrielZismanLocalisation(\mathcal{C})\).
\begin{enumerate}
\item \label{rem:criterion_for_induced_co_retraction_up_to_isomorphism_on_localisations:coretraction}
\begin{enumerate}
\item \label{rem:criterion_for_induced_co_retraction_up_to_isomorphism_on_localisations:coretraction:necessary} Given an isotransformation \(\alpha'\colon G' \comp \GabrielZismanLocalisation(F) \map \id_{\GabrielZismanLocalisation(\mathcal{C})}\), then \(\alpha' * \LocalisationFunctor[\GabrielZismanLocalisation(\mathcal{C})]\) is an isotransformation from~\((G' \comp \LocalisationFunctor[\GabrielZismanLocalisation(\mathcal{D})]) \comp F\) to \(\LocalisationFunctor[\GabrielZismanLocalisation(\mathcal{C})]\).
\item \label{rem:criterion_for_induced_co_retraction_up_to_isomorphism_on_localisations:coretraction:sufficient} Given an isotransformation \(\alpha\colon (G' \comp \LocalisationFunctor[\GabrielZismanLocalisation(\mathcal{D})]) \comp F \map \LocalisationFunctor[\GabrielZismanLocalisation(\mathcal{C})]\), then there exists a unique transformation~\(\hat \alpha\colon G' \comp \GabrielZismanLocalisation(F) \map \id_{\GabrielZismanLocalisation(\mathcal{C})}\) with \(\alpha = \hat \alpha * \LocalisationFunctor[\GabrielZismanLocalisation(\mathcal{C})]\), and this transformation~\(\hat \alpha\) is an isotransformation.
\end{enumerate}
\item \label{rem:criterion_for_induced_co_retraction_up_to_isomorphism_on_localisations:retraction}
\begin{enumerate}
\item \label{rem:criterion_for_induced_co_retraction_up_to_isomorphism_on_localisations:retraction:necessary} Given an isotransformation \(\alpha'\colon \GabrielZismanLocalisation(F) \comp G' \map \id_{\GabrielZismanLocalisation(\mathcal{D})}\), then \(\alpha' * \LocalisationFunctor[\GabrielZismanLocalisation(\mathcal{D})]\) is an isotransformation from~\(\GabrielZismanLocalisation(F) \comp (G' \comp \LocalisationFunctor[\GabrielZismanLocalisation(\mathcal{D})])\) to \(\LocalisationFunctor[\GabrielZismanLocalisation(\mathcal{D})]\).
\item \label{rem:criterion_for_induced_co_retraction_up_to_isomorphism_on_localisations:retraction:sufficient} Given an isotransformation \(\alpha\colon \GabrielZismanLocalisation(F) \comp (G' \comp \LocalisationFunctor[\GabrielZismanLocalisation(\mathcal{D})]) \map \LocalisationFunctor[\GabrielZismanLocalisation(\mathcal{D})]\), then there exists a unique transformation~\(\hat \alpha\colon \GabrielZismanLocalisation(F) \comp G' \map \id_{\GabrielZismanLocalisation(\mathcal{D})}\) with \(\alpha = \hat \alpha * \LocalisationFunctor[\GabrielZismanLocalisation(\mathcal{D})]\), and this transformation~\(\hat \alpha\) is an isotransformation.
\end{enumerate}
\end{enumerate}
\end{remark}
\begin{proof} \
\begin{enumerate}
\item
\begin{enumerate}
\item As \(\alpha'\colon G' \comp \GabrielZismanLocalisation(F) \map \id_{\GabrielZismanLocalisation(\mathcal{C})}\) is an isotransformation, the transformation \(\alpha' * \LocalisationFunctor[\GabrielZismanLocalisation(\mathcal{C})]\) is an isotransformation from \(G' \comp \GabrielZismanLocalisation(F) \comp \LocalisationFunctor[\GabrielZismanLocalisation(\mathcal{C})] = G' \comp \LocalisationFunctor[\GabrielZismanLocalisation(\mathcal{D})] \comp F\) to \(\id_{\GabrielZismanLocalisation(\mathcal{C})} \comp \LocalisationFunctor[\GabrielZismanLocalisation(\mathcal{C})] = \LocalisationFunctor[\GabrielZismanLocalisation(\mathcal{C})]\).
\item As \(G' \comp \LocalisationFunctor[\GabrielZismanLocalisation(\mathcal{D})] \comp F = G' \comp \GabrielZismanLocalisation(F) \comp \LocalisationFunctor[\GabrielZismanLocalisation(\mathcal{C})]\), there exists a unique transformation \(\hat \alpha\colon G' \comp \GabrielZismanLocalisation(F) \map \id_{\GabrielZismanLocalisation(\mathcal{C})}\) with \(\alpha = \hat \alpha * \LocalisationFunctor[\GabrielZismanLocalisation(\mathcal{C})]\), see e.g.~\cite[prop.~(1.16)]{thomas:2012:a_calculus_of_fractions_for_the_homotopy_category_of_a_brown_cofibration_category}. Moreover, \(\hat \alpha\) is an isotransformation, see e.g.~\cite[cor.~(1.18)]{thomas:2012:a_calculus_of_fractions_for_the_homotopy_category_of_a_brown_cofibration_category}.
\end{enumerate}
\item
\begin{enumerate}
\setcounter{enumii}{1}
\item This follows e.g.\ from~\cite[prop.~(1.16), cor.~(1.18)]{thomas:2012:a_calculus_of_fractions_for_the_homotopy_category_of_a_brown_cofibration_category}. \qedhere
\end{enumerate}
\end{enumerate}
\end{proof}

\begin{corollary} \label{cor:criterion_for_inducing_an_equivalence_on_localisations}
We suppose given a morphism of categories with denominators \(F\colon \mathcal{C} \map \mathcal{D}\) and a func- \linebreak 
tor \(G\colon \mathcal{D} \map \GabrielZismanLocalisation(\mathcal{C})\) that maps denominators in \(\mathcal{D}\) to isomorphisms in \(\GabrielZismanLocalisation(\mathcal{C})\) and we let \(\hat G\colon \GabrielZismanLocalisation(\mathcal{D}) \map \GabrielZismanLocalisation(\mathcal{C})\) be the unique functor with \(G = \hat G \comp \LocalisationFunctor[\GabrielZismanLocalisation(\mathcal{D})]\). Moreover, we suppose given an isotransformation \(\alpha\colon G \comp F \map \LocalisationFunctor[\GabrielZismanLocalisation(\mathcal{C})]\) and an isotransformation~\(\beta\colon \GabrielZismanLocalisation(F) \comp G \map \LocalisationFunctor[\GabrielZismanLocalisation(\mathcal{D})]\) and we let \(\hat \alpha\colon \hat G \comp \GabrielZismanLocalisation(F) \map \id_{\GabrielZismanLocalisation(\mathcal{C})}\) be the unique transformation with \(\alpha = \hat \alpha * \LocalisationFunctor[\GabrielZismanLocalisation(\mathcal{C})]\) and we let~\(\hat \beta\colon \GabrielZismanLocalisation(F) \comp \hat G \map \id_{\GabrielZismanLocalisation(\mathcal{D})}\) be the unique transformation with~\(\beta = \hat \beta * \LocalisationFunctor[\GabrielZismanLocalisation(\mathcal{D})]\). Then \(\hat \alpha\) and \(\hat \beta\) are isotransformations. In particular, \(\GabrielZismanLocalisation(F)\) and \(\hat G\) are mutually isomorphism inverse equivalences of categories.
\end{corollary}
\begin{proof}
The transformation~\(\hat \alpha\) is an isotransformation by remark~\ref{rem:criterion_for_induced_co_retraction_up_to_isomorphism_on_localisations}\ref{rem:criterion_for_induced_co_retraction_up_to_isomorphism_on_localisations:coretraction}\ref{rem:criterion_for_induced_co_retraction_up_to_isomorphism_on_localisations:coretraction:sufficient} and the transformation \(\hat \beta\) is an isotransformation by remark~\ref{rem:criterion_for_induced_co_retraction_up_to_isomorphism_on_localisations}\ref{rem:criterion_for_induced_co_retraction_up_to_isomorphism_on_localisations:retraction}\ref{rem:criterion_for_induced_co_retraction_up_to_isomorphism_on_localisations:retraction:sufficient}.
\end{proof}

\subsection*{A construction principle for functors via choices} \label{sec:preliminaries:a_construction_principle_for_functors_via_choices}

We recall from~\cite[app.~A, sec.~1]{thomas:2012:a_calculus_of_fractions_for_the_homotopy_category_of_a_brown_cofibration_category} a systematic method to construct a functor whose map on the objects depends on a choice.

We suppose given a category \(\mathcal{C}\) and a family \(\mathfrak{S} = (\mathfrak{S}_{X})_{X \in \Ob \mathcal{C}}\) over~\(\Ob \mathcal{C}\). The \newnotion{structure category} of~\(\mathcal{C}\) with respect to \(\mathfrak{S}\) is the category \(\StructureCategory{\mathcal{C}}{\mathfrak{S}}\) given as follows. The set of objects of \(\mathcal{C}_{\mathfrak{S}}\) is given by \(\Ob{\mathcal{C}_{\mathfrak{S}}} =\) \linebreak 
\(\set[(X, S)]{\text{\(X \in \Ob \mathcal{C}\), \(S \in \mathfrak{S}_X\)}}\). Given objects \((X, S)\), \((Y, T)\) in \(\StructureCategory{\mathcal{C}}{\mathfrak{S}}\), we have the hom-set \({_{\StructureCategory{\mathcal{C}}{\mathfrak{S}}}}((X, S), (Y, T)) = \set[(f, S, T)]{f \in {_{\mathcal{C}}}(X, Y)}\). The composite of morphisms \((f, S, T)\colon (X, S) \map (Y, T)\) and \((g, T, U)\colon (Y, T) \map (Z, U)\) in~\(\StructureCategory{\mathcal{C}}{\mathfrak{S}}\) is given by~\((f, S, T) (g, T, U) = (f g, S, U)\), and the identity morphism on an object \((X, S)\) in \(\StructureCategory{\mathcal{C}}{\mathfrak{S}}\) is given by~\(1_{(X, S)} = (1_X, S, S)\).

The \newnotion{forgetful functor} of \(\StructureCategory{\mathcal{C}}{\mathfrak{S}}\) is the functor \(\ForgetfulFunctor\colon \StructureCategory{\mathcal{C}}{\mathfrak{S}} \map \mathcal{C}\), \((X, S) \mapsto X\), \((f, S, T) \mapsto f\).

Given objects \((X, S)\) and \((Y, T)\) in \(\StructureCategory{\mathcal{C}}{\mathfrak{S}}\), a morphism \((f, S, T)\colon (X, S) \map (Y, T)\) in \(\StructureCategory{\mathcal{C}}{\mathfrak{S}}\) will usually be denoted just by \(f\colon (X, S) \map (Y, T)\). Moreover, we usually write \({_{\StructureCategory{\mathcal{C}}{\mathfrak{S}}}}((X, S), (Y, T)) = {_{\mathcal{C}}}(X, Y)\) instead of \({_{\StructureCategory{\mathcal{C}}{\mathfrak{S}}}}((X, S), (Y, T)) = \set[(f, S, T)]{f \in {_{\mathcal{C}}}(X, Y)}\).

Given a functor \(\bar F\colon \StructureCategory{\mathcal{C}}{\mathfrak{S}} \map \mathcal{D}\), we usually write \(\bar F_S X := \bar F{(X, S)}\) for \((X, S) \in \Ob \StructureCategory{\mathcal{C}}{\mathfrak{S}}\) and \(\bar F_{S, T} f := F{(f, S, T)}\) for a morphism \(f\colon (X, S) \map (Y, T)\) in \(\StructureCategory{\mathcal{C}}{\mathfrak{S}}\).

A \newnotion{choice of structures} for \(\mathcal{C}\) with respect to \(\mathfrak{S}\) is a family \(S = (S_X)_{X \in \Ob \mathcal{C}}\) over \(\Ob \mathcal{C}\) such that \(S_X \in \mathfrak{S}_X\) for~\(X \in \Ob \mathcal{C}\). Given a choice of structures \(S = (S_X)_{X \in \Ob \mathcal{C}}\) for \(\mathcal{C}\) with respect to \(\mathfrak{S}\), the \newnotion{structure choice functor} with respect to \(S\) is the functor \(\StructureChoiceFunctor{S}\colon \mathcal{C} \map \StructureCategory{\mathcal{C}}{\mathfrak{S}}\) given on the objects by \(\StructureChoiceFunctor{S} X = (X, S_X)\) for \(X \in \Ob \mathcal{C}\) and on the morphisms by \(\StructureChoiceFunctor{S} f = f\colon (X, S_X) \map (Y, S_Y)\) for every morphism \(f\colon X \map Y\) in \(\mathcal{C}\). It fulfils \(\ForgetfulFunctor \comp \StructureChoiceFunctor{S} = \id_{\mathcal{C}}\) and \(\StructureChoiceFunctor{S} \comp \ForgetfulFunctor \isomorphic \id_{\StructureCategory{\mathcal{C}}{\mathfrak{S}}}\), where an isotransformation \(\varepsilon\colon \StructureChoiceFunctor{S} \comp \ForgetfulFunctor \map \id_{\StructureCategory{\mathcal{C}}{\mathfrak{S}}}\) is given by \(\varepsilon_{(X, T)} = 1_X\colon (X, S_X) \map (X, T)\) for~\((X, T) \in \Ob{\StructureCategory{\mathcal{C}}{\mathfrak{S}}}\). In particular, the forgetful functor \(\ForgetfulFunctor\colon \StructureCategory{\mathcal{C}}{\mathfrak{S}} \map \mathcal{C}\) and the structure choice functor \(\StructureChoiceFunctor{S}\colon \mathcal{C} \map \StructureCategory{\mathcal{C}}{\mathfrak{S}}\) are mutually isomorphism inverse equivalences of categories.

To construct a functor \(F\colon \mathcal{C} \map \mathcal{D}\) whose definition on the objects uses a choice of structures \(S\) for \(\mathcal{C}\) with respect to \(\mathfrak{S}\), we may first construct a choice-free variant \(\bar F\colon \mathcal{C}_{\mathfrak{S}} \map \mathcal{D}\) and then define \(F := \bar F \comp \StructureChoiceFunctor{S}\). With the notations introduced above, we then have \(F X = \bar F_{S_X} X\) for every object \(X\) in \(\mathcal{C}\) and \(F f = \bar F_{S_X, S_{X'}} f\) for every morphism~\(f\colon X \map X'\) in \(\mathcal{C}\).

Given a functor \(\bar F\colon \mathcal{C}_{\mathfrak{S}} \map \mathcal{D}\) and choices of structures \(S\) and \(\tilde S\) for \(\mathcal{C}\) with respect to \(\mathfrak{S}\), then \(F_S := \bar F \comp \StructureChoiceFunctor{\tilde S}\) and~\(F_{\tilde S} := \bar F \comp \StructureChoiceFunctor{\tilde S}\) are isomorphic, an isotransformation \(\alpha_{S, \tilde S}\colon F_S \map F_{\tilde S}\) is given by \((\alpha_{S, \tilde S})_X = \bar F_{S_X, \tilde S_X} 1_X\colon F_S X \map F_{\tilde S} X\) for \(X \in \Ob \mathcal{C}\), and the inverse of \(\alpha_{S, \tilde S}\) is given by \(\alpha_{S, \tilde S}^{- 1} = \alpha_{\tilde S, S}\).

We will make use of this principle in the construction of S-replacement functors in section~\ref{sec:s-replacement_functors_and_the_s-approximation_theorem}.

\section{S-replacements} \label{sec:s-replacements}

We suppose given a morphism of categories with denominators \(F\colon \mathcal{C} \map \mathcal{D}\). If \(\GabrielZismanLocalisation(F)\colon \GabrielZismanLocalisation(\mathcal{C}) \map \GabrielZismanLocalisation(\mathcal{D})\) is an equivalence of categories, then it is in particular dense, that is, for every object \(Y\) in \(\mathcal{D}\) there is an object~\(X\) in~\(\mathcal{C}\) such that \(Y \isomorphic \GabrielZismanLocalisation(F) X = F X\) in \(\GabrielZismanLocalisation(\mathcal{D})\). Since the localisation functor \(\LocalisationFunctor\colon \mathcal{D} \map \GabrielZismanLocalisation(\mathcal{D})\) maps denominators in \(\mathcal{D}\) to isomorphisms in \(\GabrielZismanLocalisation(\mathcal{D})\), the easiest non-trivial situation where we have such an isomorphism in \(\GabrielZismanLocalisation(\mathcal{D})\) is the one where we already have a denominator \(F X \map Y\) (or, dually, a denominator~\(Y \map F X\)) in \(\mathcal{D}\).

Below we will often suppose that \(F\) admits for every object \(Y\) in \(\mathcal{D}\) an object \(X\) and a denominator~\(q\colon F X \map Y\) in \(\mathcal{D}\). In fact, to show that \(F\) induces an equivalence on localisation level (under certain additional conditions), we will use such pairs \((X, q)\) to construct an isomorphism inverse.

In the following, we will introduce terminology for these pairs and introduce a categorical setup for objects endowed with these pairs.

Throughout this section, we suppose given a morphism of categories with denominators~\(F\colon \mathcal{C} \map \mathcal{D}\). (\footnote{Some parts of this section may also make sense if \(\mathcal{C}\) is (only) supposed to be a category, \(\mathcal{D}\) is supposed to be a category with denominators and \(F\) is (only) supposed to be a functor.})

\subsection*{Concept} \label{sec:s-replacements:concept}

We begin with the definition of the basic concept of this article.

\begin{definition}[S-replacement] \label{def:s-replacement}
We suppose given an object \(Y\) in \(\mathcal{D}\). An \newnotion{S-replacement} of \(Y\) along \(F\) (\footnote{\eigenname{Kahn} and \eigenname{Maltsiniotis} use the terminology \newnotion{\(F\)-r{\'e}solution {\`a} gauche} (\newnotion{left \(F\)-resolution})~\cite[sec.~5.11, dual of d{\'e}f.~5.4]{kahn_maltsiniotis:2008:structures_de_derivabilite}.}) (or, if no confusion arises, just an \newnotion{S{\nbd}replacement} of \(Y\)) is a pair \((X, q)\) consisting of an object \(X\) in \(\mathcal{C}\) and a denominator \(q\colon F X \map Y\) in \(\mathcal{D}\).
\end{definition}

\begin{remark} \label{rem:s-replacements_and_composition_of_functors}
In addition to the morphism of categories with denominators \(F\colon \mathcal{C} \map \mathcal{D}\), we suppose given a morphism of categories with denominators \(G\colon \mathcal{D} \map \mathcal{E}\).
\begin{enumerate}
\item \label{rem:s-replacements_and_composition_of_functors:second_of_composite} Given an object \(Z\) in \(\mathcal{E}\) and an S-replacement \((X, r)\) of \(Z\) along \(G \comp F\), then \((F X, r)\) is an S-replacement of \(Z\) along \(G\).
\[\begin{tikzpicture}[baseline=(m-2-1.base)]
  \matrix (m) [diagram]{
    G F X \\
    Z \\ };
  \path[->, font=\scriptsize]
    (m-1-1) edge[den] node[right] {\(r\)} (m-2-1);
\end{tikzpicture}\]
\item \label{rem:s-replacements_and_composition_of_functors:composite} We suppose that \(\mathcal{E}\) is multiplicative. Given an object~\(Z\) in \(\mathcal{E}\), an S-replacement \((Y, r)\) of \(Z\) along \(G\) and an S-replacement \((X, q)\) of \(Y\) along~\(F\), then~\((X, (G q) r)\) is an S-replacement of \(Z\) along \(G \comp F\).
\[\begin{tikzpicture}[baseline=(m-3-1.base)]
  \matrix (m) [diagram]{
    G F X \\
    G Y \\
    Z \\ };
  \path[->, font=\scriptsize]
    (m-1-1) edge[den] node[right] {\(G q\)} (m-2-1)
    (m-2-1) edge[den] node[right] {\(r\)} (m-3-1);
\end{tikzpicture}\]
\end{enumerate}
\end{remark}

\begin{definition}[having enough S-replacements] \label{def:having_enough_s-replacements}
The category with denominators \(\mathcal{D}\) is said to \newnotion{have enough S{\nbd}replacements} along \(F\) (or, if no confusion arises, just to \newnotion{have enough~S{\nbd}replacements}) if for every object~\(Y\) in~\(\mathcal{D}\) there exists an S-replacement of \(Y\) along \(F\).
\end{definition}

\begin{remark} \label{rem:having_enough_s-replacements_and_composition_of_functors}
In addition to the morphism of categories with denominators \(F\colon \mathcal{C} \map \mathcal{D}\), we suppose given a morphism of categories with denominators \(G\colon \mathcal{D} \map \mathcal{E}\).
\begin{enumerate}
\item \label{rem:having_enough_s-replacements_and_composition_of_functors:second_of_composite} If \(\mathcal{E}\) has enough S-replacements along \(G \comp F\), then it has enough S-replacements along \(G\).
\item \label{rem:having_enough_s-replacements_and_composition_of_functors:composite} We suppose that \(\mathcal{E}\) is multiplicative. If \(\mathcal{E}\) has enough S-replacements along \(G\) and \(\mathcal{D}\) has enough S{\nbd}replacements along \(F\), then \(\mathcal{E}\) has enough S-replacements along \(G \comp F\).
\end{enumerate}
\end{remark}
\begin{proof} \
\begin{enumerate}
\item We suppose that \(\mathcal{E}\) has enough S-replacements along \(G \comp F\). Then for every object \(Z\) in \(\mathcal{E}\), there exists an S-replacement \((X, r)\) of \(Z\) along \(G \comp F\), which yields the S-replacement \((F X, r)\) of \(Z\) along \(G\). Thus \(\mathcal{E}\) has enough S-replacements along \(G\).
\item We suppose that \(\mathcal{E}\) has enough S-replacements along \(G\) and that \(\mathcal{D}\) has enough S-replacements along~\(F\). Moreover, we suppose given an object \(Z\) in \(\mathcal{E}\). Since \(\mathcal{E}\) has enough S-replacements along \(G\), there exists an S-replacement \((Y, r)\) of \(Z\) along \(G\), and since \(\mathcal{D}\) has enough S-replacements along \(F\), there exists an~S{\nbd}replacement \((X, q)\) of \(Y\) along \(F\). But then~\((X, (G q) r)\) is an S-replacement of \(Z\) along \(G \comp F\). Thus~\(\mathcal{E}\) has enough S-replacements along \(G \comp F\). \qedhere
\end{enumerate}
\end{proof}

\begin{definition}[having all trivial S-replacements] \label{def:having_all_trivial_s-replacements}
The category with denominators \(\mathcal{D}\) is said to \newnotion{have all trivial S-replacements} along \(F\) (or, if no confusion arises, just to \newnotion{have all trivial S-replacements}) if for every object~\(X\) in \(\mathcal{C}\) the identity \(1_{F X}\colon F X \map F X\) is a denominator in~\(\mathcal{D}\).
\[\begin{tikzpicture}[baseline=(m-2-1.base)]
  \matrix (m) [diagram]{
    F X \\
    F X \\ };
  \path[->, font=\scriptsize]
    (m-1-1) edge[den] node[right] {\(1_{F X}\)} (m-2-1);
\end{tikzpicture}\]
\end{definition}

\begin{remark} \label{rem:multiplicativity_implies_having_all_trivial_s-replacements}
If \(\mathcal{C}\) or \(\mathcal{D}\) is multiplicative, then \(\mathcal{D}\) has all trivial S-replacements along \(F\).
\end{remark}
 
\subsection*{The category of objects with S-replacement} \label{sec:s-replacements:the_category_of_objects_with_s-replacement}

Next, we consider structures consisting of an object in \(\mathcal{D}\) equipped with an S-replacement.

\begin{definition}[object with S-replacement] \label{def:object_with_s-replacement}
We let \(\mathfrak{R} = (\mathfrak{R}_Y)_{Y \in \Ob \mathcal{D}}\) be given by
\[\mathfrak{R}_Y = \set[(X, q)]{\text{\((X, q)\) is an S-replacement of \(Y\) along \(F\)}}\]
for \(Y \in \Ob \mathcal{D}\). The \newnotion{category of objects with S-replacement} in \(\mathcal{D}\) along \(F\) is the category with denominators~\(\CategoryOfObjectsWithSReplacement{\mathcal{D}}{F}\) whose underlying category is given by the structure category \(\mathcal{D}_{\mathfrak{R}}\) and whose set of denominators is given by
\[\Denominators \CategoryOfObjectsWithSReplacement{\mathcal{D}}{F} = \set[e \in \Mor \CategoryOfObjectsWithSReplacement{\mathcal{D}}{F}]{\text{\(\ForgetfulFunctor e\) is a denominator in \(\mathcal{D}\)}}.\]
An object in \(\CategoryOfObjectsWithSReplacement{\mathcal{D}}{F}\) is called an \newnotion{object with S-replacement} in \(\mathcal{D}\) along \(F\). A morphism in \(\CategoryOfObjectsWithSReplacement{\mathcal{D}}{F}\) is called a \newnotion{morphism of objects with S-replacement} in \(\mathcal{D}\) along \(F\). A denominator in \(\CategoryOfObjectsWithSReplacement{\mathcal{D}}{F}\) is called a \newnotion{denominator of objects with S-replacement} in \(\mathcal{D}\) along \(F\).
\end{definition}

\begin{remark} \label{rem:description_of_category_of_objects_with_s-replacement}
We have
\[\Ob{\CategoryOfObjectsWithSReplacement{\mathcal{D}}{F}} = \set[(Y, X, q)]{\text{\(Y \in \Ob \mathcal{D}\), \((X, q)\) is an S-replacement of \(Y\) along \(F\)}}.\]
For objects \((Y, X, q)\) and \((Y', X', q')\) in \(\CategoryOfObjectsWithSReplacement{\mathcal{D}}{F}\), we have the hom-set
\[{_{\CategoryOfObjectsWithSReplacement{\mathcal{D}}{F}}}((Y, X, q), (Y', X', q')) = {_{\mathcal{D}}}(Y, Y').\]
For morphisms \(g\colon (Y, X, q) \map (Y', X', q')\) and \(g'\colon (Y', X', q') \map (Y'', X'', q'')\) in \(\CategoryOfObjectsWithSReplacement{\mathcal{D}}{F}\), the compo- \linebreak 
site \(g g'\colon (Y, X, q) \map (Y'', X'', q'')\) in \(\CategoryOfObjectsWithSReplacement{\mathcal{D}}{F}\) has the underlying morphism \(g g'\colon Y \map Y''\) in \(\mathcal{D}\). For an object~\((Y, X, q)\) in \(\CategoryOfObjectsWithSReplacement{\mathcal{D}}{F}\), the identity morphism \(1_{(Y, X, q)}\colon (Y, X, q) \map (Y, X, q)\) in~\(\CategoryOfObjectsWithSReplacement{\mathcal{D}}{F}\) has the underlying morphism~\(1_Y\colon Y \map Y\) in~\(\mathcal{D}\).

The forgetful functor \(\ForgetfulFunctor\colon \CategoryOfObjectsWithSReplacement{\mathcal{D}}{F} \map \mathcal{D}\) is given on the objects by
\[\ForgetfulFunctor_{(X, q)} Y = Y\]
for \((Y, X, q) \in \Ob \CategoryOfObjectsWithSReplacement{\mathcal{D}}{F}\), and on the morphisms by
\[\ForgetfulFunctor_{(X, q), (X', q')} g = g\]
for every morphism \(g\colon (Y, X, q) \map (Y', X', q')\) in \(\CategoryOfObjectsWithSReplacement{\mathcal{D}}{F}\).
\end{remark}

\begin{remark} \label{rem:forgetful_functor_of_category_of_objects_with_s-replacement_preserves_and_reflects_denominators}
The forgetful functor \(\ForgetfulFunctor\colon \CategoryOfObjectsWithSReplacement{\mathcal{D}}{F} \map \mathcal{D}\) preserves and reflects denominators.
\end{remark}

\subsection*{Choices of S-replacements} \label{sec:s-replacements:choices_of_s-replacements}

Our construction of an isomorphism inverse on the localisation level in section~\ref{sec:s-replacement_functors_and_the_s-approximation_theorem} will use a choice of an S{\nbd}replacement for \emph{every} object of \(\mathcal{D}\). This leads us to the following notion, whose properties are just particular cases of the more general facts on choices of structures, see section~\ref{sec:preliminaries} or~\cite[app.~A, sec.~1]{thomas:2012:a_calculus_of_fractions_for_the_homotopy_category_of_a_brown_cofibration_category}.

\begin{definition}[choice of S-replacements] \label{def:choice_of_s-replacements}
We let \(\mathfrak{R} = (\mathfrak{R}_Y)_{Y \in \Ob \mathcal{D}}\) be given by
\[\mathfrak{R}_Y = \set[(X, q)]{\text{\(X \in \Ob \mathcal{C}\), \(q\colon F X \map Y\) is a denominator in \(\mathcal{D}\)}}\]
for~\(Y \in \Ob \mathcal{D}\). A \newnotion{choice of S-replacements} for \(\mathcal{D}\) along \(F\) is a choice of structures with respect to \(\mathfrak{R}\).
\end{definition}

\begin{remark} \label{rem:description_of_choice_of_s-replacements}
A choice of S-replacements for \(\mathcal{D}\) along \(F\) is a family \(((X_Y, q_Y))_{Y \in \Ob \mathcal{D}}\) such that \((X_Y, q_Y)\) is an S-replacement of \(Y\) along \(F\) for~\(Y \in \Ob \mathcal{D}\).
\end{remark}

\begin{remark} \label{rem:existence_of_choice_of_s-replacements}
There exists a choice of S-replacements for \(\mathcal{D}\) along \(F\) if and only if \(\mathcal{D}\) has enough S{\nbd}replacements along \(F\).
\end{remark}

Every choice of structures leads to a structure choice functor, see section~\ref{sec:preliminaries} or~\cite[def.~(A.8)]{thomas:2012:a_calculus_of_fractions_for_the_homotopy_category_of_a_brown_cofibration_category}. In the case of a choice of S-replacements, the structure choice functor is given as follows.

\begin{remark} \label{rem:values_of_structure_choice_functor_for_s-replacements}
We suppose given a choice of S-replacements \(R = ((X_Y, q_Y))_{Y \in \Ob \mathcal{D}}\) for \(\mathcal{D}\) along~\(F\). The structure choice functor \(\StructureChoiceFunctor{R}\colon \mathcal{D} \map \CategoryOfObjectsWithSReplacement{\mathcal{D}}{F}\) is given on the objects by
\[\StructureChoiceFunctor{R} Y = (Y, X_Y, q_Y)\]
for \(Y \in \Ob \mathcal{D}\), and on the morphisms by
\[\StructureChoiceFunctor{R} g = g\colon (Y, X_Y, q_Y) \map (Y', X_{Y'}, q_{Y'})\]
for every morphism \(g\colon Y \map Y'\) in \(\mathcal{D}\).
\end{remark}

\begin{corollary} \label{cor:structure_choice_functor_for_s-replacements_is_morphism_of_categories_with_denominators}
We suppose given a choice of S-replacements \(R = ((X_Y, q_Y))_{Y \in \Ob \mathcal{D}}\) for \(\mathcal{D}\) along~\(F\). The structure choice functor \(\StructureChoiceFunctor{R}\colon \mathcal{D} \map \CategoryOfObjectsWithSReplacement{\mathcal{D}}{F}\) is a morphism of categories with denominators.
\end{corollary}

Structure choice functors are isomorphism inverse to the forgetful functor from the structure category to the category of its underlying objects. We recall this fact in the case of a structure choice functor with respect to a choice of S-replacements and see that we obtain a pair of mutually isomorphism inverse equivalences on the localisation level:

\begin{remark} \label{rem:forgetful_functor_of_s-replacement_category_and_choices_of_s-replacements}
We suppose given a choice of S-replacements \(R = ((X_Y, q_Y))_{Y \in \Ob \mathcal{D}}\) for \(\mathcal{D}\) along~\(F\).
\begin{enumerate}
\item \label{rem:forgetful_functor_of_s-replacement_category_and_choices_of_s-replacements:retraction} We have
\[\ForgetfulFunctor \comp \StructureChoiceFunctor{R} = \id_{\mathcal{D}}.\]
\item \label{rem:forgetful_functor_of_s-replacement_category_and_choices_of_s-replacements:coretraction_up_to_isomorphism} We have
\[\StructureChoiceFunctor{R} \comp \ForgetfulFunctor \isomorphic \id_{\CategoryOfObjectsWithSReplacement{\mathcal{D}}{F}}.\]
An isotransformation \(\bar \alpha\colon \StructureChoiceFunctor{R} \comp \ForgetfulFunctor \map \id_{\CategoryOfObjectsWithSReplacement{\mathcal{D}}{F}}\) is given by
\[\bar \alpha_{(Y, X', q')} = 1_Y\colon (Y, X_Y, q_Y) \map (Y, X', q')\]
for \((Y, X', q') \in \Ob{\CategoryOfObjectsWithSReplacement{\mathcal{D}}{F}}\).
\end{enumerate}
In particular, \(\ForgetfulFunctor\colon \CategoryOfObjectsWithSReplacement{\mathcal{D}}{F} \map \mathcal{D}\) and \(\StructureChoiceFunctor{R}\colon \mathcal{D} \map \CategoryOfObjectsWithSReplacement{\mathcal{D}}{F}\) are mutually isomorphism inverse equivalences of categories.
\end{remark}
\begin{proof}
This follows from~\cite[prop.~(A.9)]{thomas:2012:a_calculus_of_fractions_for_the_homotopy_category_of_a_brown_cofibration_category}.
\end{proof}

\begin{corollary} \label{cor:forgetful_functor_of_s-replacement_category_induces_equivalence_of_categories_on_localisations}
We suppose given a choice of S-replacements \(R = ((X_Y, q_Y))_{Y \in \Ob \mathcal{D}}\) for \(\mathcal{D}\) along~\(F\).
\begin{enumerate}
\item \label{cor:forgetful_functor_of_s-replacement_category_induces_equivalence_of_categories_on_localisations:retraction} We have
\[\GabrielZismanLocalisation(\ForgetfulFunctor) \comp \GabrielZismanLocalisation(\StructureChoiceFunctor{R}) = \id_{\GabrielZismanLocalisation(\mathcal{D})}.\]
\item \label{cor:forgetful_functor_of_s-replacement_category_induces_equivalence_of_categories_on_localisations:coretraction_up_to_isomorphism} We have
\[\GabrielZismanLocalisation(\StructureChoiceFunctor{R}) \comp \GabrielZismanLocalisation(\ForgetfulFunctor) \isomorphic \id_{\GabrielZismanLocalisation(\CategoryOfObjectsWithSReplacement{\mathcal{D}}{F})}.\]
An isotransformation \(\bar \alpha\colon \GabrielZismanLocalisation(\StructureChoiceFunctor{R}) \comp \GabrielZismanLocalisation(\ForgetfulFunctor) \map \id_{\GabrielZismanLocalisation(\CategoryOfObjectsWithSReplacement{\mathcal{D}}{F})}\) is given by
\[\bar \alpha_{(Y, X', q')} = 1_Y\colon (Y, X_Y, q_Y) \map (Y, X', q')\] for \((Y, X', q') \in \Ob{\GabrielZismanLocalisation(\CategoryOfObjectsWithSReplacement{\mathcal{D}}{F})}\).
\end{enumerate}
In particular, \(\GabrielZismanLocalisation(\ForgetfulFunctor)\colon \GabrielZismanLocalisation(\CategoryOfObjectsWithSReplacement{\mathcal{D}}{F}) \map \GabrielZismanLocalisation(\mathcal{D})\) and \(\GabrielZismanLocalisation(\StructureChoiceFunctor{R})\colon \GabrielZismanLocalisation(\mathcal{D}) \map \GabrielZismanLocalisation(\CategoryOfObjectsWithSReplacement{\mathcal{D}}{F})\) are mutually isomorphism inverse equivalences of categories.
\end{corollary}
\begin{proof} \
\begin{enumerate}
\item By remark~\ref{rem:forgetful_functor_of_s-replacement_category_and_choices_of_s-replacements}\ref{rem:forgetful_functor_of_s-replacement_category_and_choices_of_s-replacements:retraction}, we have
\[\GabrielZismanLocalisation(\ForgetfulFunctor) \comp \GabrielZismanLocalisation(\StructureChoiceFunctor{R}) = \GabrielZismanLocalisation(\ForgetfulFunctor \comp \StructureChoiceFunctor{R}) = \GabrielZismanLocalisation(\id_{\mathcal{D}}) = \id_{\GabrielZismanLocalisation(\mathcal{D})}.\]
\item By remark~\ref{rem:forgetful_functor_of_s-replacement_category_and_choices_of_s-replacements}\ref{rem:forgetful_functor_of_s-replacement_category_and_choices_of_s-replacements:coretraction_up_to_isomorphism}, we have an isotransformation \(\bar \alpha'\colon \StructureChoiceFunctor{R} \comp \ForgetfulFunctor \map \id_{\CategoryOfObjectsWithSReplacement{\mathcal{D}}{F}}\) given by
\[\bar \alpha'_{(Y, X', q')} = 1_Y\colon (Y, X_Y, q_Y) \map (Y, X', q')\]
for \((Y, X', q') \in \Ob{\CategoryOfObjectsWithSReplacement{\mathcal{D}}{F}}\). But then \(\bar \alpha := \GabrielZismanLocalisation(\bar \alpha')\) is an isotransformation from \(\GabrielZismanLocalisation(\StructureChoiceFunctor{R} \comp \ForgetfulFunctor) = \GabrielZismanLocalisation(\StructureChoiceFunctor{R}) \comp \GabrielZismanLocalisation(\ForgetfulFunctor)\) to \(\GabrielZismanLocalisation(\id_{\CategoryOfObjectsWithSReplacement{\mathcal{D}}{F}}) = \id_{\GabrielZismanLocalisation(\CategoryOfObjectsWithSReplacement{\mathcal{D}}{F})}\) by \(2\)-functoriality, given by
\[\bar \alpha_{(Y, X', q')} = \LocalisationFunctor[\GabrielZismanLocalisation(\mathcal{D})](\bar \alpha'_{(Y, X', q')}) = \LocalisationFunctor[\GabrielZismanLocalisation(\mathcal{D})](1_Y) = 1_Y\colon (Y, X_Y, q_Y) \map (Y, X', q')\]
for \((Y, X', q') \in \Ob{\GabrielZismanLocalisation(\CategoryOfObjectsWithSReplacement{\mathcal{D}}{F})} = \Ob{\CategoryOfObjectsWithSReplacement{\mathcal{D}}{F}}\). \qedhere
\end{enumerate}
\end{proof}

\section{S-equivalences and the characterising conditions} \label{sec:s-equivalences_and_the_characterising_conditions}

Next, we introduce S-equivalences, that is, those morphisms of categories with denominators inducing equivalences on the localisation level that we want to characterise in this article, as well as the characterising conditions.

Throughout this section, we suppose given a morphism of categories with denominators~\(F\colon \mathcal{C} \map \mathcal{D}\).

\subsection*{S-density} \label{sec:s-equivalences_and_the_characterising_conditions:s-density}

We begin with the restriction we put on \(F\colon \mathcal{C} \map \mathcal{D}\) that ensures that \(\GabrielZismanLocalisation(F)\colon \GabrielZismanLocalisation(\mathcal{C}) \map \GabrielZismanLocalisation(\mathcal{D})\) is dense.

\begin{definition}[S-dense] \label{def:s-density}
We say that~\(F\) is \newnotion{S-dense} if \(\mathcal{D}\) has enough S-replacements along \(F\).
\end{definition}

So \(F\) is S-dense if and only if for every object \(Y\) in \(\mathcal{D}\) there exists an S-replacement of \(Y\) along \(F\).

\begin{remark} \label{rem:s-density_implies_density_on_localisations}
If \(F\) is S-dense, then \(\GabrielZismanLocalisation(F)\colon \GabrielZismanLocalisation(\mathcal{C}) \map \GabrielZismanLocalisation(\mathcal{D})\) is dense.
\end{remark}
\begin{proof}
We suppose that \(F\) is S-dense and we suppose given an object \(Y\) in \(\mathcal{D}\). Then there exists an S{\nbd}replacement~\((X, q)\) of \(Y\) along \(F\). As \(q\colon F X \map Y\) is a denominator in \(\mathcal{D}\), it follows that \(\LocalisationFunctor(q)\colon F X \map Y\) is an isomorphism in \(\GabrielZismanLocalisation(\mathcal{D})\), and so we have 
\[Y \isomorphic F X = \GabrielZismanLocalisation(F) X\]
in \(\GabrielZismanLocalisation(\mathcal{D})\). Thus \(\GabrielZismanLocalisation(F)\) is dense.
\end{proof}

We will give a characterisation of S-density (under additional assumptions on the degree of saturatedness of \(\mathcal{D}\)) via the forgetful functor \(\ForgetfulFunctor\colon \CategoryOfObjectsWithSReplacement{\mathcal{D}}{F} \map \mathcal{D}\) in proposition~\ref{prop:characterisation_of_s-replacement_axiom}.

\subsection*{S-equivalences} \label{sec:s-equivalences_and_the_characterising_conditions:s-equivalences}

The primary objective of this article is to characterise when \(F\) is an S-equivalence in the following sense.

\begin{definition}[S-equivalence] \label{def:s-equivalence}
The morphism of categories with denominators~\(F\) is called an \newnotion{S-equivalence} if it is S-dense and \(\GabrielZismanLocalisation(F)\colon \GabrielZismanLocalisation(\mathcal{C}) \map \GabrielZismanLocalisation(\mathcal{D})\) is an equivalence.
\end{definition}

The characterisation of S-equivalences will be given in corollary~\ref{cor:characterisation_of_s-equivalences_as_s-dense_s-full_and_s-faithful_morphisms_of_categories_with_denominators}, which is based on the S-approximation theorem~\ref{th:s-approximation_theorem}, where an isomorphism inverse to \(\GabrielZismanLocalisation(F)\colon \GabrielZismanLocalisation(\mathcal{C}) \map \GabrielZismanLocalisation(\mathcal{D})\) is constructed.

\subsection*{S-fullness and S-faithfulness} \label{sec:s-equivalences_and_the_characterising_conditions:s-fullness_and_s-faithfulness}

While S-density is already part of the definition of an S-equivalence, we will now introduce the remaining characterising conditions -- S-fullness and S-faithfulness.

\begin{definition}[S-fullness] \label{def:s-fullness}
We say that~\(F\) is \newnotion{S-full} if for all objects \(X\) and \(X'\) in \(\mathcal{C}\) and every S{\nbd}\(2\){\nbd}arrow~\((g, b)\colon \Stwoarrow{F X}{\tilde Y'}{F X'}\) in \(\mathcal{D}\) there exists a morphism \(\varphi\colon X \map X'\) in \(\GabrielZismanLocalisation(\mathcal{C})\) such that
\[\LocalisationFunctor(g) \, \LocalisationFunctor(b)^{- 1} = \GabrielZismanLocalisation(F) \varphi\]
in \(\GabrielZismanLocalisation(\mathcal{D})\).
\end{definition}

So roughly said, S-fullness of \(F\) means {``}fullness of \(\GabrielZismanLocalisation(F)\) on S-\(2\)-arrows{''}.

\begin{remark} \label{rem:s-fullness_implies_fullness_on_localisations}
If \(\GabrielZismanLocalisation(F)\colon \GabrielZismanLocalisation(\mathcal{C}) \map \GabrielZismanLocalisation(\mathcal{D})\) is full, then \(F\) is S-full.
\end{remark}

\begin{proposition} \label{prop:s-density_and_s-fullness_implies_fullness_of_induced_functor}
We suppose that \(\mathcal{D}\) is multiplicative and that \(F\) is S-dense. Then \(F\) is S-full if and only if~\(\GabrielZismanLocalisation(F)\colon \GabrielZismanLocalisation(\mathcal{C}) \map \GabrielZismanLocalisation(\mathcal{D})\) is full.
\end{proposition}
\begin{proof}
If \(\GabrielZismanLocalisation(F)\) is full, then in particular \(F\) is S-full. Conversely, we suppose that \(F\) is S-full. To show that~\(\GabrielZismanLocalisation(F)\) is full, we suppose given objects~\(X\) and \(X'\) in \(\mathcal{C}\) and a morphism \(\psi\colon F X \map F X'\) in \(\GabrielZismanLocalisation(\mathcal{D})\). Moreover, we choose~\(n \in \Naturals\), morphisms \(g_i\colon Y_{i - 1} \map \tilde Y_i\) in \(\mathcal{D}\) for~\(i \in [1, n]\) and denominators \(b_i\colon Y_i \map \tilde Y_i\) in \(\mathcal{D}\) for~\(i \in [1, n - 1]\) such that \(F X = Y_0\), \(F X' = \tilde Y_n\) and
\[\psi = \LocalisationFunctor(g_1) \, \LocalisationFunctor(b_1)^{- 1} \, \LocalisationFunctor(g_2) \, \dots \, \LocalisationFunctor(b_{n - 1})^{- 1} \, \LocalisationFunctor(g_n).\]
Since \(\mathcal{D}\) is multiplicative, the identity \(1_{F X'}\colon F X' \map F X'\) is a denominator in \(\mathcal{D}\). We set \(Y_n := F X'\) and~\(b_n := 1_{F X'}\). Moreover, since \(F\) is S-dense, for \(i \in [1, n - 1]\) there exists an S-replacement \((X_i, q_i)\) of~\(Y_i\). We set \((X_0, q_0) := (X, 1_{F X})\) and \((X_n, q_n) := (X', 1_{F X'})\). Then \(q_i b_i\) is a denominator in \(\mathcal{D}\) for~\(i \in [1, n]\) by multiplicativity.
\[\begin{tikzpicture}[baseline=(m-2-1.base)]
  \matrix (m) [diagram]{
    F X & \tilde Y_1 & F X_1 & \dots & F X_{n - 1} & F X' & F X' \\
    F X & \tilde Y_1 & Y_1 & \dots & Y_{n - 1} & F X' & F X' \\ };
  \path[->, font=\scriptsize]
    (m-1-1) edge node[above] {\(g_1\)} (m-1-2)
            edge[den] node[right] {\(1_{F X}\)} (m-2-1)
    (m-1-2) edge[equality] ([yshift=2pt]m-2-2.north)
    (m-1-3) edge node[above] {\(q_1 g_2\)} (m-1-4)
            edge[den] node[right] {\(q_1\)} (m-2-3)
            edge[den] node[above] {\(q_1 b_1\)} (m-1-2)
    (m-1-5) edge node[above] {\(q_{n - 1} g_n\)} (m-1-6)
            edge[den] node[right] {\(q_{n - 1}\)} (m-2-5)
            edge[den] node[above] {\(q_{n - 1} b_{n - 1}\)} (m-1-4)
    (m-1-6) edge[equality] (m-2-6)
    (m-1-7) edge[den] node[above] {\(1_{F X'}\)} (m-1-6)
            edge[den] node[right] {\(1_{F X'}\)} (m-2-7)
    (m-2-1) edge node[above] {\(g_1\)} (m-2-2)
    (m-2-3) edge node[above] {\(g_2\)} (m-2-4)
            edge[den] node[above] {\(b_1\)} (m-2-2)
    (m-2-5) edge node[above] {\(g_n\)} (m-2-6)
            edge[den] node[above] {\(b_{n - 1}\)} (m-2-4)
    (m-2-7) edge[den] node[above] {\(1_{F X'}\)} (m-2-6);
\end{tikzpicture}\]
Now the S-fullness of \(F\) implies that for \(i \in [1, n]\) there exists a morphism \(\varphi_i\colon X_{i - 1} \map X_i\) in \(\GabrielZismanLocalisation(\mathcal{C})\) such that \(\LocalisationFunctor(q_{i - 1} g_i) \, \LocalisationFunctor(q_i b_i)^{- 1} = \GabrielZismanLocalisation(F) \varphi_i\). We obtain
\begin{align*}
\psi & = \LocalisationFunctor(g_1) \, \LocalisationFunctor(b_1)^{- 1} \, \LocalisationFunctor(g_2) \, \dots \, \LocalisationFunctor(b_{n - 1})^{- 1} \, \LocalisationFunctor(g_n) \\
& = \LocalisationFunctor(q_0 g_1) \, \LocalisationFunctor(q_1 b_1)^{- 1} \, \LocalisationFunctor(q_1 g_2) \, \dots \, \LocalisationFunctor(q_{n - 1} b_{n - 1})^{- 1} \, \LocalisationFunctor(q_{n - 1} g_n) \, \LocalisationFunctor(q_{n} b_n)^{- 1} \\
& = (\GabrielZismanLocalisation(F) \varphi_1) (\GabrielZismanLocalisation(F) \varphi_2) \dots (\GabrielZismanLocalisation(F) \varphi_n) = \GabrielZismanLocalisation(F)(\varphi_1 \varphi_2 \dots \varphi_n).
\end{align*}
Thus \(\GabrielZismanLocalisation(F)\) is full.
\end{proof}

\begin{definition}[S-faithfulness] \label{def:s-faithfulness}
We say that~\(F\) is \newnotion{S-faithful} if for all objects \(X\) and \(X'\) in \(\mathcal{C}\), every S{\nbd}\(2\){\nbd}arrow~\((g, b)\colon \Stwoarrow{F X}{\tilde Y'}{F X'}\) in \(\mathcal{D}\) and all morphisms \(\varphi_1, \varphi_2\colon X \map X'\) in \(\GabrielZismanLocalisation(\mathcal{C})\) such that
\[\LocalisationFunctor(g) \, \LocalisationFunctor(b)^{- 1} = \GabrielZismanLocalisation(F) \varphi_1 = \GabrielZismanLocalisation(F) \varphi_2\]
in \(\GabrielZismanLocalisation(\mathcal{D})\), we have
\[\varphi_1 = \varphi_2\]
in \(\GabrielZismanLocalisation(\mathcal{C})\).
\end{definition}

So roughly said, S-faithfulness of \(F\) means {``}faithfulness of \(\GabrielZismanLocalisation(F)\) on S-\(2\)-arrows{''}.

\begin{remark} \label{rem:faithfulness_of_induced_functor_implies_s-faithfulness}
If \(\GabrielZismanLocalisation(F)\colon \GabrielZismanLocalisation(\mathcal{C}) \map \GabrielZismanLocalisation(\mathcal{D})\) is faithful, then \(F\colon \mathcal{C} \map \mathcal{D}\) is S-faithful.
\end{remark}

Under the (mild) additional assumption that \(\mathcal{D}\) is multiplicative we will show that \(F\) is an S-equivalence if and only if it is S-dense, S-full and S-faithful, see corollary~\ref{cor:characterisation_of_s-equivalences_as_s-dense_s-full_and_s-faithful_morphisms_of_categories_with_denominators}.

\section{S-replacement functors and the S-approximation theorem} \label{sec:s-replacement_functors_and_the_s-approximation_theorem}

We suppose given a morphism of categories with denominators \(F\colon \mathcal{C} \map \mathcal{D}\). The aim of this section is the construction of an isomorphism inverse to \(\GabrielZismanLocalisation(F)\colon \GabrielZismanLocalisation(\mathcal{C}) \map \GabrielZismanLocalisation(\mathcal{D})\), provided that \(\mathcal{D}\) is multiplicative and \(F\) fulfils the conditions of S-density, S-fullness and S-faithfulness defined in the previous section.

We give a sketch of this construction: First, we show that \(F\) lifts to the category of objects with S-replacement in \(\mathcal{D}\) along \(F\), see remark~\ref{rem:well-definedness_of_canonical_lift_along_forgetful_functor_of_s-replacement_category}, that is, we show that there exists a morphism of categories with denominators~\(\bar F\colon \mathcal{C} \map \CategoryOfObjectsWithSReplacement{\mathcal{D}}{F}\) such that the following triangle on the left commutes. By the functoriality of the Gabriel/Zisman localisation, this commutative triangle on the left induces the following commutative triangle on the right.
\[\begin{tikzpicture}[baseline=(m-2-1.base)]
  \matrix (m) [diagram]{
    & \CategoryOfObjectsWithSReplacement{\mathcal{D}}{F} & \\
    \mathcal{C} & & \mathcal{D} \\ };
  \path[->, font=\scriptsize]
    (m-2-1) edge node[above] {\(F\)} (m-2-3)
            edge node[above left=-2pt] {\(\bar F\)} (m-1-2)
    (m-1-2) edge node[above right=-2pt] {\(\ForgetfulFunctor\)} (m-2-3);
\end{tikzpicture}
\qquad
\begin{tikzpicture}[baseline=(m-2-1.base)]
  \matrix (m) [diagram]{
    & \GabrielZismanLocalisation(\CategoryOfObjectsWithSReplacement{\mathcal{D}}{F}) & \\
    \GabrielZismanLocalisation(\mathcal{C}) & & \GabrielZismanLocalisation(\mathcal{D}) \\ };
  \path[->, font=\scriptsize]
    (m-2-1) edge node[above] {\(\GabrielZismanLocalisation(F)\)} (m-2-3)
            edge node[above left=-2pt] {\(\GabrielZismanLocalisation(\bar F)\)} (m-1-2)
    (m-1-2) edge node[above right=-2pt] {\(\GabrielZismanLocalisation(\ForgetfulFunctor)\)} (m-2-3);
\end{tikzpicture}\]

By remark~\ref{rem:forgetful_functor_of_s-replacement_category_and_choices_of_s-replacements} we already know that the forgetful functor \(\ForgetfulFunctor\colon \CategoryOfObjectsWithSReplacement{\mathcal{D}}{F} \map \mathcal{D}\) is an equivalence of categories if \(F\colon \mathcal{C} \map \mathcal{D}\) is~S{\nbd}dense, where an isomorphism inverse \(\StructureChoiceFunctor{R}\colon \mathcal{D} \map \CategoryOfObjectsWithSReplacement{\mathcal{D}}{F}\) is constructed by a choice of an S{\nbd}replacement for each object in \(\mathcal{D}\), see definition~\ref{def:choice_of_s-replacements} and remark~\ref{rem:forgetful_functor_of_s-replacement_category_and_choices_of_s-replacements}. This pair of mutually inverse equivalences induces a pair of mutually inverse equivalences on the localisation level, see corollary~\ref{cor:forgetful_functor_of_s-replacement_category_induces_equivalence_of_categories_on_localisations}.

So in order to show that the functor \(\GabrielZismanLocalisation(F)\colon \GabrielZismanLocalisation(\mathcal{C}) \map \GabrielZismanLocalisation(\mathcal{D})\) is an equivalence of categories, it suffices to show that~\(\GabrielZismanLocalisation(\bar F)\colon \GabrielZismanLocalisation(\CategoryOfObjectsWithSReplacement{\mathcal{D}}{F}) \map \GabrielZismanLocalisation(\mathcal{D})\) is an equivalence of categories. To this end, we construct the so-called total S-replacement functor \(\TotalSReplacementFunctor F\colon \CategoryOfObjectsWithSReplacement{\mathcal{D}}{F} \map \GabrielZismanLocalisation(\mathcal{C})\), see proposition~\ref{prop:well-definedness_of_total_s-replacement_functor}, which induces an isomorphism inverse \(\InducedFunctorOfTotalSReplacementFunctor F\colon \GabrielZismanLocalisation(\CategoryOfObjectsWithSReplacement{\mathcal{D}}{F}) \map \GabrielZismanLocalisation(\mathcal{C})\) to~\(\GabrielZismanLocalisation(\bar F)\colon \GabrielZismanLocalisation(\mathcal{C}) \map \GabrielZismanLocalisation(\CategoryOfObjectsWithSReplacement{\mathcal{D}}{F})\), see corollary~\ref{cor:total_s-replacement_functor_maps_denominators_to_isomorphisms} and corollary~\ref{cor:canonical_lift_along_forgetful_functor_of_s-replacement_category_induces_equivalence_on_localisations}.

\[\begin{tikzpicture}[baseline=(m-2-1.base)]
  \matrix (m) [diagram=4.0em]{
    \mathcal{C} & \CategoryOfObjectsWithSReplacement{\mathcal{D}}{F} & \mathcal{D} \\
    \GabrielZismanLocalisation(\mathcal{C}) & \GabrielZismanLocalisation(\CategoryOfObjectsWithSReplacement{\mathcal{D}}{F}) & \GabrielZismanLocalisation(\mathcal{D}) \\ };
  \path[->, font=\scriptsize]
    (m-1-1) edge node[above] {\(\bar F\)} (m-1-2)
            edge node[right] {\(\LocalisationFunctor\)} (m-2-1)
    (m-1-2) edge node[above] {\(\ForgetfulFunctor\)} node[below=1pt] {\(\categoricallyequivalent\)} (m-1-3)
            edge node[right] {\(\LocalisationFunctor\)} (m-2-2)
            edge node[above left=-3pt] {\(\TotalSReplacementFunctor F\)} (m-2-1)
    (m-1-3) edge node[right] {\(\LocalisationFunctor\)} (m-2-3)
    (m-1-3.south west) edge[bend left=15] node[below] {\(\StructureChoiceFunctor{R}\)} (m-1-2.south east)
    (m-2-1) edge node[above] {\(\GabrielZismanLocalisation(\bar F)\)} node[below] {\(\categoricallyequivalent\)} (m-2-2)
    (m-2-2) edge node[above] {\(\GabrielZismanLocalisation(\ForgetfulFunctor)\)} node[below] {\(\categoricallyequivalent\)} (m-2-3)
    (m-2-2.south west) edge[bend left=15] node[below] {\(\InducedFunctorOfTotalSReplacementFunctor F\)} (m-2-1.south east)
    (m-2-3.south west) edge[bend left=15] node[below] {\(\GabrielZismanLocalisation(\StructureChoiceFunctor{R})\)} (m-2-2.south east);
\end{tikzpicture}\]

The proof of the S-approximation theorem~\ref{th:s-approximation_theorem} is concluded by showing that an isomorphism inverse \linebreak 
to~\(\GabrielZismanLocalisation(F)\colon \GabrielZismanLocalisation(\mathcal{C}) \map \GabrielZismanLocalisation(\mathcal{D})\) can be induced by a so-called S-replacement functor \(\SReplacementFunctor F\colon \mathcal{D} \map \GabrielZismanLocalisation(\mathcal{C})\) that is given as composite \(\SReplacementFunctor F = \TotalSReplacementFunctor F \comp \StructureChoiceFunctor{R}\), see definition~\ref{def:s-replacement_functor}.

As the structure choice functor \(\StructureChoiceFunctor{R}\colon \mathcal{D} \map \CategoryOfObjectsWithSReplacement{\mathcal{D}}{F}\) depends on a choice of S-replacements, this also holds for the S-replacement functor \(\SReplacementFunctor F = \TotalSReplacementFunctor F \comp \StructureChoiceFunctor{R}\). Thus the total S-replacement functor \(\TotalSReplacementFunctor F\) may be seen as a uniform variant of the various possible isomorphism inverse inducing S-replacement functors, which do necessitate choices.

Throughout this section, we suppose given a morphism of categories with denominators~\(F\colon \mathcal{C} \map \mathcal{D}\).

\subsection*{The canonical lift} \label{sec:s-replacement_functors_and_the_s-approximation_theorem:the_canonical_lift}

Under the assumption that \(\mathcal{D}\) has all trivial S-replacements along \(F\), we may lift \(F\) to the corresponding category of objects with S-replacement:

\begin{remark} \label{rem:well-definedness_of_canonical_lift_along_forgetful_functor_of_s-replacement_category}
We suppose that \(\mathcal{D}\) has all trivial S-replacements along \(F\).
\begin{enumerate}
\item \label{rem:well-definedness_of_canonical_lift_along_forgetful_functor_of_s-replacement_category:functoriality} We have a functor \(\bar F\colon \mathcal{C} \map \CategoryOfObjectsWithSReplacement{\mathcal{D}}{F}\), given on the objects by
\[\bar F X = (F X, X, 1_{F X})\]
for \(X \in \Ob \mathcal{C}\), and on the morphisms by
\[\bar F f = F f\colon (F X, X, 1_{F X}) \map (F X', X', 1_{F X'})\]
for every morphism \(f\colon X \map X'\) in \(\mathcal{C}\).
\[\begin{tikzpicture}[baseline=(m-2-1.base)]
  \matrix (m) [diagram]{
    F X \\
    F X \\ };
  \path[->, font=\scriptsize]
    (m-1-1) edge[den] node[right] {\(1_{F X}\)} (m-2-1);
\end{tikzpicture}
\qquad
\begin{tikzpicture}[baseline=(m-2-1.base)]
  \matrix (m) [diagram]{
    F X & F X' \\
    F X & F X' \\ };
  \path[->, font=\scriptsize]
    (m-1-1) edge node[above] {\(F f\)} (m-1-2)
            edge[den] node[left] {\(1_{F X}\)} (m-2-1)
    (m-1-2) edge[den] node[right] {\(1_{F X'}\)} (m-2-2)
    (m-2-1) edge node[above] {\(F f\)} (m-2-2);
\end{tikzpicture}\]
\item \label{rem:well-definedness_of_canonical_lift_along_forgetful_functor_of_s-replacement_category:factorisation} We have
\[F = \ForgetfulFunctor \comp \bar F.\]
\[\begin{tikzpicture}[baseline=(m-2-1.base)]
  \matrix (m) [diagram]{
    & \CategoryOfObjectsWithSReplacement{\mathcal{D}}{F} & \\
    \mathcal{C} & & \mathcal{D} \\ };
  \path[->, font=\scriptsize]
    (m-2-1) edge node[above] {\(F\)} (m-2-3)
            edge node[above left=-2pt] {\(\bar F\)} (m-1-2)
    (m-1-2) edge node[above right=-2pt] {\(\ForgetfulFunctor\)} (m-2-3);
\end{tikzpicture}\]
\end{enumerate}
\end{remark}

\begin{definition}[canonical lift] \label{def:canonical_lift_along_forgetful_functor_of_s-replacement_category}
We suppose that \(\mathcal{D}\) has all trivial S-replacements along \(F\). The functor~\(\bar F\colon \mathcal{C} \map \CategoryOfObjectsWithSReplacement{\mathcal{D}}{F}\) in remark~\ref{rem:well-definedness_of_canonical_lift_along_forgetful_functor_of_s-replacement_category} is called the \newnotion{canonical lift} of~\(F\) along \(\ForgetfulFunctor\colon \CategoryOfObjectsWithSReplacement{\mathcal{D}}{F} \map \mathcal{D}\).
\end{definition}

In fact, if \(F\) is S-dense, then one can construct several (non-canonical) lifts along the forgetful func- \linebreak 
tor~\(\ForgetfulFunctor\colon \CategoryOfObjectsWithSReplacement{\mathcal{D}}{F} \map \mathcal{D}\): Every choice of S-replacements \(R = ((X_Y, q_Y))_{Y \in \Ob \mathcal{D}}\) for \(\mathcal{D}\) along~\(F\) leads to a lift \(F_R := \StructureChoiceFunctor{R} \comp F\colon \mathcal{C} \map \CategoryOfObjectsWithSReplacement{\mathcal{D}}{F}\) as \(\ForgetfulFunctor \comp F_R = \ForgetfulFunctor \comp \StructureChoiceFunctor{R} \comp F = F\) by remark~\ref{rem:forgetful_functor_of_s-replacement_category_and_choices_of_s-replacements}. However, to prove an assertion analogous to proposition~\ref{prop:canonical_lift_and_total_s-replacement_functor}\ref{prop:canonical_lift_and_total_s-replacement_functor:coretraction} below, it seems that one still (at least implicitly) needs the canonical lift~\(\bar F\colon \mathcal{C} \map \CategoryOfObjectsWithSReplacement{\mathcal{D}}{F}\) for the construction of an isotransformation \(\TotalSReplacementFunctor F \comp F_R \map \LocalisationFunctor[\GabrielZismanLocalisation(\mathcal{C})]\), where \linebreak 
\(\TotalSReplacementFunctor F\colon \CategoryOfObjectsWithSReplacement{\mathcal{D}}{F} \map \GabrielZismanLocalisation(\mathcal{C})\) denotes the total S-replacement functor introduced in definition~\ref{def:total_s-replacement_functor} below.

\begin{remark} \label{rem:canonical_lift_along_forgetful_functor_of_s-replacement_category_is_s-dense}
We suppose that \(\mathcal{D}\) has all trivial S-replacements along \(F\). For every object~\((Y, X, q)\) in~\(\CategoryOfObjectsWithSReplacement{\mathcal{D}}{F}\), the pair \((X, (q, (X, 1_{F X}), (X, q)))\) is an S-replacement of~\((Y, X, q)\) along the canonical lift \(\bar F\colon \mathcal{C} \map \CategoryOfObjectsWithSReplacement{\mathcal{D}}{F}\).
\[\begin{tikzpicture}[baseline=(m-1-1.base)]
  \matrix (m) [diagram]{
    (F X, X, 1_{F X}) & (Y, X, q) \\ };
  \path[->, font=\scriptsize]
    (m-1-1) edge[den] node[above] {\(q\)} (m-1-2);
\end{tikzpicture}\]
In particular, the category with denominators \(\CategoryOfObjectsWithSReplacement{\mathcal{D}}{F}\) has enough S{\nbd}replacements along~\(\bar F\).
\end{remark}

\begin{proposition} \label{prop:characterisation_of_s-replacement_axiom}
We suppose that \(\mathcal{D}\) is multiplicative. The following conditions are equivalent.
\begin{enumerate}
\item \label{prop:characterisation_of_s-replacement_axiom:definition} The morphism of categories with denominators \(F\) is S-dense.
\item \label{prop:characterisation_of_s-replacement_axiom:forgetful_functor_is_s-dense} The forgetful functor \(\ForgetfulFunctor\colon \CategoryOfObjectsWithSReplacement{\mathcal{D}}{F} \map \mathcal{D}\) is S-dense.
\item \label{prop:characterisation_of_s-replacement_axiom:forgetful_functor_is_surjective_on_objects} The forgetful functor \(\ForgetfulFunctor\colon \CategoryOfObjectsWithSReplacement{\mathcal{D}}{F} \map \mathcal{D}\) is surjective on the objects.
\end{enumerate}
If \(\mathcal{D}\) is isosaturated, then these conditions are also equivalent to the following conditions.
\begin{enumerate}
\setcounter{enumi}{3}
\item \label{prop:characterisation_of_s-replacement_axiom:forgetful_functor_is_dense} The forgetful functor \(\ForgetfulFunctor\colon \CategoryOfObjectsWithSReplacement{\mathcal{D}}{F} \map \mathcal{D}\) is dense.
\item \label{prop:characterisation_of_s-replacement_axiom:forgetful_functor_is_equivalence} The forgetful functor \(\ForgetfulFunctor\colon \CategoryOfObjectsWithSReplacement{\mathcal{D}}{F} \map \mathcal{D}\) is an equivalence of categories.
\end{enumerate}
\end{proposition}
\begin{proof}
First, we show that condition~\ref{prop:characterisation_of_s-replacement_axiom:definition}, condition~\ref{prop:characterisation_of_s-replacement_axiom:forgetful_functor_is_s-dense} and condition~\ref{prop:characterisation_of_s-replacement_axiom:forgetful_functor_is_surjective_on_objects} are equivalent.

By remark~\ref{rem:well-definedness_of_canonical_lift_along_forgetful_functor_of_s-replacement_category}\ref{rem:well-definedness_of_canonical_lift_along_forgetful_functor_of_s-replacement_category:factorisation}, we have \(F = \ForgetfulFunctor \comp \bar F\), where \(\bar F\colon \mathcal{C} \map \CategoryOfObjectsWithSReplacement{\mathcal{D}}{F}\) denotes the canonical lift of \(F\) along~\(\ForgetfulFunctor\). The canonical lift \(\bar F\) is S-dense by remark~\ref{rem:canonical_lift_along_forgetful_functor_of_s-replacement_category_is_s-dense}. So as \(\mathcal{D}\) is multiplicative, remark~\ref{rem:having_enough_s-replacements_and_composition_of_functors} implies that \(F\) is S-dense if and only if \(\ForgetfulFunctor\) is S-dense, that is, condition~\ref{prop:characterisation_of_s-replacement_axiom:definition} and condition~\ref{prop:characterisation_of_s-replacement_axiom:forgetful_functor_is_s-dense} are equivalent.

Moreover, for every object \(Y\) in \(\mathcal{D}\) there exists an S-replacement \((X, q)\) of \(Y\) along \(F\) if and only if there exists an object with S-replacement \((Y, X, q)\) in \(\mathcal{D}\) along \(F\). Thus \(F\) is S-dense if and only if \(\ForgetfulFunctor\) is surjective on the objects, that is, condition~\ref{prop:characterisation_of_s-replacement_axiom:definition} and condition~\ref{prop:characterisation_of_s-replacement_axiom:forgetful_functor_is_surjective_on_objects} are equivalent.

Thus condition~\ref{prop:characterisation_of_s-replacement_axiom:definition}, condition~\ref{prop:characterisation_of_s-replacement_axiom:forgetful_functor_is_s-dense} and condition~\ref{prop:characterisation_of_s-replacement_axiom:forgetful_functor_is_surjective_on_objects} are equivalent.

Second, we suppose that \(\mathcal{D}\) is isosaturated and show that this premise implies that all five conditions are equivalent.

We suppose that condition~\ref{prop:characterisation_of_s-replacement_axiom:definition} holds, that is, we suppose that \(F\) is S-dense. Then there exists a choice of S{\nbd}replacements \(R\) for \(\mathcal{D}\) along \(F\), and \(\ForgetfulFunctor\) and \(\StructureChoiceFunctor{R}\) are mutually isomorphism inverse equivalences by remark~\ref{rem:forgetful_functor_of_s-replacement_category_and_choices_of_s-replacements}. Thus condition~\ref{prop:characterisation_of_s-replacement_axiom:forgetful_functor_is_equivalence} holds.

If condition~\ref{prop:characterisation_of_s-replacement_axiom:forgetful_functor_is_equivalence} holds, that is, if \(\ForgetfulFunctor\) is an equivalence, then in particular \(\ForgetfulFunctor\) is dense by the dense-full-faithful criterion, that is, condition~\ref{prop:characterisation_of_s-replacement_axiom:forgetful_functor_is_dense} holds.

Finally, we suppose~that condition~\ref{prop:characterisation_of_s-replacement_axiom:forgetful_functor_is_dense} holds, that is, we suppose that \(\ForgetfulFunctor\colon \CategoryOfObjectsWithSReplacement{\mathcal{D}}{F} \map \mathcal{D}\) is dense. Moreover, we suppose given an object~\(Y\) in~\(\mathcal{D}\). As~\(\ForgetfulFunctor\) is dense, there exists an object \((Y', X, q)\) in \(\CategoryOfObjectsWithSReplacement{\mathcal{D}}{F}\) and an isomorphism~\(g\colon \ForgetfulFunctor (Y', X, q) \map Y\) in \(\mathcal{D}\). The isosaturatedness of~\(\mathcal{D}\) implies that \(g\colon Y' \map Y\) is a denominator in~\(\mathcal{D}\), and so \(((Y', X, q), g)\) is an S-replacement of \(Y\) along \(\ForgetfulFunctor\). Thus \(\ForgetfulFunctor\colon \CategoryOfObjectsWithSReplacement{\mathcal{D}}{F} \map \mathcal{D}\) is S-dense, that is, condition~\ref{prop:characterisation_of_s-replacement_axiom:forgetful_functor_is_s-dense} holds.

Thus condition~\ref{prop:characterisation_of_s-replacement_axiom:definition}, condition~\ref{prop:characterisation_of_s-replacement_axiom:forgetful_functor_is_s-dense}, condition~\ref{prop:characterisation_of_s-replacement_axiom:forgetful_functor_is_surjective_on_objects}, condition~\ref{prop:characterisation_of_s-replacement_axiom:forgetful_functor_is_dense} and condition~\ref{prop:characterisation_of_s-replacement_axiom:forgetful_functor_is_equivalence} are equivalent.
\end{proof}

\subsection*{The total S-replacement functor} \label{sec:s-replacement_functors_and_the_s-approximation_theorem:the_total_s-replacement_functor}

Next, we construct a functor that leads to an isomorphism inverse of the canonical lift.

\begin{proposition} \label{prop:well-definedness_of_total_s-replacement_functor}
We suppose that \(F\) is S-full and S-faithful. Then we have a functor
\[\TotalSReplacementFunctor F\colon \CategoryOfObjectsWithSReplacement{\mathcal{D}}{F} \map \GabrielZismanLocalisation(\mathcal{C}),\]
given on the objects by
\[(\TotalSReplacementFunctor F)_{(X, q)} Y = X\]
for \((Y, X, q) \in \Ob \CategoryOfObjectsWithSReplacement{\mathcal{D}}{F}\), and on the morphisms as follows. Given a morphism \(g\colon (Y, X, q) \map (Y', X', q')\) in~\(\CategoryOfObjectsWithSReplacement{\mathcal{D}}{F}\), then \((\TotalSReplacementFunctor F)_{(X, q), (X', q')} g\colon X \map X'\) is the unique morphism in \(\GabrielZismanLocalisation(\mathcal{C})\) with
\[\LocalisationFunctor(q) \, \LocalisationFunctor(g) = (\GabrielZismanLocalisation(F) (\TotalSReplacementFunctor F)_{(X, q), (X', q')} g) \, \LocalisationFunctor(q')\]
in \(\GabrielZismanLocalisation(\mathcal{D})\).
\[\begin{tikzpicture}[baseline=(m-2-1.base)]
  \matrix (m) [diagram]{
    F X &[7.5em] F X' \\
    Y & Y' \\ };
  \path[->, font=\scriptsize]
    (m-1-1) edge node[above] {\(\GabrielZismanLocalisation(F) (\TotalSReplacementFunctor F)_{(X, q), (X', q')} g\)} (m-1-2)
            edge node[left] {\(\LocalisationFunctor(q)\)} node[sloped, above] {\(\isomorphic\)} (m-2-1)
    (m-1-2) edge node[right] {\(\LocalisationFunctor(q')\)} node[sloped, below] {\(\isomorphic\)} (m-2-2)
    (m-2-1) edge node[above] {\(\LocalisationFunctor(g)\)} (m-2-2);
\end{tikzpicture}\]
\end{proposition}
\begin{proof}
We define a map
\[\text{\(\bar Q_0\colon \Ob \CategoryOfObjectsWithSReplacement{\mathcal{D}}{F} \map \Ob \GabrielZismanLocalisation(\mathcal{C})\), \((Y, X, q) \mapsto X\).}\]
Moreover, as \(F\colon \mathcal{C} \map \mathcal{D}\) is S-full and S-faithful, for all \((Y, X, q), (Y', X', q') \in \Ob \CategoryOfObjectsWithSReplacement{\mathcal{D}}{F}\) we obtain a well-defined map 
\[\bar Q_{(Y, X, q), (Y', X', q')}\colon {_{\CategoryOfObjectsWithSReplacement{\mathcal{D}}{F}}}((Y, X, q), (Y', X', q')) \map {_{\GabrielZismanLocalisation(\mathcal{C})}}(X, X'),\]
where \(\bar Q_{(Y, X, q), (Y', X', q')} g \in {_{\GabrielZismanLocalisation(\mathcal{C})}}(X, X')\) for \(g \in {_{\CategoryOfObjectsWithSReplacement{\mathcal{D}}{F}}}((Y, X, q), (Y', X', q'))\) is the unique element with
\[\LocalisationFunctor(q g) \, \LocalisationFunctor(q')^{- 1} = \GabrielZismanLocalisation(F) \bar Q_{(Y, X, q), (Y', X', q')} g,\]
that is, with
\[\LocalisationFunctor(q) \, \LocalisationFunctor(g) = (\GabrielZismanLocalisation(F) \bar Q_{(Y, X, q), (Y', X', q')} g) \, \LocalisationFunctor(q'),\]
in \(\GabrielZismanLocalisation(\mathcal{D})\).

Given morphisms \(g\colon (Y, X, q) \map (Y', X', q')\) and \(g'\colon (Y', X', q') \map (Y'', X'', q'')\) in \(\CategoryOfObjectsWithSReplacement{\mathcal{D}}{F}\), we have
\begin{align*}
\LocalisationFunctor(q) \, \LocalisationFunctor(g g') & = \LocalisationFunctor(q) \, \LocalisationFunctor(g) \, \LocalisationFunctor(g') = (\GabrielZismanLocalisation(F) \bar Q_{(Y, X, q), (Y', X', q')} g) \, \LocalisationFunctor(q') \, \LocalisationFunctor(g') \\
& = (\GabrielZismanLocalisation(F) \bar Q_{(Y, X, q), (Y', X', q')} g) \, (\GabrielZismanLocalisation(F) \bar Q_{(Y', X', q'), (Y'', X'', q'')} g') \, \LocalisationFunctor(q'') \\
& = \GabrielZismanLocalisation(F)((\bar Q_{(Y, X, q), (Y', X', q')} g) \, (\bar Q_{(Y', X', q'), (Y'', X'', q'')} g')) \, \LocalisationFunctor(q'')
\end{align*}
in \(\GabrielZismanLocalisation(\mathcal{D})\) and therefore \(\bar Q_{(Y, X, q), (Y'', X'', q'')}(g g') = (\bar Q_{(Y, X, q), (Y', X', q')} g) (\bar Q_{(Y', X', q'), (Y'', X'', q'')} g')\) in \(\GabrielZismanLocalisation(\mathcal{C})\). Moreover, for \((Y, X, q) \in \Ob \CategoryOfObjectsWithSReplacement{\mathcal{D}}{F}\) we have
\[\LocalisationFunctor(q) \, \LocalisationFunctor(1_Y) = 1_{F X} \, \LocalisationFunctor(q) = (\GabrielZismanLocalisation(F) 1_X) \, \LocalisationFunctor(q)\]
in \(\GabrielZismanLocalisation(\mathcal{D})\) and therefore \(\bar Q_{(Y, X, q), (Y, X, q)}(1_Y) = 1_X = 1_{\bar Q_0 (Y, X, q)}\) in \(\GabrielZismanLocalisation(\mathcal{C})\).

Thus we have a functor \(\TotalSReplacementFunctor F\colon \CategoryOfObjectsWithSReplacement{\mathcal{D}}{F} \map \GabrielZismanLocalisation(\mathcal{C})\) given by \(\Ob \TotalSReplacementFunctor F = \bar Q_0\) and by \((\TotalSReplacementFunctor F)_{(X, q), (X', q')} g = \bar Q_{(Y, X, q), (Y', X', q')} g\) for every morphism \(g\colon (Y, X, q) \map (Y', X', q')\) in~\(\CategoryOfObjectsWithSReplacement{\mathcal{D}}{F}\).
\end{proof}

\begin{definition}[total S-replacement functor] \label{def:total_s-replacement_functor}
We suppose that \(F\) is S-full and S-faithful. The functor~\(\TotalSReplacementFunctor F\colon \CategoryOfObjectsWithSReplacement{\mathcal{D}}{F} \map \GabrielZismanLocalisation(\mathcal{C})\) from proposition~\ref{prop:well-definedness_of_total_s-replacement_functor} is called the \newnotion{total S{\nbd}replacement functor} along \(F\).
\end{definition}

\begin{remark} \label{rem:total_s-replacement_functor_and_shortening}
We suppose that \(\mathcal{D}\) is multiplicative and that \(F\) is S-full and S-faithful. Moreover, we suppose given a morphism \(g\colon (Y, X, q) \map (Y', X', q')\) in \(\CategoryOfObjectsWithSReplacement{\mathcal{D}}{F}\), denominators \(e\colon Y \map \tilde Y\) and \(e'\colon Y' \map \tilde Y'\) in \(\mathcal{D}\) and a morphism \(\tilde g\colon \tilde Y \map \tilde Y'\) in \(\mathcal{D}\) such that \(g e' = e \tilde g\) in \(\mathcal{D}\).
\[\begin{tikzpicture}[baseline=(m-3-1.base)]
  \matrix (m) [diagram]{
    F X & F X' \\
    Y & Y' \\
    \tilde Y & \tilde Y' \\ };
  \path[->, font=\scriptsize]
    (m-1-1) edge[den] node[left] {\(q\)} (m-2-1)
    (m-1-2) edge[den] node[right] {\(q'\)} (m-2-2)
    (m-2-1) edge node[above] {\(g\)} (m-2-2)
            edge[den] node[left] {\(e\)} ([yshift=2pt]m-3-1.north)
    (m-2-2) edge[den] node[right] {\(e'\)} ([yshift=2pt]m-3-2.north)
    (m-3-1) edge node[above] {\(\tilde g\)} (m-3-2);
\end{tikzpicture}\]
Then we have
\[(\TotalSReplacementFunctor F)_{(X, q e), (X', q' e')} \tilde g = (\TotalSReplacementFunctor F)_{(X, q), (X', q')} g\]
in \(\GabrielZismanLocalisation(\mathcal{C})\).
\end{remark}
\begin{proof}
The pair \((X, q e)\) is an S-replacement of \(\tilde Y\) and the pair \((X', q' e')\) is an S-replacement of \(\tilde Y'\) by the multiplicativity of \(\mathcal{D}\). Thus
\begin{align*}
\LocalisationFunctor(q e) \, \LocalisationFunctor(\tilde g) & = \LocalisationFunctor(q) \, \LocalisationFunctor(e) \, \LocalisationFunctor(\tilde g) = \LocalisationFunctor(q) \, \LocalisationFunctor(g) \, \LocalisationFunctor(e') = (\GabrielZismanLocalisation(F) (\TotalSReplacementFunctor F)_{(X, q), (X', q')} g) \, \LocalisationFunctor(q') \, \LocalisationFunctor(e') \\
& = (\GabrielZismanLocalisation(F) (\TotalSReplacementFunctor F)_{(X, q), (X', q')} g) \, \LocalisationFunctor(q' e')
\end{align*}
implies that
\[(\TotalSReplacementFunctor F)_{(X, q e), (X', q' e')} \tilde g = (\TotalSReplacementFunctor F)_{(X, q), (X', q')} g\]
in \(\GabrielZismanLocalisation(\mathcal{C})\).
\[\begin{tikzpicture}[baseline=(m-3-1.base)]
  \matrix (m) [diagram]{
    F X &[7.5em] F X' \\
    Y & Y' \\
    \tilde Y & \tilde Y' \\ };
  \path[->, font=\scriptsize]
    (m-1-1) edge node[above] {\(\GabrielZismanLocalisation(F) (\TotalSReplacementFunctor F)_{(X, q), (X', q')} g\)} (m-1-2)
            edge node[left] {\(\LocalisationFunctor(q)\)} node[sloped, above] {\(\isomorphic\)} (m-2-1)
    (m-1-2) edge node[right] {\(\LocalisationFunctor(q')\)} node[sloped, below] {\(\isomorphic\)} (m-2-2)
    (m-2-1) edge node[above] {\(\LocalisationFunctor(g)\)} (m-2-2)
            edge node[left] {\(\LocalisationFunctor(e)\)} node[sloped, above] {\(\isomorphic\)} ([yshift=2pt]m-3-1.north)
    (m-2-2) edge node[right] {\(\LocalisationFunctor(e')\)} node[sloped, below] {\(\isomorphic\)} ([yshift=2pt]m-3-2.north)
    (m-3-1) edge node[above] {\(\LocalisationFunctor(\tilde g)\)} (m-3-2);
\end{tikzpicture} \qedhere\]
\end{proof}

\begin{corollary} \label{cor:values_of_denominators_under_total_s-replacement_functor}
We suppose that \(\mathcal{D}\) is multiplicative and that \(F\) is S-full and S-faithful. Moreover, we suppose given a denominator~\(e\colon (Y, X, q) \map (Y', X', q')\) in~\(\CategoryOfObjectsWithSReplacement{\mathcal{D}}{F}\). Then we have
\[(\TotalSReplacementFunctor F)_{(X, q), (X', q')} e = (\TotalSReplacementFunctor F)_{(X, q e), (X', q')} 1_{Y'}\]
in \(\GabrielZismanLocalisation(\mathcal{C})\).
\end{corollary}
\begin{proof}
This follows from remark~\ref{rem:total_s-replacement_functor_and_shortening}.
\[\begin{tikzpicture}[baseline=(m-3-1.base)]
  \matrix (m) [diagram]{
    F X & F X' \\
    Y & Y' \\
    Y' & Y' \\ };
  \path[->, font=\scriptsize]
    (m-1-1) edge[den] node[left] {\(q\)} (m-2-1)
    (m-1-2) edge[den] node[right] {\(q'\)} (m-2-2)
    (m-2-1) edge[den] node[above] {\(e\)} (m-2-2)
            edge[den] node[left] {\(e\)} ([yshift=2pt]m-3-1.north)
    (m-2-2) edge[den] node[right] {\(1_{Y'}\)} ([yshift=2pt]m-3-2.north)
    (m-3-1) edge[den] node[above] {\(1_{Y'}\)} (m-3-2);
\end{tikzpicture} \qedhere\]
\end{proof}

\begin{corollary} \label{cor:total_s-replacement_functor_maps_denominators_to_isomorphisms}
We suppose that \(\mathcal{D}\) is multiplicative and that \(F\) is S-full and S-faithful. The total S{\nbd}replacement functor~\(\TotalSReplacementFunctor F\colon \CategoryOfObjectsWithSReplacement{\mathcal{D}}{F} \map \GabrielZismanLocalisation(\mathcal{C})\) maps denominators in \(\CategoryOfObjectsWithSReplacement{\mathcal{D}}{F}\) to isomorphisms in~\(\GabrielZismanLocalisation(\mathcal{C})\).
\end{corollary}
\begin{proof}
We suppose given a denominator \(e\colon (Y, X, q) \map (Y', X', q')\) in \(\CategoryOfObjectsWithSReplacement{\mathcal{D}}{F}\). Then we have
\[(\TotalSReplacementFunctor F)_{(X, q), (X', q')} e = (\TotalSReplacementFunctor F)_{(X, q e), (X', q')} 1_{Y'}\]
in \(\GabrielZismanLocalisation(\mathcal{C})\) by corollary~\ref{cor:values_of_denominators_under_total_s-replacement_functor}. In particular, \((\TotalSReplacementFunctor F)_{(X, q), (X', q')} e = (\TotalSReplacementFunctor F)_{(X, q e), (X', q')} 1_{Y'}\) is an isomorphism in~\(\GabrielZismanLocalisation(\mathcal{C})\) since~\(1_{Y'}\colon (Y', X, q e) \map (Y', X', q')\) is an isomorphism in \(\CategoryOfObjectsWithSReplacement{\mathcal{D}}{F}\).
\end{proof}

\begin{notation} \label{not:induced_functor_of_total_s-replacement_functor}
We suppose that \(\mathcal{D}\) is multiplicative and that \(F\) is S-full and S-faithful. We denote by
\[\InducedFunctorOfTotalSReplacementFunctor F\colon \GabrielZismanLocalisation(\CategoryOfObjectsWithSReplacement{\mathcal{D}}{F}) \map \GabrielZismanLocalisation(\mathcal{C})\]
the unique functor with \(\TotalSReplacementFunctor F = \InducedFunctorOfTotalSReplacementFunctor F \comp \LocalisationFunctor[\GabrielZismanLocalisation(\CategoryOfObjectsWithSReplacement{\mathcal{D}}{F})]\).
\end{notation}

\begin{remark} \label{rem:description_of_induced_functor_of_total_s-replacement_functor_under_fullness_and_faithfulness_assumption}
We suppose that \(\mathcal{D}\) is multiplicative and that \(\GabrielZismanLocalisation(F)\) is full and faithful. The functor \(\InducedFunctorOfTotalSReplacementFunctor F\colon \GabrielZismanLocalisation(\CategoryOfObjectsWithSReplacement{\mathcal{D}}{F}) \map \GabrielZismanLocalisation(\mathcal{C})\) is given on the objects by
\[(\InducedFunctorOfTotalSReplacementFunctor F)_{(X, q)} Y = X\]
for \((Y, X, q) \in \Ob \GabrielZismanLocalisation(\CategoryOfObjectsWithSReplacement{\mathcal{D}}{F}) = \Ob \CategoryOfObjectsWithSReplacement{\mathcal{D}}{F}\), and on the morphisms as follows. Given a morphism \linebreak 
\(\psi\colon (Y, X, q) \map (Y', X', q')\) in \(\GabrielZismanLocalisation(\CategoryOfObjectsWithSReplacement{\mathcal{D}}{F})\), then \((\InducedFunctorOfTotalSReplacementFunctor F) \psi\colon X \map X'\) is the unique morphism in~\(\GabrielZismanLocalisation(\mathcal{C})\) with
\[\LocalisationFunctor(q) \, (\GabrielZismanLocalisation(\ForgetfulFunctor) \psi) = (\GabrielZismanLocalisation(F) (\InducedFunctorOfTotalSReplacementFunctor F) \psi) \, \LocalisationFunctor(q')\]
in \(\GabrielZismanLocalisation(\mathcal{D})\).
\[\begin{tikzpicture}[baseline=(m-2-1.base)]
  \matrix (m) [diagram]{
    F X &[3.5em] F X' \\
    Y & Y' \\ };
  \path[->, font=\scriptsize]
    (m-1-1) edge node[above] {\(\GabrielZismanLocalisation(F) (\InducedFunctorOfTotalSReplacementFunctor F) \psi\)} (m-1-2)
            edge node[left] {\(\LocalisationFunctor(q)\)} node[sloped, above] {\(\isomorphic\)} (m-2-1)
    (m-1-2) edge node[right] {\(\LocalisationFunctor(q')\)} node[sloped, below] {\(\isomorphic\)} (m-2-2)
    (m-2-1) edge node[above] {\(\GabrielZismanLocalisation(\ForgetfulFunctor) \psi\)} (m-2-2);
\end{tikzpicture}\]
\end{remark}
\begin{proof}
For~\((Y, X, q) \in \Ob \CategoryOfObjectsWithSReplacement{\mathcal{D}}{F}\) we have
\[(\InducedFunctorOfTotalSReplacementFunctor F)_{(X, q)} Y = (\TotalSReplacementFunctor F)_{(X, q)} Y = X\]
in \(\GabrielZismanLocalisation(\mathcal{C})\) by proposition~\ref{prop:well-definedness_of_total_s-replacement_functor}. We suppose given a morphism \(\psi\colon (Y, X, q) \map (Y', X', q')\) in \(\GabrielZismanLocalisation(\CategoryOfObjectsWithSReplacement{\mathcal{D}}{F})\) and we let \(\varphi\colon X \map X'\) be the unique morphism in \(\GabrielZismanLocalisation(\mathcal{C})\) with \(\LocalisationFunctor(q) \, (\GabrielZismanLocalisation(\ForgetfulFunctor) \psi) \, \LocalisationFunctor(q')^{- 1} = \GabrielZismanLocalisation(F) \varphi\), that is, with \(\LocalisationFunctor(q) \, (\GabrielZismanLocalisation(\ForgetfulFunctor) \psi) = (\GabrielZismanLocalisation(F) \varphi) \, \LocalisationFunctor(q')\) in~\(\GabrielZismanLocalisation(\mathcal{D})\). There exist \(n \in \Naturals\), morphisms~\(g_i\colon (Y_{i - 1}, X_{i - 1}, q_{i - 1}) \map (\tilde Y_i, \tilde X_i, \tilde q_i)\) in~\(\CategoryOfObjectsWithSReplacement{\mathcal{D}}{F}\) for~\(i \in [1, n]\) and denominators \(b_i\colon (Y_i, X_i, q_i) \map (\tilde Y_i, \tilde X_i, \tilde q_i)\) in~\(\CategoryOfObjectsWithSReplacement{\mathcal{D}}{F}\) for \(i \in [1, n - 1]\) with \((Y, X, q) = (Y_0, X_0, q_0)\), \((Y', X', q') = (\tilde Y_n, \tilde X_n, \tilde q_n)\) and such that
\[\psi = \LocalisationFunctor[\GabrielZismanLocalisation(\CategoryOfObjectsWithSReplacement{\mathcal{D}}{F})](g_1) \, \LocalisationFunctor[\GabrielZismanLocalisation(\CategoryOfObjectsWithSReplacement{\mathcal{D}}{F})](b_1)^{- 1} \, \dots \, \LocalisationFunctor[\GabrielZismanLocalisation(\CategoryOfObjectsWithSReplacement{\mathcal{D}}{F})](g_n)\]
in \(\GabrielZismanLocalisation(\CategoryOfObjectsWithSReplacement{\mathcal{D}}{F})\).
\[\begin{tikzpicture}[baseline=(m-1-1.base)]
  \matrix (m) [diagram]{
    (Y, X, q) & (\tilde Y_1, \tilde X_1, \tilde q_1) & (Y_1, X_1, q_1) & \dots & (Y_{n - 1}, X_{n - 1}, q_{n - 1}) & (Y', X', q') \\ };
  \path[->, font=\scriptsize]
    (m-1-1) edge[exists] node[above] {\(g_1\)} (m-1-2)
    (m-1-3) edge[exists] node[above] {\(g_2\)} (m-1-4)
            edge[exists, den] node[above] {\(b_1\)} (m-1-2)
    (m-1-5) edge[exists] node[above] {\(g_n\)} (m-1-6)
            edge[exists, den] node[above] {\(b_{n - 1}\)} (m-1-4);
\end{tikzpicture}\]
By proposition~\ref{prop:well-definedness_of_total_s-replacement_functor} we have
\begin{align*}
& \LocalisationFunctor[\GabrielZismanLocalisation(\mathcal{D})](q) \, (\GabrielZismanLocalisation(\ForgetfulFunctor) \psi) \\
& = \LocalisationFunctor[\GabrielZismanLocalisation(\mathcal{D})](q) \, \LocalisationFunctor[\GabrielZismanLocalisation(\mathcal{D})](g_1) \, \LocalisationFunctor[\GabrielZismanLocalisation(\mathcal{D})](b_1)^{- 1} \, \LocalisationFunctor[\GabrielZismanLocalisation(\mathcal{D})](g_2) \, \dots \, \LocalisationFunctor[\GabrielZismanLocalisation(\mathcal{D})](b_{n - 1})^{- 1} \, \LocalisationFunctor[\GabrielZismanLocalisation(\mathcal{D})](g_n) \\
& = (\GabrielZismanLocalisation(F) (\TotalSReplacementFunctor F)_{(X_0, q_0), (\tilde X_1, \tilde q_1)} g_1) \, (\GabrielZismanLocalisation(F) (\TotalSReplacementFunctor F)_{(X_1, q_1), (\tilde X_1, \tilde q_1)} b_1)^{- 1} \, (\GabrielZismanLocalisation(F) (\TotalSReplacementFunctor F)_{(X_1, q_1), (\tilde X_2, \tilde q_2)} g_2) \\
& \qquad \dots \, (\GabrielZismanLocalisation(F) (\TotalSReplacementFunctor F)_{(X_{n - 1}, q_{n - 1}), (\tilde X_n, \tilde q_n)} g_n) \, \LocalisationFunctor(q') \\
& = (\GabrielZismanLocalisation(F)(((\TotalSReplacementFunctor F)_{(X_0, q_0), (\tilde X_1, \tilde q_1)} g_1) \, ((\TotalSReplacementFunctor F)_{(X_1, q_1), (\tilde X_1, \tilde q_1)} b_1)^{- 1} \, ((\TotalSReplacementFunctor F)_{(X_1, q_1), (\tilde X_2, \tilde q_2)} g_2) \\
& \qquad \dots \, ((\TotalSReplacementFunctor F)_{(X_{n - 1}, q_{n - 1}), (\tilde X_n, \tilde q_n)} g_n))) \, \LocalisationFunctor(q')
\end{align*}
in \(\GabrielZismanLocalisation(\mathcal{D})\) and therefore
\begin{align*}
(\InducedFunctorOfTotalSReplacementFunctor F) \psi & = (\InducedFunctorOfTotalSReplacementFunctor F)(\LocalisationFunctor[\GabrielZismanLocalisation(\CategoryOfObjectsWithSReplacement{\mathcal{D}}{F})](g_1) \, \LocalisationFunctor[\GabrielZismanLocalisation(\CategoryOfObjectsWithSReplacement{\mathcal{D}}{F})](b_1)^{- 1} \, \dots \, \LocalisationFunctor[\GabrielZismanLocalisation(\CategoryOfObjectsWithSReplacement{\mathcal{D}}{F})](g_n)) \\
& = ((\TotalSReplacementFunctor F)_{(X_0, q_0), (\tilde X_1, \tilde q_1)} g_1) \, ((\TotalSReplacementFunctor F)_{(X_1, q_1), (\tilde X_1, \tilde q_1)} b_1)^{- 1} \, \dots \, ((\TotalSReplacementFunctor F)_{(X_{n - 1}, q_{n - 1}), (\tilde X_n, \tilde q_n)} g_n) = \varphi
\end{align*}
in \(\GabrielZismanLocalisation(\mathcal{C})\).
\[\begin{tikzpicture}[baseline=(m-2-1.base)]
  \matrix (m) [diagram]{
    F X &[7.5em] F \tilde X_1 &[8.5em] \dots &[10.0em] F X' \\
    Y & \tilde Y_1 & \dots & Y' \\ };
  \path[->, font=\scriptsize]
    (m-1-1) edge node[above] {\(\GabrielZismanLocalisation(F) (\TotalSReplacementFunctor F)_{(X, q), (\tilde X_1, \tilde q_1)} g_1\)} (m-1-2)
            edge node[right] {\(\LocalisationFunctor(q)\)} node[sloped, below] {\(\isomorphic\)} (m-2-1)
    (m-1-2) edge node[right] {\(\LocalisationFunctor(\tilde q_1)\)} node[sloped, below] {\(\isomorphic\)} ([yshift=2pt]m-2-2.north)
    (m-1-3) edge node[above] {\(\GabrielZismanLocalisation(F) (\TotalSReplacementFunctor F)_{(X_{n - 1}, q_{n - 1}), (X', q')} g_n\)} (m-1-4)
            edge node[above] {\(\GabrielZismanLocalisation(F) (\TotalSReplacementFunctor F)_{(X_1, q_1), (\tilde X_1, \tilde q_1)} b_1\)} node[below] {\(\isomorphic\)} (m-1-2)
    (m-1-4) edge node[right] {\(\LocalisationFunctor(q')\)} node[sloped, below] {\(\isomorphic\)} (m-2-4)
    (m-2-1) edge node[above] {\(\LocalisationFunctor(g_1)\)} (m-2-2)
    (m-2-3) edge node[above] {\(\LocalisationFunctor(g_n)\)} (m-2-4)
            edge node[above] {\(\LocalisationFunctor(b_1)\)} node[below] {\(\isomorphic\)} (m-2-2);
\end{tikzpicture} \qedhere\]
\end{proof}

The essential device in the proof of corollary~\ref{cor:total_s-replacement_functor_maps_denominators_to_isomorphisms} is corollary~\ref{cor:values_of_denominators_under_total_s-replacement_functor}, where we use the multiplicativity of~\(\mathcal{D}\) already for its formulation (\(q e\) has to be a denominator in \(\mathcal{D}\)). However, corollary~\ref{cor:total_s-replacement_functor_maps_denominators_to_isomorphisms} is the only step in our proof of the S{\nbd}approximation theorem~\ref{th:s-approximation_theorem} that needs the closedness of the denominators in \(\mathcal{D}\) under composition. This leads to the following question:

\begin{question} \label{qu:necessity_of_multiplicativity_of_s-approximation_functors}
Does corollary~\ref{cor:total_s-replacement_functor_maps_denominators_to_isomorphisms} still hold if we omit the assumption that \(\mathcal{D}\) is multiplicative?
\end{question}

If there is a positive answer to question~\ref{qu:necessity_of_multiplicativity_of_s-approximation_functors}, in order to prove the S{\nbd}approximation theorem~\ref{th:s-approximation_theorem}, it would suffice to replace the multiplicativity of \(\mathcal{D}\) by the assumption that \(\mathcal{D}\) has all trivial S-replacements, see definition~\ref{def:having_all_trivial_s-replacements}, which is needed for definition~\ref{def:canonical_lift_along_forgetful_functor_of_s-replacement_category} of the canonical lift.

\begin{proposition} \label{prop:canonical_lift_and_total_s-replacement_functor}
We suppose that \(\mathcal{D}\) has all trivial S-replacements along \(F\) and we suppose that \(F\) is S-full and~S{\nbd}faithful. Moreover, we let~\(\bar F\colon \mathcal{C} \map \CategoryOfObjectsWithSReplacement{\mathcal{D}}{F}\) be the canonical lift of \(F\) along the forgetful functor~\(\ForgetfulFunctor\colon \CategoryOfObjectsWithSReplacement{\mathcal{D}}{F} \map \mathcal{D}\).
\begin{enumerate}
\item \label{prop:canonical_lift_and_total_s-replacement_functor:coretraction} We have
\[\TotalSReplacementFunctor F \comp \bar F = \LocalisationFunctor[\GabrielZismanLocalisation(\mathcal{C})].\]
\item \label{prop:canonical_lift_and_total_s-replacement_functor:retraction_up_to_isomorphism_and_forgetful_functor} We have
\[\GabrielZismanLocalisation(F) \comp \TotalSReplacementFunctor F \isomorphic \LocalisationFunctor[\GabrielZismanLocalisation(\mathcal{D})] \comp \ForgetfulFunctor.\]
An isotransformation \(\bar \beta\colon \GabrielZismanLocalisation(F) \comp \TotalSReplacementFunctor F \map \LocalisationFunctor[\GabrielZismanLocalisation(\mathcal{D})] \comp \ForgetfulFunctor\) is given by
\[\bar \beta_{(Y, X, q)} = \LocalisationFunctor[\GabrielZismanLocalisation(\mathcal{D})](q)\colon F X \map Y\]
for \((Y, X, q) \in \Ob{\CategoryOfObjectsWithSReplacement{\mathcal{D}}{F}}\).
\item \label{prop:canonical_lift_and_total_s-replacement_functor:retraction_up_to_isomorphism} We suppose that \(F\) is S-dense. Then we have
\[\GabrielZismanLocalisation(\bar F) \comp \TotalSReplacementFunctor F \isomorphic \LocalisationFunctor[\GabrielZismanLocalisation(\CategoryOfObjectsWithSReplacement{\mathcal{D}}{F})].\]
An isotransformation \(\bar \beta\colon \GabrielZismanLocalisation(\bar F) \comp \TotalSReplacementFunctor F \map \LocalisationFunctor[\GabrielZismanLocalisation(\CategoryOfObjectsWithSReplacement{\mathcal{D}}{F})]\) is given by
\[\bar \beta_{(Y, X, q)} = \LocalisationFunctor[\GabrielZismanLocalisation(\CategoryOfObjectsWithSReplacement{\mathcal{D}}{F})]_{(X, 1_{F X}), (X, q)}(q)\colon (F X, X, 1_{F X}) \map (Y, X, q)\]
for \((Y, X, q) \in \Ob{\CategoryOfObjectsWithSReplacement{\mathcal{D}}{F}}\).
\end{enumerate}
\end{proposition}
\begin{proof} \
\begin{enumerate}
\item For every morphism \(f\colon X \map X'\) in \(\mathcal{C}\), we have \(\bar F f = F f\colon (F X, X, 1_{F X}) \map (F X', X', 1_{F X'})\). As
\[\LocalisationFunctor(1_{F X}) \, \LocalisationFunctor(F f) = \LocalisationFunctor(F f) = \GabrielZismanLocalisation(F) \LocalisationFunctor(f) = (\GabrielZismanLocalisation(F) \LocalisationFunctor(f)) \, \LocalisationFunctor(1_{F X'})\]
in \(\GabrielZismanLocalisation(\mathcal{D})\), we obtain
\[(\TotalSReplacementFunctor F) (\bar F f) = (\TotalSReplacementFunctor F)_{(X, 1_{F X}), (X', 1_{F X'})}(F f) = \LocalisationFunctor(f)\]
in \(\GabrielZismanLocalisation(\mathcal{C})\).
\[\begin{tikzpicture}[baseline=(m-2-1.base)]
  \matrix (m) [diagram]{
    F X & F X' \\
    F X & F X' \\ };
  \path[->, font=\scriptsize]
    (m-1-1) edge node[above] {\(F f\)} (m-1-2)
            edge[den] node[left=-2pt] {\(1_{F X}\)} (m-2-1)
    (m-1-2) edge[den] node[right=-1pt] {\(1_{F X'}\)} (m-2-2)
    (m-2-1) edge node[above] {\(F f\)} (m-2-2);
\end{tikzpicture}
\qquad
\begin{tikzpicture}[baseline=(m-2-1.base)]
  \matrix (m) [diagram]{
    F X &[3.0em] F X' \\
    F X & F X' \\ };
  \path[->, font=\scriptsize]
    (m-1-1) edge node[above] {\(\GabrielZismanLocalisation(F) \LocalisationFunctor(f)\)} (m-1-2)
            edge node[left] {\(\LocalisationFunctor(1_{F X})\)} node[sloped, above] {\(\isomorphic\)} (m-2-1)
    (m-1-2) edge node[right] {\(\LocalisationFunctor(1_{F X'})\)} node[sloped, below] {\(\isomorphic\)} (m-2-2)
    (m-2-1) edge node[above] {\(\LocalisationFunctor(F f)\)} (m-2-2);
\end{tikzpicture}\]
Thus we have \(\TotalSReplacementFunctor F \comp \bar F = \LocalisationFunctor[\GabrielZismanLocalisation(\mathcal{C})]\).
\item For every morphism \(g\colon (Y, X, q) \map (Y', X', q')\) in \(\CategoryOfObjectsWithSReplacement{\mathcal{D}}{F}\) the following quadrangle in \(\GabrielZismanLocalisation(\mathcal{D})\) commutes by definition of \(\TotalSReplacementFunctor F\colon \CategoryOfObjectsWithSReplacement{\mathcal{D}}{F} \map \GabrielZismanLocalisation(\mathcal{C})\).
\[\begin{tikzpicture}[baseline=(m-2-1.base)]
  \matrix (m) [diagram]{
    F X &[7.0em] F X' \\
    Y & Y' \\ };
  \path[->, font=\scriptsize]
    (m-1-1) edge node[above] {\(\GabrielZismanLocalisation(F) (\TotalSReplacementFunctor F)_{(X, q), (X', q')} g\)} (m-1-2)
            edge node[left] {\(\LocalisationFunctor(q)\)} node[sloped, above] {\(\isomorphic\)} (m-2-1)
    (m-1-2) edge node[right] {\(\LocalisationFunctor(q')\)} node[sloped, below] {\(\isomorphic\)} (m-2-2)
    (m-2-1) edge node[above] {\(\LocalisationFunctor(g)\)} (m-2-2);
\end{tikzpicture}\]
Thus we have a transformation \(\bar \beta\colon \GabrielZismanLocalisation(F) \comp \TotalSReplacementFunctor F \map \LocalisationFunctor[\GabrielZismanLocalisation(\mathcal{D})] \comp \ForgetfulFunctor\), given by
\[\bar \beta_{(Y, X, q)} = \LocalisationFunctor(q)\colon F X \map Y\]
for~\((Y, X, q) \in \Ob{\CategoryOfObjectsWithSReplacement{\mathcal{D}}{F}}\). Moreover, for every object \((Y, X, q)\) in~\(\CategoryOfObjectsWithSReplacement{\mathcal{D}}{F}\), as \((X, q)\) is an S{\nbd}replacement of \(Y\), the morphism \(q\colon F X \map Y\) is a denominator in \(\mathcal{D}\) and hence \(\bar \beta_{(Y, X, q)} = \LocalisationFunctor(q)\colon F X \map Y\) is an isomorphism in \(\GabrielZismanLocalisation(\mathcal{D})\). Thus \(\bar \beta\) is an isotransformation.
\item As \(F\) is S-dense, there exists a choice of S-replacements for \(\mathcal{D}\) along \(F\). Thus \(\ForgetfulFunctor\colon \CategoryOfObjectsWithSReplacement{\mathcal{D}}{F} \map \mathcal{D}\) is an equivalence of categories by remark~\ref{rem:forgetful_functor_of_s-replacement_category_and_choices_of_s-replacements} and hence \(\GabrielZismanLocalisation(\ForgetfulFunctor)\colon \GabrielZismanLocalisation(\CategoryOfObjectsWithSReplacement{\mathcal{D}}{F}) \map \GabrielZismanLocalisation(\mathcal{D})\) is an equivalence of categories by \(2\)-functoriality. In particular, \(\GabrielZismanLocalisation(\ForgetfulFunctor)\colon \GabrielZismanLocalisation(\CategoryOfObjectsWithSReplacement{\mathcal{D}}{F}) \map \GabrielZismanLocalisation(\mathcal{D})\) is faithful and so for every morphism \(g\colon (Y, X, q) \map (Y', X', q')\) in \(\CategoryOfObjectsWithSReplacement{\mathcal{D}}{F}\) the commutativity of the quadrangle
\[\begin{tikzpicture}[baseline=(m-2-1.base)]
  \matrix (m) [diagram]{
    F X &[7.0em] F X' \\
    Y & Y' \\ };
  \path[->, font=\scriptsize]
    (m-1-1) edge node[above] {\(\GabrielZismanLocalisation(F) (\TotalSReplacementFunctor F)_{(X, q), (X', q')} g\)} (m-1-2)
            edge node[left] {\(\LocalisationFunctor(q)\)} node[sloped, above] {\(\isomorphic\)} (m-2-1)
    (m-1-2) edge node[right] {\(\LocalisationFunctor(q')\)} node[sloped, below] {\(\isomorphic\)} (m-2-2)
    (m-2-1) edge node[above] {\(\LocalisationFunctor(g)\)} (m-2-2);
\end{tikzpicture}\]
in \(\GabrielZismanLocalisation(\mathcal{D})\) implies the commutativity of the quadrangle
\[\begin{tikzpicture}[baseline=(m-2-1.base)]
  \matrix (m) [diagram]{
    (F X, X, 1_{F X}) &[7.0em] (F X', X', 1_{F X'}) \\
    (Y, X, q) & (Y', X', q') \\ };
  \path[->, font=\scriptsize]
    (m-1-1) edge node[above] {\(\GabrielZismanLocalisation(\bar F) (\TotalSReplacementFunctor F)_{(X, q), (X', q')} g\)} (m-1-2)
            edge node[left] {\(\LocalisationFunctor_{(X, 1_{F X}), (X, q)}(q)\)} node[sloped, above] {\(\isomorphic\)} (m-2-1)
    (m-1-2) edge node[right] {\(\LocalisationFunctor_{(X', 1_{F X'}), (X', q')}(q')\)} node[sloped, below] {\(\isomorphic\)} (m-2-2)
    (m-2-1) edge node[above] {\(\LocalisationFunctor_{(X, q), (X', q')}(g)\)} (m-2-2);
\end{tikzpicture}\]
in \(\GabrielZismanLocalisation(\CategoryOfObjectsWithSReplacement{\mathcal{D}}{F})\). Thus we have a transformation \(\bar \beta\colon \GabrielZismanLocalisation(\bar F) \comp \TotalSReplacementFunctor F \map \LocalisationFunctor[\GabrielZismanLocalisation(\CategoryOfObjectsWithSReplacement{\mathcal{D}}{F})]\), given by
\[\bar \beta_{(Y, X, q)} = \LocalisationFunctor_{(X, 1_{F X}), (X, q)}(q)\colon (F X, X, 1_{F X}) \map (Y, X, q)\]
for~\((Y, X, q) \in \Ob{\CategoryOfObjectsWithSReplacement{\mathcal{D}}{F}}\). Moreover, for every object~\((Y, X, q)\) in~\(\CategoryOfObjectsWithSReplacement{\mathcal{D}}{F}\), as \((X, q)\) is an S{\nbd}replacement of \(Y\), the morphism \(q\colon F X \map Y\) is a denominator in \(\mathcal{D}\), hence \(q\colon (F X, X, 1_{F X}) \map (Y, X, q)\) is a denominator in \(\CategoryOfObjectsWithSReplacement{\mathcal{D}}{F}\) and therefore~\(\bar \beta_{(Y, X, q)} = \LocalisationFunctor_{(X, 1_{F X}), (X, q)}(q)\colon (F X, X, 1_{F X}) \map (Y, X, q)\) is an isomorphism in \(\GabrielZismanLocalisation(\CategoryOfObjectsWithSReplacement{\mathcal{D}}{F})\). Thus \(\bar \beta\) is an isotransformation. \qedhere
\end{enumerate}
\end{proof}

\begin{corollary} \label{cor:canonical_lift_along_forgetful_functor_of_s-replacement_category_induces_equivalence_on_localisations}
We suppose that \(\mathcal{D}\) is multiplicative and that \(F\) is S-full and S-faithful. Moreover, we let~\(\bar F\colon \mathcal{C} \map \CategoryOfObjectsWithSReplacement{\mathcal{D}}{F}\) be the canonical lift of \(F\) along the forgetful functor~\(\ForgetfulFunctor\colon \CategoryOfObjectsWithSReplacement{\mathcal{D}}{F} \map \mathcal{D}\).
\begin{enumerate}
\item \label{cor:canonical_lift_along_forgetful_functor_of_s-replacement_category_induces_equivalence_on_localisations:coretraction} We have
\[\InducedFunctorOfTotalSReplacementFunctor F \comp \GabrielZismanLocalisation(\bar F) = \id_{\GabrielZismanLocalisation(\mathcal{C})}.\]
\item \label{cor:canonical_lift_along_forgetful_functor_of_s-replacement_category_induces_equivalence_on_localisations:retraction_up_to_isomorphism_and_forgetful_functor} We have
\[\GabrielZismanLocalisation(F) \comp \InducedFunctorOfTotalSReplacementFunctor F \isomorphic \GabrielZismanLocalisation(\ForgetfulFunctor).\]
An isotransformation \(\bar \beta\colon \GabrielZismanLocalisation(F) \comp \InducedFunctorOfTotalSReplacementFunctor F \map \GabrielZismanLocalisation(\ForgetfulFunctor)\) is given by
\[\bar \beta_{(Y, X, q)} = \LocalisationFunctor[\GabrielZismanLocalisation(\mathcal{D})](q)\colon F X \map Y\]
for \((Y, X, q) \in \Ob{\GabrielZismanLocalisation(\CategoryOfObjectsWithSReplacement{\mathcal{D}}{F})} = \Ob \CategoryOfObjectsWithSReplacement{\mathcal{D}}{F}\).
\item \label{cor:canonical_lift_along_forgetful_functor_of_s-replacement_category_induces_equivalence_on_localisations:retraction_up_to_isomorphism} We suppose that \(F\) is S-dense. Then we have
\[\GabrielZismanLocalisation(\bar F) \comp \InducedFunctorOfTotalSReplacementFunctor F \isomorphic \id_{\GabrielZismanLocalisation(\CategoryOfObjectsWithSReplacement{\mathcal{D}}{F})}.\]
An isotransformation \(\bar \beta\colon \GabrielZismanLocalisation(\bar F) \comp \InducedFunctorOfTotalSReplacementFunctor F \map \id_{\GabrielZismanLocalisation(\CategoryOfObjectsWithSReplacement{\mathcal{D}}{F})}\) is given by
\[\bar \beta_{(Y, X, q)} = \LocalisationFunctor[\GabrielZismanLocalisation(\CategoryOfObjectsWithSReplacement{\mathcal{D}}{F})]_{(X, 1_{F X}), (X, q)}(q)\colon (F X, X, 1_{F X}) \map (Y, X, q)\]
for \((Y, X, q) \in \Ob{\GabrielZismanLocalisation(\CategoryOfObjectsWithSReplacement{\mathcal{D}}{F})} = \Ob \CategoryOfObjectsWithSReplacement{\mathcal{D}}{F}\).
\end{enumerate}
In particular, if \(F\) is S-dense, then \(\GabrielZismanLocalisation(\bar F)\colon \GabrielZismanLocalisation(\mathcal{C}) \map \GabrielZismanLocalisation(\CategoryOfObjectsWithSReplacement{\mathcal{D}}{F})\) and \(\InducedFunctorOfTotalSReplacementFunctor F\colon \GabrielZismanLocalisation(\CategoryOfObjectsWithSReplacement{\mathcal{D}}{F}) \map \GabrielZismanLocalisation(\mathcal{C})\) are mutually isomorphism inverse equivalences of categories.
\end{corollary}
\begin{proof} \
\begin{enumerate}
\item By proposition~\ref{prop:canonical_lift_and_total_s-replacement_functor}\ref{prop:canonical_lift_and_total_s-replacement_functor:coretraction}, we have
\[\InducedFunctorOfTotalSReplacementFunctor F \comp \GabrielZismanLocalisation(\bar F) \comp \LocalisationFunctor[\GabrielZismanLocalisation(\mathcal{C})] = \InducedFunctorOfTotalSReplacementFunctor F \comp \LocalisationFunctor[\GabrielZismanLocalisation(\CategoryOfObjectsWithSReplacement{\mathcal{D}}{F})] \comp \bar F = \TotalSReplacementFunctor F \comp \bar F = \LocalisationFunctor[\GabrielZismanLocalisation(\mathcal{C})]\]
and hence \(\InducedFunctorOfTotalSReplacementFunctor F \comp \GabrielZismanLocalisation(\bar F) = \id_{\GabrielZismanLocalisation(\mathcal{C})}\).
\item By proposition~\ref{prop:canonical_lift_and_total_s-replacement_functor}\ref{prop:canonical_lift_and_total_s-replacement_functor:retraction_up_to_isomorphism_and_forgetful_functor}, we have an isotransformation \(\bar \beta'\colon \GabrielZismanLocalisation(F) \comp \TotalSReplacementFunctor F \map \LocalisationFunctor[\GabrielZismanLocalisation(\mathcal{D})] \comp \ForgetfulFunctor\) given by
\[\bar \beta'_{(Y, X, q)} = \LocalisationFunctor[\GabrielZismanLocalisation(\mathcal{D})](q)\colon F X \map Y\]
for \((Y, X, q) \in \Ob{\CategoryOfObjectsWithSReplacement{\mathcal{D}}{F}}\). Since we have \(\GabrielZismanLocalisation(F) \comp \TotalSReplacementFunctor F = \GabrielZismanLocalisation(F) \comp \InducedFunctorOfTotalSReplacementFunctor F \comp \LocalisationFunctor[\GabrielZismanLocalisation(\CategoryOfObjectsWithSReplacement{\mathcal{D}}{F})]\) and \linebreak 
\(\LocalisationFunctor[\GabrielZismanLocalisation(\mathcal{D})] \comp \ForgetfulFunctor = \GabrielZismanLocalisation(\ForgetfulFunctor) \comp \LocalisationFunctor[\GabrielZismanLocalisation(\CategoryOfObjectsWithSReplacement{\mathcal{D}}{F})]\), there exists a unique transformation \(\bar \beta\colon \GabrielZismanLocalisation(F) \comp \InducedFunctorOfTotalSReplacementFunctor F \map \GabrielZismanLocalisation(\ForgetfulFunctor)\) with \(\bar \beta' = \bar \beta * \LocalisationFunctor[\GabrielZismanLocalisation(\CategoryOfObjectsWithSReplacement{\mathcal{D}}{F})]\), given by
\[\bar \beta_{(Y, X, q)} = \bar \beta'_{(Y, X, q)} = \LocalisationFunctor[\GabrielZismanLocalisation(\mathcal{D})](q)\colon F X \map Y\]
for \((Y, X, q) \in \Ob{\GabrielZismanLocalisation(\CategoryOfObjectsWithSReplacement{\mathcal{D}}{F})} = \Ob{\CategoryOfObjectsWithSReplacement{\mathcal{D}}{F}}\), and this transformation is an isotransformation by remark~\ref{rem:criterion_for_induced_co_retraction_up_to_isomorphism_on_localisations}\ref{rem:criterion_for_induced_co_retraction_up_to_isomorphism_on_localisations:retraction}\ref{rem:criterion_for_induced_co_retraction_up_to_isomorphism_on_localisations:retraction:sufficient}.
\item By proposition~\ref{prop:canonical_lift_and_total_s-replacement_functor}\ref{prop:canonical_lift_and_total_s-replacement_functor:retraction_up_to_isomorphism}, we have an isotransformation \(\bar \beta'\colon \GabrielZismanLocalisation(\bar F) \comp \TotalSReplacementFunctor F \map \LocalisationFunctor[\GabrielZismanLocalisation(\CategoryOfObjectsWithSReplacement{\mathcal{D}}{F})]\) given by
\[\bar \beta'_{(Y, X, q)} = \LocalisationFunctor[\GabrielZismanLocalisation(\CategoryOfObjectsWithSReplacement{\mathcal{D}}{F})]_{(X, 1_{F X}), (X, q)}(q)\colon (F X, X, 1_{F X}) \map (Y, X, q)\]
for \((Y, X, q) \in \Ob{\CategoryOfObjectsWithSReplacement{\mathcal{D}}{F}}\). As \(\GabrielZismanLocalisation(\bar F) \comp \TotalSReplacementFunctor F = \GabrielZismanLocalisation(\bar F) \comp \InducedFunctorOfTotalSReplacementFunctor F \comp \LocalisationFunctor[\GabrielZismanLocalisation(\CategoryOfObjectsWithSReplacement{\mathcal{D}}{F})]\) there exists a unique transformation \(\bar \beta\colon \GabrielZismanLocalisation(\bar F) \comp \InducedFunctorOfTotalSReplacementFunctor F \map \id_{\GabrielZismanLocalisation(\CategoryOfObjectsWithSReplacement{\mathcal{D}}{F})}\) with \(\bar \beta' = \bar \beta * \LocalisationFunctor[\GabrielZismanLocalisation(\CategoryOfObjectsWithSReplacement{\mathcal{D}}{F})]\), given by
\[\bar \beta_{(Y, X, q)} = \bar \beta'_{(Y, X, q)} = \LocalisationFunctor[\GabrielZismanLocalisation(\CategoryOfObjectsWithSReplacement{\mathcal{D}}{F})]_{(X, 1_{F X}), (X, q)}(q)\colon (F X, X, 1_{F X}) \map (Y, X, q)\]
for \((Y, X, q) \in \Ob{\GabrielZismanLocalisation(\CategoryOfObjectsWithSReplacement{\mathcal{D}}{F})} = \Ob{\CategoryOfObjectsWithSReplacement{\mathcal{D}}{F}}\), and this transformation is an isotransformation by remark~\ref{rem:criterion_for_induced_co_retraction_up_to_isomorphism_on_localisations}\ref{rem:criterion_for_induced_co_retraction_up_to_isomorphism_on_localisations:retraction}\ref{rem:criterion_for_induced_co_retraction_up_to_isomorphism_on_localisations:retraction:sufficient}. \qedhere
\end{enumerate}
\end{proof}

\begin{question} \label{qu:necessity_of_s-density_such_that_canonical_lift_is_s-equivalence}
Do proposition~\ref{prop:canonical_lift_and_total_s-replacement_functor}\ref{prop:canonical_lift_and_total_s-replacement_functor:retraction_up_to_isomorphism} and hence corollary~\ref{cor:canonical_lift_along_forgetful_functor_of_s-replacement_category_induces_equivalence_on_localisations}\ref{cor:canonical_lift_along_forgetful_functor_of_s-replacement_category_induces_equivalence_on_localisations:retraction_up_to_isomorphism} still hold if we omit the assumption that \(F\) is S-dense?
\end{question}

\subsection*{S-replacement functors} \label{sec:s-replacement_functors_and_the_s-approximation_theorem:s-replacement_functors}

To conclude the proof of the S-approximation theorem~\ref{th:s-approximation_theorem}, we have to compose the isomorphism inverses induced by the structure choice functor \(\StructureChoiceFunctor{R}\colon \mathcal{D} \map \CategoryOfObjectsWithSReplacement{\mathcal{D}}{F}\) for a choice of S-replacements \(R = ((X_Y, q_Y))_{Y \in \Ob \mathcal{D}}\) for \(\mathcal{D}\) along~\(F\), see corollary~\ref{cor:forgetful_functor_of_s-replacement_category_induces_equivalence_of_categories_on_localisations}, and by the total S-replacement functor \(\TotalSReplacementFunctor F\colon \CategoryOfObjectsWithSReplacement{\mathcal{D}}{F} \map \GabrielZismanLocalisation(\mathcal{C})\), see corollary~\ref{cor:canonical_lift_along_forgetful_functor_of_s-replacement_category_induces_equivalence_on_localisations}. This leads to the following definition.

\begin{definition}[S-replacement functor] \label{def:s-replacement_functor}
We suppose that \(F\) is S-full and S-faithful and we suppose given a choice of S-replacements \(R = ((X_Y, q_Y))_{Y \in \Ob \mathcal{D}}\) for \(\mathcal{D}\) along~\(F\). The composite
\[\SReplacementFunctor F = \SReplacementFunctor[R] F := \TotalSReplacementFunctor F \comp \StructureChoiceFunctor{R}\colon \mathcal{D} \map \GabrielZismanLocalisation(\mathcal{C})\]
is called the \newnotion{S-replacement functor} along \(F\) with respect to \(R\).
\end{definition}

\begin{remark} \label{rem:description_of_s-replacement_functor}
We suppose that \(F\) is S-full and S-faithful and we suppose given a choice of S{\nbd}replace\-ments~\(R = ((X_Y, q_Y))_{Y \in \Ob \mathcal{D}}\) for \(\mathcal{D}\) along~\(F\). The S-replacement functor~\(\SReplacementFunctor[R] F\colon \mathcal{D} \map \GabrielZismanLocalisation(\mathcal{C})\) along \(F\) with respect to \(R\) is given on the objects by
\[(\SReplacementFunctor[R] F) Y = X_Y\]
for \(Y \in \Ob \mathcal{D}\), and on the morphisms as follows. Given a morphism \(g\colon Y \map Y'\) in \(\mathcal{D}\), then \((\SReplacementFunctor[R] F) g\colon X_Y \map X_{Y'}\) is the unique morphism in \(\GabrielZismanLocalisation(\mathcal{C})\) with
\[\LocalisationFunctor(q_Y) \, \LocalisationFunctor(g) = (\GabrielZismanLocalisation(F) (\SReplacementFunctor[R] F) g) \, \LocalisationFunctor(q_{Y'})\]
in \(\GabrielZismanLocalisation(\mathcal{D})\).
\[\begin{tikzpicture}[baseline=(m-2-1.base)]
  \matrix (m) [diagram]{
    F X_Y &[3.5em] F X_{Y'} \\
    Y & Y' \\ };
  \path[->, font=\scriptsize]
    (m-1-1) edge node[above] {\(\GabrielZismanLocalisation(F) (\SReplacementFunctor F) g\)} (m-1-2)
            edge node[left] {\(\LocalisationFunctor(q_Y)\)} node[sloped, above] {\(\isomorphic\)} (m-2-1)
    (m-1-2) edge node[right] {\(\LocalisationFunctor(q_{Y'})\)} node[sloped, below] {\(\isomorphic\)} (m-2-2)
    (m-2-1) edge node[above] {\(\LocalisationFunctor(g)\)} (m-2-2);
\end{tikzpicture}\]
\end{remark}
\begin{proof}
For~\(Y \in \Ob \mathcal{D}\), we have \(\StructureChoiceFunctor{R} Y = (Y, X_Y, q_Y)\) in \(\CategoryOfObjectsWithSReplacement{\mathcal{D}}{F}\) and therefore
\[(\SReplacementFunctor F) Y = (\TotalSReplacementFunctor F) \StructureChoiceFunctor{R} Y = (\TotalSReplacementFunctor F)_{(X_Y, q_Y)} Y = X_Y\]
in \(\GabrielZismanLocalisation(\mathcal{C})\). We suppose given a morphism \(g\colon Y \map Y'\) in \(\mathcal{D}\). Then \((\SReplacementFunctor F) g = (\TotalSReplacementFunctor F) \StructureChoiceFunctor{R} g = (\TotalSReplacementFunctor F)_{(X_Y, q_Y), (X_{Y'}, q_{Y'})} g\) is the unique morphism in \(\GabrielZismanLocalisation(\mathcal{C})\) with
\[\LocalisationFunctor(q_Y) \, \LocalisationFunctor(g) = (\GabrielZismanLocalisation(F) (\TotalSReplacementFunctor F)_{(X_Y, q_Y), (X_{Y'}, q_{Y'})} g) \, \LocalisationFunctor(q_{Y'}) = (\GabrielZismanLocalisation(F) (\SReplacementFunctor F) g) \, \LocalisationFunctor(q_{Y'})\]
in~\(\GabrielZismanLocalisation(\mathcal{D})\).
\end{proof}

\begin{remark} \label{rem:s-replacement_functor_can_be_computed_up_to_isomorphism_via_arbitrary_s-replacements}
We suppose that \(F\) is S-full and S-faithful and we suppose given a choice of S{\nbd}replace- \linebreak 
ments~\(R = ((X_Y, q_Y))_{Y \in \Ob \mathcal{D}}\) for \(\mathcal{D}\) along~\(F\). For every object \(Y\) in \(\mathcal{D}\) and every S-replacement \((\tilde X, \tilde q)\) of \(Y\) along~\(F\) we have the isomorphism
\[(\TotalSReplacementFunctor F)_{(X_Y, q_Y), (\tilde X, \tilde q)} 1_Y\colon (\SReplacementFunctor[R] F) Y \map \tilde X\]
in \(\GabrielZismanLocalisation(\mathcal{C})\).
\end{remark}
\begin{proof}
Given an object \(Y\) in \(\mathcal{D}\) and an S-replacement \((\tilde X, \tilde q)\) of \(Y\) along~\(F\), then \((\TotalSReplacementFunctor F)_{(X_Y, q_Y), (\tilde X, \tilde q)} 1_Y\) is an isomorphism from \((\TotalSReplacementFunctor F)_{(X_Y, q_Y)} Y = (\TotalSReplacementFunctor F) \StructureChoiceFunctor{R} Y = (\SReplacementFunctor F) Y\) to \((\TotalSReplacementFunctor F)_{(\tilde X, \tilde q)} Y = \tilde X\) in \(\GabrielZismanLocalisation(\mathcal{C})\) with inverse~\(((\TotalSReplacementFunctor F)_{(X_Y, q_Y), (\tilde X, \tilde q)} 1_Y)^{- 1} = (\TotalSReplacementFunctor F)_{(\tilde X, \tilde q), (X_Y, q_Y)} 1_Y\).
\end{proof}

\begin{remark} \label{rem:isomorphism_of_s-replacement_functors}
We suppose that \(F\) is S-full and S-faithful and we suppose given choices of S-replace- \linebreak 
ments \(R = ((X_Y, q_Y))_{Y \in \Ob \mathcal{D}}\) and \(\tilde R = ((\tilde X_Y, \tilde q_Y))_{Y \in \Ob \mathcal{D}}\) for \(\mathcal{D}\) along~\(F\). Then we have
\[\SReplacementFunctor[R] F \isomorphic \SReplacementFunctor[\tilde R] F\colon \mathcal{D} \map \GabrielZismanLocalisation(\mathcal{C}).\]
An isotransformation \(\alpha_{R, \tilde R}\colon \SReplacementFunctor[R] F \map \SReplacementFunctor[\tilde R] F\) is given by
\[(\alpha_{R, \tilde R})_Y = (\TotalSReplacementFunctor F)_{(X_Y, q_Y), (\tilde X_Y, \tilde q_Y)} 1_Y\colon \SReplacementFunctor[R] F Y \map \SReplacementFunctor[\tilde R] F Y\]
for \(Y \in \Ob \mathcal{D}\). The inverse of~\(\alpha_{R, \tilde R}\) is given by \(\alpha_{R, \tilde R}^{- 1} = \alpha_{\tilde R, R}\).
\end{remark}
\begin{proof}
We have \(\SReplacementFunctor[R] F = \TotalSReplacementFunctor F \comp \StructureChoiceFunctor{R}\) and~\(\SReplacementFunctor[\tilde R] F = \TotalSReplacementFunctor F \comp \StructureChoiceFunctor{\tilde R}\). By~\cite[cor.~(A.12)]{thomas:2012:a_calculus_of_fractions_for_the_homotopy_category_of_a_brown_cofibration_category}, we have
\[\SReplacementFunctor[R] F = \TotalSReplacementFunctor F \comp \StructureChoiceFunctor{R} \isomorphic \TotalSReplacementFunctor F \comp \StructureChoiceFunctor{\tilde R} = \SReplacementFunctor[\tilde R] F,\]
an isotransformation \(\alpha_{R, \tilde R}\colon \SReplacementFunctor[R] F \map \SReplacementFunctor[\tilde R] F\) is given by
\[(\alpha_{R, \tilde R})_Y = \TotalSReplacementFunctor F_{(X_Y, q_Y), (\tilde X_Y, \tilde q_Y)} 1_Y\colon (\SReplacementFunctor[R] F) Y \map (\SReplacementFunctor[\tilde R] F) Y\]
for \(Y \in \Ob \mathcal{D}\), and the inverse of \(\alpha_{R, \tilde R}\) is given by \(\alpha_{R, \tilde R}^{- 1} = \alpha_{\tilde R, R}\).
\end{proof}

\begin{remark} \label{rem:s-replacement_functor_maps_denominators_to_isomorphisms}
We suppose that \(\mathcal{D}\) is multiplicative and that \(F\) is S-full and S-faithful. Moreover, we suppose given a choice of S{\nbd}replacements \(R = ((X_Y, q_Y))_{Y \in \Ob \mathcal{D}}\) for \(\mathcal{D}\) along~\(F\). The S-replacement func\-tor~\(\SReplacementFunctor F\colon \mathcal{D} \map \GabrielZismanLocalisation(\mathcal{C})\) along \(F\) with respect to~\(R\) maps denominators in \(\mathcal{D}\) to isomorphisms in~\(\GabrielZismanLocalisation(\mathcal{C})\).
\end{remark}
\begin{proof}
The structure choice functor \(\StructureChoiceFunctor{R}\colon \mathcal{D} \map \CategoryOfObjectsWithSReplacement{\mathcal{D}}{F}\) preserves denominators, that is, it maps denominators in~\(\mathcal{D}\) to denominators in \(\CategoryOfObjectsWithSReplacement{\mathcal{D}}{F}\). The total S-replacement functor \(\TotalSReplacementFunctor F\colon \CategoryOfObjectsWithSReplacement{\mathcal{D}}{F} \map \GabrielZismanLocalisation(\mathcal{C})\) maps denominators in \(\CategoryOfObjectsWithSReplacement{\mathcal{D}}{F}\) to isomorphisms in \(\GabrielZismanLocalisation(\mathcal{C})\) by corollary~\ref{cor:total_s-replacement_functor_maps_denominators_to_isomorphisms}. Thus \(\SReplacementFunctor F = \TotalSReplacementFunctor F \comp \StructureChoiceFunctor{R}\) maps denominators in \(\mathcal{D}\) to isomorphisms in~\(\GabrielZismanLocalisation(\mathcal{C})\).
\end{proof}

\begin{notation} \label{not:induced_functor_of_s-replacement_functor}
We suppose that \(\mathcal{D}\) is multiplicative and that \(F\) is S-full and S-faithful. Moreover, we suppose given a choice of S{\nbd}replacements \(R = ((X_Y, q_Y))_{Y \in \Ob \mathcal{D}}\) for \(\mathcal{D}\) along~\(F\). We denote by
\[\InducedFunctorOfSReplacementFunctor F = \InducedFunctorOfSReplacementFunctor[R] F\colon \GabrielZismanLocalisation(\mathcal{D}) \map \GabrielZismanLocalisation(\mathcal{C})\]
the unique functor with \(\SReplacementFunctor[R] F = \InducedFunctorOfSReplacementFunctor[R] F \comp \LocalisationFunctor[\GabrielZismanLocalisation(\mathcal{D})]\).
\end{notation}

\begin{remark} \label{rem:induced_functor_of_s-replacement_functor_and_induced_functor_of_total_s-replacement_functor}
We suppose that \(\mathcal{D}\) is multiplicative and that \(F\) is S-full and S-faithful. Moreover, we suppose given a choice of S{\nbd}replacements \(R = ((X_Y, q_Y))_{Y \in \Ob \mathcal{D}}\) for \(\mathcal{D}\) along~\(F\). Then we have
\[\InducedFunctorOfSReplacementFunctor[R] F = \InducedFunctorOfTotalSReplacementFunctor F \comp \GabrielZismanLocalisation(\StructureChoiceFunctor{R}).\]
\end{remark}
\begin{proof}
Since
\[\InducedFunctorOfTotalSReplacementFunctor F \comp \GabrielZismanLocalisation(\StructureChoiceFunctor{R}) \comp \LocalisationFunctor[\GabrielZismanLocalisation(\mathcal{D})] = \InducedFunctorOfTotalSReplacementFunctor F \comp \LocalisationFunctor[\GabrielZismanLocalisation(\CategoryOfObjectsWithSReplacement{\mathcal{D}}{F})] \comp \StructureChoiceFunctor{R} = \TotalSReplacementFunctor F \comp \StructureChoiceFunctor{R} = \SReplacementFunctor[R] F,\]
we necessarily have \(\InducedFunctorOfSReplacementFunctor[R] F = \InducedFunctorOfTotalSReplacementFunctor F \comp \GabrielZismanLocalisation(\StructureChoiceFunctor{R})\).
\end{proof}

\begin{remark} \label{rem:description_of_induced_functor_of_s-replacement_functor_under_fullness_and_faithfulness_assumption}
We suppose that \(\mathcal{D}\) is multiplicative and that \(\GabrielZismanLocalisation(F)\) is full and faithful. Moreover, we suppose given a choice of S{\nbd}replacements \(R = ((X_Y, q_Y))_{Y \in \Ob \mathcal{D}}\) for \(\mathcal{D}\) along~\(F\). The functor \(\InducedFunctorOfSReplacementFunctor[R] F\colon \GabrielZismanLocalisation(\mathcal{D}) \map \GabrielZismanLocalisation(\mathcal{C})\) is given on the objects by
\[(\InducedFunctorOfSReplacementFunctor[R] F) Y = X_Y\]
for \(Y \in \Ob \GabrielZismanLocalisation(\mathcal{D}) = \Ob \mathcal{D}\), and on the morphisms as follows. Given a morphism~\(\psi\colon Y \map Y'\) in \(\GabrielZismanLocalisation(\mathcal{D})\), then \((\InducedFunctorOfSReplacementFunctor[R] F) \psi\colon X_Y \map X_{Y'}\) is the unique morphism in~\(\GabrielZismanLocalisation(\mathcal{C})\) with
\[\LocalisationFunctor(q_Y) \, \psi = (\GabrielZismanLocalisation(F) (\InducedFunctorOfSReplacementFunctor[R] F) \psi) \, \LocalisationFunctor(q_{Y'})\]
in \(\GabrielZismanLocalisation(\mathcal{D})\).
\[\begin{tikzpicture}[baseline=(m-2-1.base)]
  \matrix (m) [diagram]{
    F X_Y &[3.5em] F X_{Y'} \\
    Y & Y' \\ };
  \path[->, font=\scriptsize]
    (m-1-1) edge node[above] {\(\GabrielZismanLocalisation(F) (\InducedFunctorOfSReplacementFunctor[R] F) \psi\)} (m-1-2)
            edge node[left] {\(\LocalisationFunctor(q_Y)\)} node[sloped, above] {\(\isomorphic\)} (m-2-1)
    (m-1-2) edge node[right] {\(\LocalisationFunctor(q_{Y'})\)} node[sloped, below] {\(\isomorphic\)} (m-2-2)
    (m-2-1) edge node[above] {\(\psi\)} (m-2-2);
\end{tikzpicture}\]
\end{remark}
\begin{proof}
For~\(Y \in \Ob \mathcal{D}\) we have
\[(\InducedFunctorOfSReplacementFunctor[R] F) Y = (\InducedFunctorOfTotalSReplacementFunctor F) \GabrielZismanLocalisation(\StructureChoiceFunctor{R}) Y = (\InducedFunctorOfTotalSReplacementFunctor F)_{(X_Y, q_Y)} Y = X_Y\]
in \(\GabrielZismanLocalisation(\mathcal{C})\) by remark~\ref{rem:induced_functor_of_s-replacement_functor_and_induced_functor_of_total_s-replacement_functor} and remark~\ref{rem:description_of_induced_functor_of_total_s-replacement_functor_under_fullness_and_faithfulness_assumption}. We suppose given a morphism \(\psi\colon Y \map Y'\) in \(\GabrielZismanLocalisation(\mathcal{D})\) and we let~\(\varphi\colon X_Y \map X_{Y'}\) be the unique morphism in \(\GabrielZismanLocalisation(\mathcal{C})\) with \(\LocalisationFunctor(q_Y) \, \psi \, \LocalisationFunctor(q_{Y'})^{- 1} = \GabrielZismanLocalisation(F) \varphi\), that is, \linebreak 
with \(\LocalisationFunctor(q_Y) \, \psi = (\GabrielZismanLocalisation(F) \varphi) \, \LocalisationFunctor(q_{Y'})\) in~\(\GabrielZismanLocalisation(\mathcal{D})\). Then \(\GabrielZismanLocalisation(\StructureChoiceFunctor{R}) \psi\colon (Y, X_Y, q_Y) \map (Y', X_{Y'}, q_{Y'})\) is a morphism \linebreak 
in~\(\GabrielZismanLocalisation(\CategoryOfObjectsWithSReplacement{\mathcal{D}}{F})\) with
\[\LocalisationFunctor(q_Y) \, (\GabrielZismanLocalisation(\ForgetfulFunctor) \GabrielZismanLocalisation(\StructureChoiceFunctor{R}) \psi) = \LocalisationFunctor(q_Y) \, \psi = (\GabrielZismanLocalisation(F) \varphi) \, \LocalisationFunctor(q_{Y'})\]
in \(\GabrielZismanLocalisation(\mathcal{D})\). Thus we have
\[(\InducedFunctorOfSReplacementFunctor[R] F) \psi = (\InducedFunctorOfTotalSReplacementFunctor F) \GabrielZismanLocalisation(\StructureChoiceFunctor{R}) \psi = \varphi\]
by remark~\ref{rem:induced_functor_of_s-replacement_functor_and_induced_functor_of_total_s-replacement_functor} and remark~\ref{rem:description_of_induced_functor_of_total_s-replacement_functor_under_fullness_and_faithfulness_assumption}.
\end{proof}

\subsection*{The S-approximation theorem} \label{sec:s-replacement_functors_and_the_s-approximation_theorem:the_s-approximation_theorem}

Finally, we can state and prove the main result of this article.

\begin{theorem}[S-approximation theorem] \label{th:s-approximation_theorem}
We suppose that \(\mathcal{D}\) is multiplicative and that \(F\) is S-full and S-faithful. Moreover, we suppose given a choice of S{\nbd}replacements \(R = ((X_Y, q_Y))_{Y \in \Ob \mathcal{D}}\) for \(\mathcal{D}\) along~\(F\). The functors
\begin{align*}
& \GabrielZismanLocalisation(F)\colon \GabrielZismanLocalisation(\mathcal{C}) \map \GabrielZismanLocalisation(\mathcal{D}), \\
& \InducedFunctorOfSReplacementFunctor[R] F\colon \GabrielZismanLocalisation(\mathcal{D}) \map \GabrielZismanLocalisation(\mathcal{C})
\end{align*}
are mutually isomorphism inverse equivalences of categories. An isotransformation~\(\alpha\colon \InducedFunctorOfSReplacementFunctor[R] F \comp \GabrielZismanLocalisation(F) \map \id_{\GabrielZismanLocalisation(\mathcal{C})}\) is given by
\[\alpha_{X'} = (\TotalSReplacementFunctor F)_{(X_{F X'}, q_{F X'}), (X', 1_{F X'})} 1_{F X'}\colon X_{F X'} \map X'\]
for \(X' \in \Ob \GabrielZismanLocalisation(\mathcal{C}) = \Ob \mathcal{C}\), and an isotransformation \(\beta\colon \GabrielZismanLocalisation(F) \comp \InducedFunctorOfSReplacementFunctor[R] F \map \id_{\GabrielZismanLocalisation(\mathcal{D})}\) is given by
\[\beta_Y = \LocalisationFunctor[\GabrielZismanLocalisation(\mathcal{D})](q_{Y})\colon F X_Y \map Y\]
for \(Y \in \Ob \GabrielZismanLocalisation(\mathcal{D}) = \Ob \mathcal{D}\).
\end{theorem}
\begin{proof}
By remark~\ref{rem:well-definedness_of_canonical_lift_along_forgetful_functor_of_s-replacement_category}\ref{rem:well-definedness_of_canonical_lift_along_forgetful_functor_of_s-replacement_category:factorisation}, we have \(F = \ForgetfulFunctor \comp \bar F\), where \(\bar F\colon \mathcal{C} \map \CategoryOfObjectsWithSReplacement{\mathcal{D}}{F}\) denotes the canonical lift of \(F\) along the forgetful functor \(\ForgetfulFunctor\colon \CategoryOfObjectsWithSReplacement{\mathcal{D}}{F} \map \mathcal{D}\). By corollary~\ref{cor:forgetful_functor_of_s-replacement_category_induces_equivalence_of_categories_on_localisations}, we have \(\GabrielZismanLocalisation(\ForgetfulFunctor) \comp \GabrielZismanLocalisation(\StructureChoiceFunctor{R}) = \id_{\GabrielZismanLocalisation(\mathcal{D})}\) and an~isotransformation \(\bar \alpha\colon \GabrielZismanLocalisation(\StructureChoiceFunctor{R}) \comp \GabrielZismanLocalisation(\ForgetfulFunctor) \map \id_{\GabrielZismanLocalisation(\CategoryOfObjectsWithSReplacement{\mathcal{D}}{F})}\) is given by \(\bar \alpha_{(Y, X', q')} = 1_Y\colon (Y, X_Y, q_Y) \map (Y, X', q')\) for \((Y, X', q') \in \Ob{\GabrielZismanLocalisation(\CategoryOfObjectsWithSReplacement{\mathcal{D}}{F})} = \Ob{\CategoryOfObjectsWithSReplacement{\mathcal{D}}{F}}\). By corollary~\ref{cor:canonical_lift_along_forgetful_functor_of_s-replacement_category_induces_equivalence_on_localisations}\ref{cor:canonical_lift_along_forgetful_functor_of_s-replacement_category_induces_equivalence_on_localisations:coretraction}, \ref{cor:canonical_lift_along_forgetful_functor_of_s-replacement_category_induces_equivalence_on_localisations:retraction_up_to_isomorphism_and_forgetful_functor}, we have \(\InducedFunctorOfTotalSReplacementFunctor F \comp \GabrielZismanLocalisation(\bar F) = \id_{\GabrielZismanLocalisation(\mathcal{C})}\) and an isotransformation~\(\bar \beta\colon \GabrielZismanLocalisation(F) \comp \InducedFunctorOfTotalSReplacementFunctor F \map \GabrielZismanLocalisation(\ForgetfulFunctor)\) given by \(\bar \beta_{(Y, X', q')} = \LocalisationFunctor[\GabrielZismanLocalisation(\mathcal{D})](q')\colon F X' \map Y\) for~\((Y, X', q') \in \Ob{\GabrielZismanLocalisation(\CategoryOfObjectsWithSReplacement{\mathcal{D}}{F})} = \Ob \CategoryOfObjectsWithSReplacement{\mathcal{D}}{F}\).
\[\begin{tikzpicture}[baseline=(m-4-1.base)]
  \matrix (m) [diagram=4.0em]{
    \mathcal{C} & & \mathcal{D} \\ [-1.5em]
    \mathcal{C} & \CategoryOfObjectsWithSReplacement{\mathcal{D}}{F} & \mathcal{D} \\
    \GabrielZismanLocalisation(\mathcal{C}) & \GabrielZismanLocalisation(\CategoryOfObjectsWithSReplacement{\mathcal{D}}{F}) & \GabrielZismanLocalisation(\mathcal{D}) \\ [-1.5em]
    \GabrielZismanLocalisation(\mathcal{C}) & & \GabrielZismanLocalisation(\mathcal{D}) \\ };
  \path[->, font=\scriptsize]
    (m-1-1) edge node[above] {\(F\)} (m-1-3)
            edge[equality] (m-2-1)
    (m-1-3) edge[equality] (m-2-3)
    (m-2-1) edge node[above] {\(\bar F\)} (m-2-2)
            edge node[right] {\(\LocalisationFunctor\)} (m-3-1)
    (m-2-2) edge node[above] {\(\ForgetfulFunctor\)} node[below=1pt] {\(\categoricallyequivalent\)} (m-2-3)
            edge node[right] {\(\LocalisationFunctor\)} (m-3-2)
            edge node[above left=-3pt] {\(\TotalSReplacementFunctor F\)} (m-3-1)
    (m-2-3) edge node[right] {\(\LocalisationFunctor\)} (m-3-3)
    (m-2-3.south west) edge[bend left=15] node[below] {\(\StructureChoiceFunctor{R}\)} (m-2-2.south east)
    (m-3-1) edge node[above] {\(\GabrielZismanLocalisation(\bar F)\)} node[below] {\(\categoricallyequivalent\)} (m-3-2)
    (m-3-2) edge node[above] {\(\GabrielZismanLocalisation(\ForgetfulFunctor)\)} node[below] {\(\categoricallyequivalent\)} (m-3-3)
    (m-3-2.south west) edge[bend left=15] node[below] {\(\InducedFunctorOfTotalSReplacementFunctor F\)} (m-3-1.south east)
    (m-3-3.south west) edge[bend left=15] node[below] {\(\GabrielZismanLocalisation(\StructureChoiceFunctor{R})\)} (m-3-2.south east)
    (m-4-1) edge node[above] {\(\GabrielZismanLocalisation(F)\)} node[below=5pt] {\(\categoricallyequivalent\)} (m-4-3)
            edge[equality] (m-3-1)
    (m-4-3) edge[equality] (m-3-3)
    (m-4-3.south west) edge[bend left=15] node[below] {\(\InducedFunctorOfSReplacementFunctor[R] F\)} (m-4-1.south east);
\end{tikzpicture}\]
By remark~\ref{rem:induced_functor_of_s-replacement_functor_and_induced_functor_of_total_s-replacement_functor}, we have \(\InducedFunctorOfSReplacementFunctor F = \InducedFunctorOfTotalSReplacementFunctor F \comp \GabrielZismanLocalisation(\StructureChoiceFunctor{R})\). We obtain
\begin{align*}
\InducedFunctorOfSReplacementFunctor F \comp \GabrielZismanLocalisation(F) & = \InducedFunctorOfTotalSReplacementFunctor F \comp \GabrielZismanLocalisation(\StructureChoiceFunctor{R}) \comp \GabrielZismanLocalisation(\ForgetfulFunctor) \comp \GabrielZismanLocalisation(\bar F) \isomorphic \InducedFunctorOfTotalSReplacementFunctor F \comp \id_{\GabrielZismanLocalisation(\CategoryOfObjectsWithSReplacement{\mathcal{D}}{F})} \comp \GabrielZismanLocalisation(\bar F) = \InducedFunctorOfTotalSReplacementFunctor F \comp \GabrielZismanLocalisation(\bar F) \\
& = \id_{\GabrielZismanLocalisation(\mathcal{C})}, \\
\GabrielZismanLocalisation(F) \comp \InducedFunctorOfSReplacementFunctor F & = \GabrielZismanLocalisation(F) \comp \InducedFunctorOfTotalSReplacementFunctor F \comp \GabrielZismanLocalisation(\StructureChoiceFunctor{R}) \isomorphic \GabrielZismanLocalisation(\ForgetfulFunctor) \comp \GabrielZismanLocalisation(\StructureChoiceFunctor{R}) = \id_{\GabrielZismanLocalisation(\mathcal{D})},
\end{align*}
where isotransformations \(\alpha\colon \InducedFunctorOfSReplacementFunctor F \comp \GabrielZismanLocalisation(F) \map \id_{\GabrielZismanLocalisation(\mathcal{C})}\) and \(\beta\colon \GabrielZismanLocalisation(F) \comp \InducedFunctorOfSReplacementFunctor F \map \id_{\GabrielZismanLocalisation(\mathcal{D})}\) are given by \linebreak 
\(\alpha = \InducedFunctorOfTotalSReplacementFunctor F * \bar \alpha * \GabrielZismanLocalisation(\bar F)\) and~\(\beta = \bar \beta * \GabrielZismanLocalisation(\StructureChoiceFunctor{R})\). Thus \(\GabrielZismanLocalisation(F)\colon \GabrielZismanLocalisation(\mathcal{C}) \map \GabrielZismanLocalisation(\mathcal{D})\) and \(\InducedFunctorOfSReplacementFunctor F\colon \GabrielZismanLocalisation(\mathcal{D}) \map \GabrielZismanLocalisation(\mathcal{C})\) are mutually isomorphism inverse equivalences of categories.

For \(X' \in \Ob \mathcal{C}\), we have
\[\bar \alpha_{\GabrielZismanLocalisation(\bar F) X'} = \bar \alpha_{(F X', X', 1_{F X'})} = 1_{F X'}\colon (F X', X_{F X'}, q_{F X'}) \map (F X', X', 1_{F X'})\]
in~\(\GabrielZismanLocalisation(\CategoryOfObjectsWithSReplacement{\mathcal{D}}{F})\) and thus
\[\alpha_{X'} = (\InducedFunctorOfTotalSReplacementFunctor F) \bar \alpha_{\GabrielZismanLocalisation(\bar F) X'} = (\InducedFunctorOfTotalSReplacementFunctor F)_{(X_{F X'}, q_{F X'}), (X', 1_{F X'})} 1_{F X'} = (\TotalSReplacementFunctor F)_{(X_{F X'}, q_{F X'}), (X', 1_{F X'})} 1_{F X'}\colon X_{F X'} \map X'\]
in \(\GabrielZismanLocalisation(\mathcal{C})\). Moreover, for \(Y \in \Ob \mathcal{D}\), we have
\[\beta_{Y} = \bar \beta_{\GabrielZismanLocalisation(\StructureChoiceFunctor{R}) Y} = \bar \beta_{(Y, X_Y, q_Y)} = \LocalisationFunctor[\GabrielZismanLocalisation(\mathcal{D})](q_Y)\colon F X_Y \map Y\]
in \(\GabrielZismanLocalisation(\mathcal{D})\).
\end{proof}

\begin{corollary} \label{cor:characterisation_of_s-equivalences_as_s-dense_s-full_and_s-faithful_morphisms_of_categories_with_denominators}
We suppose that \(\mathcal{D}\) is multiplicative. The morphism of categories with denomina- \linebreak 
tors~\(F\colon \mathcal{C} \map \mathcal{D}\) is an S-equivalence if and only if it is S-dense, S-full and S-faithful.
\end{corollary}
\begin{proof}
First, we suppose that \(F\) is an S-equivalence, that is, we suppose that \(F\) is S-dense and that the induced functor \(\GabrielZismanLocalisation(F)\colon \GabrielZismanLocalisation(\mathcal{C}) \map \GabrielZismanLocalisation(\mathcal{D})\) is an equivalence. Then \(\GabrielZismanLocalisation(F)\) is full, so in particular \(F\) is S-full. Moreover, \(\GabrielZismanLocalisation(F)\) is faithful, so in particular \(F\) is S-faithful.

Conversely, we suppose that \(F\) is S-dense, S-full and S-faithful. As \(F\) is S-dense, there exists a choice of S{\nbd}replacements \(R = ((X_Y, q_Y))_{Y \in \Ob \mathcal{D}}\) for \(\mathcal{D}\) along~\(F\). But then \(\GabrielZismanLocalisation(F)\) is an equivalence of categories by the S-approximation theorem~\ref{th:s-approximation_theorem}, that is, \(F\) is an S-equivalence.
\end{proof}

It would have been nice to replace the multiplicativity of \(\mathcal{D}\) in theorem~\ref{th:s-approximation_theorem} and corollary~\ref{cor:characterisation_of_s-equivalences_as_s-dense_s-full_and_s-faithful_morphisms_of_categories_with_denominators} by requesting~\(\mathcal{D}\) to have all trivial S-replacements, which is weaker. However, the closedness under composition seems to be needed in the proof of corollary~\ref{cor:total_s-replacement_functor_maps_denominators_to_isomorphisms}. Cf.\ question~\ref{qu:necessity_of_multiplicativity_of_s-approximation_functors}.

In the whole proof of corollary~\ref{cor:characterisation_of_s-equivalences_as_s-dense_s-full_and_s-faithful_morphisms_of_categories_with_denominators}, including the preparing facts, we did not apply the dense-full-faithful criterion to the induced functor \(\GabrielZismanLocalisation(F)\). In fact, corollary~\ref{cor:characterisation_of_s-equivalences_as_s-dense_s-full_and_s-faithful_morphisms_of_categories_with_denominators} may be seen as a generalisation of this well-known result, which one reobtains if the denominators in \(\mathcal{C}\) and \(\mathcal{D}\) are supposed to be precisely the isomorphisms, respectively. Cf.~\cite[app.~A, sec.~1]{thomas:2012:a_calculus_of_fractions_for_the_homotopy_category_of_a_brown_cofibration_category}. 

We record a symmetric relationship between the isotransformations from the S-approximation theorem~\ref{th:s-approximation_theorem}:

\begin{remark} \label{rem:symmetric_relationship_between_isotransformations_from_s-approximation_theorem}
We suppose that \(\mathcal{D}\) is multiplicative and that \(F\) is S-full and S-faithful. Moreover, we suppose given a choice of S{\nbd}replacements \(R = ((X_Y, q_Y))_{Y \in \Ob \mathcal{D}}\) for \(\mathcal{D}\) along~\(F\). We let \(\alpha\colon \InducedFunctorOfSReplacementFunctor[R] F \comp \GabrielZismanLocalisation(F) \map \id_{\GabrielZismanLocalisation(\mathcal{C})}\) be the isotransformation given by
\[\alpha_{X'} = (\TotalSReplacementFunctor F)_{(X_{F X'}, q_{F X'}), (X', 1_{F X'})} 1_{F X'}\colon X_{F X'} \map X'\]
for \(X' \in \Ob \mathcal{C}\), and we let \(\beta\colon \GabrielZismanLocalisation(F) \comp \InducedFunctorOfSReplacementFunctor[R] F \map \id_{\GabrielZismanLocalisation(\mathcal{D})}\) be the isotransformation given by
\[\beta_Y = \LocalisationFunctor[\GabrielZismanLocalisation(\mathcal{D})](q_{Y})\colon F X_Y \map Y\]
for \(Y \in \Ob \mathcal{D}\). Then we have
\begin{align*}
\GabrielZismanLocalisation(F) * \alpha = \beta * \GabrielZismanLocalisation(F), \\
\InducedFunctorOfSReplacementFunctor[R] F * \beta = \alpha * \InducedFunctorOfSReplacementFunctor[R] F.
\end{align*}
\end{remark}
\begin{proof}
For \(X' \in \Ob \mathcal{C}\) we have \(\alpha_{X'} = (\TotalSReplacementFunctor F)_{(X_{F X'}, q_{F X'}), (X', 1_{F X'})} 1_{F X'}\) in \(\GabrielZismanLocalisation(\mathcal{C})\) and therefore
\begin{align*}
\GabrielZismanLocalisation(F) \alpha_{X'} & = (\GabrielZismanLocalisation(F) (\TotalSReplacementFunctor F)_{(X_{F X'}, q_{F X'}), (X', 1_{F X'})} 1_{F X'}) \, \LocalisationFunctor(1_{F X'}) = \LocalisationFunctor(q_{F X'}) \, \LocalisationFunctor(1_{F X'}) = \beta_{F X'} \\
& = \beta_{\GabrielZismanLocalisation(F) X'}
\end{align*}
in \(\GabrielZismanLocalisation(\mathcal{D})\).
\[\begin{tikzpicture}[baseline=(m-2-1.base)]
  \matrix (m) [diagram]{
    F X_{F X'} &[12.5em] F X' \\
    F X' & F X' \\ };
  \path[->, font=\scriptsize]
    (m-1-1) edge node[above] {\(\GabrielZismanLocalisation(F) (\TotalSReplacementFunctor F)_{(X_{F X'}, q_{F X'}), (X', 1_{F X'})} 1_{F X'}\)} (m-1-2)
            edge node[left] {\(\LocalisationFunctor(q_{F X'})\)} node[sloped, above] {\(\isomorphic\)} (m-2-1)
    (m-1-2) edge node[right] {\(\LocalisationFunctor(1_{F X'})\)} node[sloped, below] {\(\isomorphic\)} (m-2-2)
    (m-2-1) edge node[above] {\(\LocalisationFunctor(1_{F X'})\)} (m-2-2);
\end{tikzpicture}\]
Thus we have \(\GabrielZismanLocalisation(F) * \alpha = \beta * \GabrielZismanLocalisation(F)\). Since \(\GabrielZismanLocalisation(F)\) is an equivalence of categories by the S-approximation theorem~\ref{th:s-approximation_theorem}, it is in particular faithful, so that we also obtain \(\InducedFunctorOfSReplacementFunctor F * \beta = \alpha * \InducedFunctorOfSReplacementFunctor F\).~(\footnote{We suppose given a faithful functor \(H\colon \mathcal{A} \map \mathcal{B}\), a functor \(K\colon \mathcal{B} \map \mathcal{A}\), a transformation~\(\gamma\colon K \comp H \map \id_{\mathcal{A}}\) and an isotransformation~\(\delta\colon H \comp K \map \id_{\mathcal{B}}\) with \(H * \gamma = \delta * H\). Then as \(\delta\) is a transformation, we have \((H * K * \delta) \delta = (\delta * H * K) \delta\), and as~\(\delta\) is an isotransformation, we even have \(H * K * \delta = \delta * H * K = H * \gamma * K\). The faithfulness of \(H\) yields~\(K * \delta = \gamma * K\).})
\end{proof}


\bigskip

{\raggedleft Sebastian Thomas \\ sebastian.thomas@math.rwth-aachen.de \\ \url{http://www.math.rwth-aachen.de/~Sebastian.Thomas/} \\}


\begin{thebibliography}{11}
\bibitem{artin_grothendieck_verdier:1972:sga_4_1}
  \eigenname{Artin, Michael}; \eigenname{Grothendieck, Alexander}; \eigenname{Verdier, Jean-Louis}. \newblock
  \booktitle{Th{\'e}orie des topos et cohomologie {\'e}tale des sch{\'e}mas. Tome~1: Th{\'e}orie des topos.} Lecture Notes in Mathematics, vol.~269. \newblock
  Springer-Verlag, Berlin-New York, 1972. \newblock
  S{\'e}minaire de G{\'e}om{\'e}trie Alg{\'e}brique du Bois-Marie 1963--1964 (SGA4). \newblock
  With the collaboration of \eigenname{N.~Bourbaki}, \eigenname{P.~Deligne} and \eigenname{B.~Saint-Donat}.
\bibitem{brown:1974:abstract_homotopy_theory_and_generalized_sheaf_cohomology}
  \eigenname{Brown, Kenneth~S.} \newblock
  \booktitle{Abstract homotopy theory and generalized sheaf cohomology}. \newblock
  Transactions of the American Mathematical Society \textbf{186} (1974), pp.~419--458.
  DOI: \href{http://dx.doi.org/10.2307/1996573}{10.2307/1996573}.
\bibitem{cisinski:2010:categories_derivables}
  \eigenname{Cisinski, Denis-Charles}. \newblock
  \booktitle{Cat{\'e}gories d{\'e}rivables}. \newblock
  Bulletin de la Soci{\'e}t{\'e} Math{\'e}matique de France \textbf{138}(3) (2010), pp.~317--393.
\bibitem{gabriel_zisman:1967:calculus_of_fractions_and_homotopy_theory}
  \eigenname{Gabriel, Peter}; \eigenname{Zisman, Michel}. \newblock
  \booktitle{Calculus of Fractions and Homotopy Theory}. Ergebnisse der Mathematik und ihrer Grenzgebiete, Band~35. \newblock
  Springer-Verlag, New York, 1967.
\bibitem{gelfand_manin:2003:methods_of_homological_algebra}
   \eigenname{Gelfand, Sergei~I.}; \eigenname{Manin, Yuri~I.} \newblock
   \booktitle{Methods of homological algebra}. Springer Monographs in Mathematics, second edition. \newblock
   Springer-Verlag, Berlin, 2003.
\bibitem{kahn_maltsiniotis:2008:structures_de_derivabilite}
  \eigenname{Kahn, Bruno}; \eigenname{Maltsiniotis, Georges}. \newblock
  \booktitle{Structures de d{\'e}rivabilit{\'e}}. \newblock
  Advances in Mathematics~\textbf{218}(4)~(2008), pp.~1286--1318. \newblock
  DOI: \href{http://dx.doi.org/10.1016/j.aim.2008.03.010}{10.1016/j.aim.2008.03.010}.
\bibitem{kahn_sujatha:2007:a_few_localisation_theorems}
  \eigenname{Kahn, Bruno}; \eigenname{Sujatha, Ramdorai}. \newblock
  \booktitle{A few localisation theorems}. \newblock
  Homology, Homotopy and Applications~\textbf{9}(2) (2007), pp.~137--161. \newblock
  DOI: \href{http://dx.doi.org/10.4310/HHA.2007.v9.n2.a5}{10.4310/HHA.2007.v9.n2.a5}.
\bibitem{quillen:1967:homotopical_algebra}
  \eigenname{Quillen, Daniel~G.} \newblock
  \booktitle{Homotopical Algebra}. Lecture Notes in Mathematics, vol.~43. \newblock
  Springer-Verlag, Berlin-New York, 1967.
\bibitem{radulescu-banu:2006:cofibrations_in_homotopy_theory}
  \eigenname{R{\u{a}}dulescu-Banu, Andrei}. \newblock
  \booktitle{Cofibrations in Homotopy Theory}. \newblock
  Preprint, 2006 (vers.~4, February~8, 2009).
  \href{http://arxiv.org/abs/math/0610009v4}{\texttt{arXiv:math/0610009v4 [math.AT]}}.
\bibitem{thomas:2011:on_the_3-arrow_calculus_for_homotopy_categories}
  \eigenname{Thomas, Sebastian}. \newblock
  \booktitle{On the 3-arrow calculus for homotopy categories}. \newblock
  Homology, Homotopy and Applications \textbf{13}(1) (2011), pp.~89--119. \newblock
  DOI: \href{http://dx.doi.org/10.4310/HHA.2011.v13.n1.a5}{10.4310/HHA.2011.v13.n1.a5}.
\bibitem{thomas:2012:a_calculus_of_fractions_for_the_homotopy_category_of_a_brown_cofibration_category}
  \eigenname{Thomas, Sebastian}. \newblock
  \booktitle{A calculus of fractions for the homotopy category of a Brown cofibration category}. \newblock
  Dissertation, RWTH~Aachen University, 2012. \newblock
  \url{http://publications.rwth-aachen.de/record/210492}
\end{thebibliography}
\end{document}